\documentclass[a4paper,fleqn]{cas-sc}

\newcommand{\bsx}{\boldsymbol{x}}
\newcommand{\bsy}{\boldsymbol{y}}
\newcommand{\bS}{\mathbb{S}}
\newcommand{\RR}{\mathbb{R}}
\newcommand{\N}{\mathbb{N}}
\newcommand{\bE}{\mathbb{E}}
\newcommand{\bC}{\mathbb{C}}

\newcommand{\calL}{\mathcal L}
\newcommand{\calA}{\mathcal A}
\newcommand{\calC}{\mathcal C}
\newcommand{\calI}{\mathcal I}
\newcommand{\calJ}{\mathcal J}
\renewcommand{\Re}{\operatorname{Re}}
\renewcommand{\Im}{\operatorname{Im}}
\newcommand{\mi}{\mathrm{i}}
\newcommand{\wt}{\widehat{t}}

\newtheorem{theo}{Theorem}[section]
\newtheorem{lem}{Lemma}[section]
\newtheorem{prop}{Proposition}[section]
\newtheorem{rem}{Remark}[section]

\newtheorem{ass}{Assumption}[section]
\newtheorem{defin}{Definition}[section]

\newcommand{\tlg}[1]{{\color{magenta} #1}}

\newcommand{\itemprime}[1]
{\bgroup
	\addtocounter{enumi}{-1}%
	\item#1
	\egroup}

\usepackage[numbers,sort&compress]{natbib}
\usepackage{subfig}
\usepackage{float}
\usepackage{enumerate}
\usepackage{caption}
\usepackage{graphicx}

\def\tsc#1{\csdef{#1}{\textsc{\lowercase{#1}}\xspace}}
\tsc{WGM}
\tsc{QE}

\newproof{proof}{Proof}

\begin{document}
	\let\WriteBookmarks\relax
	\def\floatpagepagefraction{1}
	\def\textpagefraction{.001}
	
	\shorttitle{Fractional stochastic diffusion equations on the sphere}    
	
	\shortauthors{Alodat. T. et.al}  
	
	\title [mode = title]{On approximation for time-fractional stochastic diffusion equations on the unit sphere}   
	
	
	%

	\author{Tareq Alodat}
	
	
	
	\ead{T.Alodat@latrobe.edu.au}
	
	
	
	\affiliation{organization={La Trobe University},
		addressline={Melbourne}, 
		citysep={Melbourne}, 
		postcode={VIC 3086}, 
		country={Australia}}
	
	\author{Quoc T. Le Gia}
	\fnmark[*]
	
	\ead{qlegia@unsw.edu.au}
	
	\ead[url]{https://research.unsw.edu.au/people/dr-quoc-thong-le-gia}
	\affiliation{organization={The University of New South Wales},
		addressline={Sydney}, 
		citysep={Kensington}, 
		postcode={NSW 2052}, 
		country={Australia}}
	
	\author{Ian H. Sloan}
	\ead{I.sloan@unsw.edu.au}
	
	\cortext[1]{Quoc T. Le Gia}
	
	
	\nonumnote{}
	
\begin{abstract}
This paper develops a two-stage stochastic model to investigate the evolution of random fields on the unit sphere $\bS^2$ in $\RR^3$. The model is defined by a time-fractional stochastic diffusion equation on $\bS^2$ governed by a diffusion operator with a time-fractional derivative defined in the Riemann-Liouville sense. In the first stage, the model is characterized by a homogeneous problem with an isotropic Gaussian random field on $\bS^2$ as an initial condition. In the second stage, the model becomes an inhomogeneous problem driven by a time-delayed Brownian motion on $\bS^2$.
The solution to the model is given in the form of 
an expansion in terms of complex spherical harmonics. An approximation to the solution is given by truncating the expansion of the solution at degree $L\geq1$. The rate of convergence of the truncation errors as a function of $L$ and the mean square errors as a function of time are also derived. It is shown that the convergence rates depend not only on the decay of the angular power spectrum of the driving noise and the initial condition, but also on the order of the fractional derivative. We study sample properties of the stochastic solution and show that the solution is an isotropic H\"{o}lder continuous random field. Numerical examples and simulations inspired by the cosmic microwave background (CMB) are given to illustrate the theoretical findings.
\end{abstract}
	
	

	
	\maketitle
	
\section{Introduction}\label{1}
	
Let $\bS^2 \subset \RR^3$ be the 2-dimensional unit sphere in the Euclidean space $\RR^3$.
In this paper, we develop a two-stage stochastic model defined by time-fractional stochastic diffusion equations on $\bS^2$. The stochastic model can be used to model the evolution of spatio-temporal stochastic systems such as climate changes and the density fluctuations in the primeval universe (see \cite{Stefano2013,Philip2019,PlanckA1,Anh_et_al}). For general stochastic diffusion models (or It\^o diffusion) we refer the reader to \cite{Oks2003}). 
	
The model is constructed as follows. Let $\xi$ be a strongly isotropic Gaussian random field on $\bS^2$ and let $U(t)$, $ t\in(0,\infty)$, be a random field on $\bS^2$, which solves the following time-fractional diffusion equation:
\begin{align}\label{System}
		dU(t)   -  D_{t}^{1-\alpha} \Delta_{\bS^2} U(t) \; dt =
		\begin{cases}
			0,&t\in (0,\tau],\\
			dW_{\tau}(t),&t\in  [\tau,\infty),
\end{cases}
\end{align}
with the initial condition $U(0)= \xi$, where $\Delta_{\bS^2}$ is the Laplace-Beltrami operator, $\alpha$ is a constant satisfying $0 < \alpha \leq1$, $D_{t}^{\beta}$ is the Riemann-Liouville time fractional derivative of order $0\leq\beta<1$ defined, see \cite{Podlubny}, as 
 \begin{align}\label{D:Frac}
     D_{t}^{\beta} g(t):= \frac{d}{dt} 
	\int_0^t \frac{1}{\Gamma(1-\beta)}\frac{1}{ (t-s)^{\beta}} g(s) ds,\quad \ t>0,
 \end{align}
and $W_{\tau}$,\ $\tau>0$, is a time-delayed Brownian motion defined in $L_2(\Omega\times\bS^2)$ (the product space of the probability space $\Omega$ and the sphere $\bS^2$). The structure of the noise $W_\tau$ will be given in Section~\ref{sub2.3}.


In words, the model in \eqref{System} describes an initial isotropic random field evolving by a time-fractional diffusion process, with no external input until time $\tau$, at which time external noise is switched on.
	
Note that in the model \eqref{System}, it is assumed that the random field $\xi$ is independent of the noise $W_\tau$.
Since equation \eqref{System} is linear, it is easy to see that $U$ can be decomposed as
\begin{align}\label{Fullsol}
		U(t):= U^{H}(t)+U^{I}(t),\quad t\in (0,\infty), 
\end{align}
where $U^{H}$ satisfies the homogeneous equation 
\begin{align}\label{to}
		dU^{H}(t)   -  D_t^{1-\alpha} \Delta_{\bS^2} U^{H}(t) \; dt= 0,\quad\ \ t\in  (0,\infty),
\end{align}
with the initial condition $ U^{H}(0) = \xi$, and 
$U^{I}$ satisfies the inhomogeneous equation 
\begin{align}\label{pde}
		dU^{I}(t)   -  D_{t}^{1-\alpha} \Delta_{\bS^2} U^{I}(t) \; dt =
		\begin{cases}
			0,&t\in (0,\tau],\\
			dW_{\tau}(t),&t\in [\tau,\infty),
		\end{cases}
\end{align}
with the condition $U^{I}(t)=0$, $t\in (0,\tau]$, from which it follows that $U^{I}(t)$ is zero until time $\tau$.
	
We write the equation \eqref{pde} in the integral form, for $t>\tau$, as
\begin{align}\label{UIintform}
   U^{I}(t)-  \frac{1}{\Gamma(\alpha)}\int_\tau^t\frac{\Delta_{\bS^2} U^{I}(s)}{(t-s)^{1-\alpha}}ds= W_{\tau}(t). 
\end{align}
Note that in \eqref{System} there are two sources of randomness, namely the initial condition $\xi$ and the driving noise $W_\tau$. The split of the stochastic solution $U$ in equation \eqref{Fullsol} into two independent parts (due to linearity) separates the two sources of randomness: the homogeneous solution $U^H$ is random only through the initial condition, while the inhomogeneous solution $U^I$ is random only through the time-delayed Brownian motion $W_\tau$.

We shall derive the stochastic solution $U$ to the equation \eqref{System}, which represents a random field defined in $L_2(\Omega\times\bS^2)$. The solution $U$ is given in the form of an expansion in terms of complex orthonormal spherical harmonics $\{Y_{\ell,m}: \ell\in\N_{0}, m=-\ell,\dots,\ell\}$, i.e., 
\begin{align}\label{A1}
		U(\omega,t):= \sum_{\ell=0}^\infty \sum_{m=-\ell}^{\ell} \widehat{U}_{\ell,m}(\omega,t) Y_{\ell,m}, 
         \ t\in(0,\infty),
\end{align}
where the random variables $\widehat{U}_{\ell,m}$ represent the Fourier coefficients of $U$ and are given by 
\begin{align*}
		\widehat{U}_{\ell,m}(\omega,t):=\int_{\bS^2}U(\omega,\bsx,t)\overline{Y_{\ell,m}(\bsx)}\mu(d\bsx),\quad t\in(0,\infty),
\end{align*}
where $\overline{Y_{\ell,m}}$ denotes the complex conjugate of $Y_{\ell,m}$ and $\mu$ is the normalized Riemann surface measure on the sphere $\bS^2$ (details are given in Section \ref{Sec2}).
	
Given that the initial field $\xi$ is a strongly isotropic Gaussian random field on $\bS^2$, we shall show that the stochastic solution $U$ is also a strongly isotropic Gaussian random field on $\bS^2$.
	
We truncate the expansion \eqref{A1} at a level $\ell=L\geq1$ to obtain an approximate version $U_L$ of the solution $U$, i.e.,
\begin{align*}
		U_L(\omega,t) := \sum_{\ell=0}^L\sum_{m=-\ell}^{\ell} \widehat{U}_{\ell,m}(\omega,t) Y_{\ell,m}, \quad t\in(0,\infty).
\end{align*}
We investigate the approximate solution $U_L$ and its rate of convergence as $L\to\infty$. We shall show that the convergence rate (in $L_2$-norm on $L_2(\Omega\times\bS^2)$) depends not only on the decay of the power spectrum of the driving noise $W_{\tau}$ and the initial condition $\xi$, but also on the order of the fractional derivative $\alpha$. 
Also, we show that the stochastic solution $U(\omega,t)$, $t>\tau$, evolves continuously with time by deriving an upper bound in the $L_2$-norm for its temporal increments from time $t$ to $t+h$ of the form $q(t)h^{\frac{1}{2}},\ h>0$, $t\in(\tau,\infty)$, where the continuous function $q(\cdot)$ does not depend on $h$. Finally, under some conditions, we show the existence of a locally H\"{o}lder continuous modification of $U$.
	
Note that the proposed model, defined in \eqref{System}, describes the evolution (in time) of a two-stage stochastic system. The homogeneous equation given in \eqref{to} determines the evolution for the system with no external input while the inhomogeneous equation \eqref{pde} provides the perturbation produced by the Brownian motion starting at time $\tau$. The cosmic microwave background (CMB) provides motivation for considering multi-stage stochastic systems. 
Roughly speaking, CMB is electromagnetic radiation that has been emitted across the cosmos since ionised atoms and electrons recombined roughly 370,000 years after the Big Bang \cite{PlanckA9,PlanckA16,Dodelson}. 
Since CMB has passed through several stages called formation epochs or phases (such as Planck, Grand Unification, Inflation and recombination), it can be considered as an example of a multi-stage stochastic system.

The literature shows a variety of studies where models using stochastic partial differential equations were developed, see for example \cite{Philip2019,CohLan21,KazLeG19,LanSch15,LeoVaz19,McLTho10}. 
However, little attention has been paid to two-stage stochastic models of the kind considered here. A fractional SPDE,  governed by a fractional derivative in time and  a fractional diffusion operator in space, was developed in \cite{Vho2021} to model evolution of tangent vector random fields on the unit sphere. In a recent paper, Anh et.al. \cite{Anh_et_al} considered a two-stage stochastic model defined by SPDEs on the sphere governed by a fractional diffusion operator in space and a fractional Brownian motion as driving noise. The case where a random initial condition given by a fractional stochastic Cauchy problem was considered. They showed that the truncation errors of the stochastic solution have, in the $L_2-$norm, the convergence rates $C_tL^{-r}$, $r>1$, under the assumption that the variances of the driving noise satisfy some smoothness condition. Also, as a function of time $t$, it was demonstrated that the ``constant'' $C_t$ blows up when $t$ is sufficiently small.
In a recent article \cite{Oksendal2003frac}, a time-fractional stochastic heat equation driven by time-space white noise in $\mathbb{R}^d$ was considered. The model is defined in the distribution sense, and an explicit solution is derived within the space of tempered distributions. The stochastic solution is demonstrated to exhibit mild behavior exclusively when the dimension is either $d=1$ or $d=2$, and the order of the fractional derivative falls within the range $\alpha\in(1,2)$, while for $\alpha<1$, the solution is not mild for any $d$.

The paper is structured as follows.
Section \ref{Sec2} presents necessary material from the theory of functions on the sphere $\bS^2$, spherical harmonics, Gaussian random fields, $L_2(\bS^2)$-valued time-delayed Brownian motions, and some necessary tools from the theory of stochastic integrals. Section \ref{Sec4} derives properties of the solution to the homogeneous equation \eqref{to}. In Section \ref{Sec5} we study the solution to the inhomogeneous equation \eqref{pde}. In Section \ref{FullSec} we give results for the combined equation \eqref{System}. Section \ref{Sec6} derives an approximation to the solution of equation \eqref{System}, the rate of convergence for the truncation errors. In Section \ref{Sec7} we study the temporal increments, in the $L_2$-norm, of the stochastic solution of \eqref{System} and its sample H\"{o}lder continuity property. 
In Section \ref{Num} we provide some numerical examples to explain the theoretical findings. In particular, Section \ref{Evol} illustrates the evolution of the stochastic solution of \eqref{System} using simulated data inspired by the CMB map. Section \ref{Sim-stage} explains the convergence rates of the truncation errors, in the $L_2$-norm, of the stochastic solution of \eqref{System}. Finally,  Section \ref{Sim-inc} explores the convergence rates of the temporal increments, in the $L_2$-norm, of the solution of \eqref{System}.

\section{Preliminaries}\label{Sec2}
Let $\RR^3$ be the 3-dimensional Euclidean space. Let $\bsx,\bsy\in\RR^3$, i.e. $\bsx=(x_1,x_2,x_3)$ and $\bsy=(y_1,y_2,y_3)$, then the inner product of $\bsx$ and $\bsy$  is 
$\bsx\cdot\bsy:=\sum_{i=1}^{3}x_{i}y_{i}$ and the Euclidean norm of $\bsx$ is $\|\bsx\|:=\sqrt{\bsx\cdot\bsx}$. We denote by $\bS^2 \subset \RR^3$ the 2-dimensional unit sphere in $\RR^3$, i.e.,
	$\bS^2 :=\left\lbrace\bsx \in \RR^3: \|\bsx \|=1 \right\rbrace$. The pair $(\bS^2,\tilde{d})$ forms a compact metric space, where $\tilde{d}$ is the geodesic metric defined by $\tilde{d}(\bsx,\bsy):=\arccos(\bsx\cdot\bsy)$, $\bsx,\bsy\in\bS^2$. 
	
Let $(\Omega,\mathcal{F},\mathbb{P})$ be a probability space and $L_{2}(\Omega,\mathbb{P})$ be the $L_{2}$-space on $\Omega$ with respect to the probability measure $\mathbb{P}$, endowed with the norm $\|\cdot\|_{L_{2}(\Omega)}$. For two random variables $X$ and $Y$ on $(\Omega,\mathcal{F},\mathbb{P})$, we denote by $\bE [X]$ the expected value of $X$ and by $Cov(X,Y):=\bE[(X-\bE[X])(Y-\bE[Y])]$ the covariance between $X$ and $Y$. For two random variables $X$ and $Y$, we use $X\simeq Y$ to denote the equality in distribution. Note that when $X\simeq Y$, $Cov(X,X)$ is the variance of $X$, which is denoted by $Var[X]$. For a random variable $X$ on $(\Omega,\mathcal{F},\mathbb{P})$, we denote by $\varphi_{X}(\cdot)$ the characteristic function of $X$, i.e. $\varphi_{X}(r):=\bE\big[e^{\mi r X}\big]$.

	\subsection{Functions on the sphere}\label{RF}
	Let $\mu$ be the normalized Riemann surface measure on the sphere $\bS^2$ (i.e. $\int_{\bS^2}d\mu=1$). We denote by $L_{2}(\bS^2):=L_{2}(\bS^2,\mu)$ the space of complex-valued square $\mu$-integrable functions on $\bS^2$, i.e.,
	\[
	L_{2}(\bS^2)=\left\{f: \int_{\bS^2}|f(\bsx)|^2d\mu(\bsx)<\infty\right\}.
	\]
	Let $f, g$ be two arbitrary functions in $L_{2}(\bS^2)$, the inner product $\langle f,g\rangle_{L_{2}(\bS^2)}$ is defined as
	\[
	\langle f,g\rangle_{L_{2}(\bS^2)}:=\int_{\bS^2}f(\bsx)\overline{g(\bsx)}d\mu(\bsx)
	\]
	and the norm $\|f\|_{L_{2}(\bS^2)}$ of $f$ is 
	\[
	\|f\|_{L_{2}(\bS^2)}:=\left(\int_{\bS^2}|f(\bsx)|^2d\mu(\bsx)\right)^{\frac{1}{2}}.
	\]
	Let $\N_{0}=\{0,1,2,\dots\}$ and let $P_\ell$, $\ell\in\N_{0}$, be the Legendre polynomials of degree $\ell$, i.e.,
	\[
	P_\ell(t):= 2^{-\ell}\frac{1}{\ell!}\dfrac{d^\ell}{d t^{\ell}}(t^2-1)^\ell, \quad \ell\in\N_{0},\ t\in[-1,1].
	\]
	The Legendre polynomials define the associated Legendre functions $P_{\ell,m}$, $\ell\in\N_{0}$, $m=0,\dots,\ell$, by 
	\[
	P_{\ell,m}(t):=(-1)^m(1-t^2)^{\frac{m}{2}}\dfrac{d^m}{d t^{m}}P_\ell(t), \quad \ell\in\N_{0},\ t\in[-1,1].
	\]
Let $\{ Y_{\ell,m}: \ell\in\N_{0}, m=-\ell,\dots,\ell \}$ be an orthonormal basis of complex-valued spherical harmonics 
of $L_{2}(\bS^2)$. Using spherical coordinates $\bsx:=(\sin\theta\cos\varphi,\sin\theta\sin\varphi,\cos\theta)\in\bS^2$, $\theta\in[0,\pi]$, $\varphi\in[0,2\pi)$, the
	spherical harmonic functions can be written as, for $\ell\in\N_{0}$,
	\begin{align*}
		&Y_{\ell,m}(\theta,\varphi):=\sqrt{\dfrac{(2\ell+1)(\ell-m)!}{(\ell+m)!}} P_{\ell,m}(\cos\theta)e^{\mi m\varphi},\ \  m=0,\dots,\ell,\\
		&Y_{\ell,m}:=(-1)^m\overline{Y_{\ell,-m}},\ \  m=-\ell,\dots,-1.
	\end{align*}
	
In what follows, we will adopt the notation $Y_{\ell,m}$ to represent spherical harmonics as a function of both spherical and Euclidean coordinates. Specifically, for all $\bsx \in \mathbb{S}^2$, we shall write (with a slight abuse of notation)
$Y_{\ell,m}(\mathbf{\bsx}) := Y_{\ell,m} (\theta, \varphi)$, where $\bsx = (\sin \theta \cos \varphi, \sin \theta \sin \varphi, \cos \theta)$ and $(\theta, \varphi) \in [0, \pi] \times [0, 2\pi).$

The basis $Y_{\ell,m}$ and the Legendre polynomial $P_{\ell}$ satisfy the addition theorem (see \cite{Muller}), that is for $\bsx,\bsy\in\bS^2$, there holds
	\begin{align}\label{addition}
		\sum_{m=-\ell}^{\ell} Y_{\ell,m}(\bsx)\overline{Y_{\ell,m}(\bsy)}=(2\ell+1)P_{\ell}(\bsx\cdot\bsy).
	\end{align}
	The Laplace-Beltrami operator $\Delta_{\bS^2}$ (or the spherical Laplacian) on the sphere $\bS^2$ at $\bsx$ is given in terms of the spherical coordinates by (see \cite{Muller,Dai})
	\[
	\Delta_{\bS^2}:= \frac{1}{\sin\theta}\frac{\partial}{\partial\theta}\Big(\sin\theta\frac{\partial}{\partial\theta}\Big)+\Big(\frac{1}{\sin\theta}\Big)^2\frac{\partial^2}{\partial\varphi^2}.
	\]
	It is well-known that the spherical harmonic functions $\{Y_{\ell,m}: \ell\in\N_{0}, m=-\ell,\dots,\ell \}$ are the eigenfunctions of the negative Laplace-Beltrami operator $-\Delta_{\bS^2}$ on the sphere $\bS^2$ with eigenvalues 
	\begin{align}\label{lam}
		\lambda_{\ell}:=\ell(\ell+1),\quad \ell\in\N_{0},
	\end{align}
	that is for $\ell\in\N_{0}$, $m=-\ell,\dots,\ell$, 
	\begin{align}\label{beltrami}
		-\Delta_{\bS^2}Y_{\ell,m}=\lambda_{\ell}Y_{\ell,m}.
	\end{align}
	An arbitrary complex-valued function $f\in L_{2}(\bS^2)$ can be expanded in terms of a Fourier-Laplace series, in the $ L_{2}(\bS^2)$ sense,
	\begin{align}\label{Expansion}
		f&=\sum_{\ell=0}^{\infty}\sum_{m=-\ell}^{\ell}\widehat{f}_{\ell,m}Y_{\ell,m},\ \text{with}\  \widehat{f}_{\ell,m}:=\int_{\bS^2}f(\bsx)\overline{Y_{\ell,m}(\bsx)}d\mu(\bsx).
	\end{align}
	The $\widehat{f}_{\ell,m}$, for $\ell\in\N_{0},m=-\ell,\dots,\ell$, are known as the Fourier coefficients for the function $f$ under the Fourier basis $Y_{\ell,m}$.
	
	By Parseval’s theorem, for $f\in L_{2}(\bS^2)$, there holds 
	\[
	\|f\|_{L_{2}(\bS^2)}^2= \sum_{\ell=0}^{\infty}\sum_{m=-\ell}^{\ell} |\widehat{f}_{\ell,m} |^2.
	\]
	Note that by the properties of the spherical harmonic functions, the Fourier-Laplace series \eqref{Expansion} of any real-valued $f\in L_{2}(\bS^2)$ takes the form
	\begin{align*}
		f&=\sum_{\ell=0}^{\infty}\Big(\widehat{f}_{\ell,0}Y_{\ell,0}+2\sum_{m=1}^{\ell}\Big(\Re \widehat{f}_{\ell,m}\Re Y_{\ell,m}-\Im \widehat{f}_{\ell,m}\Im Y_{\ell,m}\Big)\Big).
	\end{align*}
	
	\subsection{Isotropic random fields on the sphere}
	This subsection introduces isotropic Gaussian random fields on the sphere $\bS^2$ and their expansions in terms of complex spherical harmonics \cite{LeGia2020,LanSch15}. 
	
	We denote by $\mathfrak{B}(\bS^2)$ the Borel $\sigma$-algebra on the sphere $\bS^2$ and $SO(3)$ the rotation group on $\RR^3$. 
	An $\mathcal{F}\otimes\mathfrak{B}(\bS^2)$-measurable function $\xi: \Omega\times\bS^2\to\RR$ is called a (jointly measurable) real-valued random field on $\bS^2$. In the paper, we assume that the random field $\xi\in L_{2}(\Omega\times\bS^2)$.
	Let  $L_{2}(\Omega\times\bS^2,\mathbb{P}\otimes\mu)=: L_{2}(\Omega\times\bS^2)$ be the real-valued $L_2$-space of random fields on $\Omega\times\bS^2$, with the product measure $\mathbb{P}\otimes\mu$. Using Fubini’s theorem, the inner product of $\xi_1,\xi_2\in L_{2}(\Omega\times\bS^2)$ can be written as
	\begin{align*}
		\langle \xi_1,\xi_2\rangle_{L_{2}(\Omega\times\bS^2)}:&=\int_{\Omega\times\bS^2}\xi_1(\omega,\bsx) \xi_2(\omega,\bsx)d(\mathbb{P}\otimes\mu)(\omega,\bsx)\\
		&= \bE\left[	\langle \xi_1,\xi_2\rangle_{L_{2}(\bS^2)}\right],
	\end{align*}
	and the norm of $\xi_1\in L_{2}(\Omega\times\bS^2)$ is
	\[
	\|\xi_1\|_{L_{2}(\Omega\times\bS^2)}:=\sqrt{\langle \xi_1,\xi_1\rangle_{L_{2}(\Omega\times\bS^2)}}.
	\]
	In particular, $\xi(\omega,\cdot)\in L_{2}(\bS^2)$ $\mathbb{P}$-a.s., since 
	$\|\xi\|_{L_{2}(\Omega\times\bS^2)}^2=\bE\big[	\|\xi\|_{L_{2}(\bS^2)}^2\big]$. Note that an arbitrary random field $\xi(\omega,\cdot)\in L_{2}(\bS^2)$ $\mathbb{P}$-a.s., admits an expansion in terms of spherical harmonics, $\mathbb{P}$-a.s.,
	\begin{align}\label{xi}
		\xi=\sum_{\ell=0}^\infty \sum_{m=-\ell}^{\ell}  \widehat{\xi}_{\ell,m} Y_{\ell,m},\quad \widehat{\xi}_{\ell,m}= \int_{\bS^2}\xi(\bsx)\overline{Y_{\ell,m}(\bsx)}d\mu(\bsx),
	\end{align}
where the convergence is in the $L_{2}(\Omega\times\bS^2)$ sense.

	
In this paper, we consider a particular class of random fields called isotropic random fields. Following \cite{MarPec11}, the random field $\xi$ is strongly isotropic if for any $k\in\N$ and for all sets of $k$ points $\bsx_1,\dots,\bsx_k\in\bS^2$, and for any rotation $R\in SO(3)$, the joint distributions of $\xi(\bsx_1),\dots,\xi(\bsx_k)$ and $\xi(R\bsx_1),\dots,\xi(R\bsx_k)$ coincide. 
	
The random field $\xi$ is called a $2$-weakly, isotropic random field if for all $\bsx\in\bS^2$, the second moment of $\xi(\bsx)$ is finite, that is $\bE[|\xi(\bsx)|^2]<\infty$, and for all $\bsx_1,\bsx_2\in\bS^2$, and for any rotation $R\in SO(3)$, (see \cite{MarPec11}), there holds
\[
	\bE [\xi(\bsx_1)]=\bE [\xi(R\bsx_1)],\ \bE [\xi(\bsx_1)\xi(\bsx_2)]=\bE [\xi(R\bsx_1)\xi(R\bsx_2)].
\]

If $\xi$ is a centered, 2-weakly isotropic random field, then the Fourier coefficients $\widehat{\xi}_{\ell,m}$ of $\xi$ in \eqref{xi} are uncorrelated mean-zero complex-valued random variables (see Remark 5.15 in \cite{MarPec11}), i.e., for $\ell\in\N_{0}$, $m=-\ell,\dots,\ell$,
	\[
	\bE\left[\widehat{\xi}_{\ell,m}\right]=0,\quad \bE\left[\widehat{\xi}_{\ell,m}\overline{
		\widehat{\xi}_{\ell^{\prime},m^{\prime}}}\right]=\calC_{\ell,m}\delta_{\ell\ell^\prime}\delta_{mm^\prime},
	\]
	where the $\calC_{\ell,m}$ are non-negative numbers and $\delta_{\ell\ell^\prime}$ is the Kronecker delta function. The sequence $\{\calC_{\ell,m}: \ell\in\N_{0},m=-\ell,\dots,\ell\}$ is called the angular power
	spectrum of the random field $\xi$. It follows that for each $\bsx\in\bS^2$, $\xi(\bsx)$ is a centered random variable and its covariance function takes the form
 \begin{align*}
     \bE\left[\xi(\bsx)\xi(\bsy)\right]=\sum_{\ell=0}^\infty \sum_{m=-\ell}^{\ell}\calC_{\ell,m}Y_{\ell,m}(\bsx)\overline{Y_{\ell,m}(\bsy)},\quad \bsx,\bsy\in\bS^2.
 \end{align*}
 
For a centered, 2-weakly isotropic random field $\xi$ it holds, for $\bsx,\bsy\in\bS^2$, that $\bE[\xi(\bsx)]=0$ and $\bE\left[\xi(\bsx)\xi(\bsy)\right]$ is rotationally invariant. In this case the angular power spectrum is independent of $m$, that is $\calC_{\ell,m}=\calC_{\ell}$, $\ell\in\N_{0}$, $m=-\ell,\dots,\ell$, and can be written as
\begin{equation}\label{eq:defCell}
\calC_{\ell}=\bE[|\widehat{\xi}_{\ell,m}|^2].    
\end{equation}
Thus, the covariance function $\bE\left[\xi(\bsx)\xi(\bsy)\right]$ takes the form, using the addition theorem for spherical harmonics (see equation \eqref{addition}),
	\begin{align}\label{EqCl}
		\bE\left[\xi(\bsx)\xi(\bsy)\right]=\sum_{\ell=0}^{\infty} (2\ell+1)\calC_{\ell}P_{\ell}(\bsx\cdot\bsy).
	\end{align}
Since $|P_{\ell}(\cdot)|\leq1$, for $\ell\to\infty$, the series \eqref{EqCl} is convergent provided that
\begin{equation}\label{eq:condCell}
    \sum_{\ell=0}^{\infty} (2\ell+1)\calC_{\ell}<\infty.
\end{equation}
Throughout all subsequent sections, we will assume that the condition given in equation \eqref{eq:condCell} remains valid.

	\noindent
A random field $\xi$ on $\bS^2$ is Gaussian if for each $k\in\N=\{1,2,3,\dots\}$ and each $k$ points $\bsx_1,\dots,\bsx_k\in\bS^2$, the vector $(\xi(\bsx_1),\dots,\xi(\bsx_k))$ has a multivariate Gaussian distribution. 
	
Note that a Gaussian random field $\xi$ is strongly isotropic if and only if it is $2$-weakly isotropic (see \cite{MarPec11}). 
In this paper, we consider random fields that are Gaussian and strongly isotropic.

We recall the following result (see \cite{LanSch15}, Corollary 2.5).
	\begin{prop}
		{\rm{\cite{LanSch15}}}\label{Coro}
		Let $\{\calC_{\ell}: \ell\in\N_{0}\}$ be the angular power spectrum of a centered, $2$-weakly isotropic Gaussian random random field $\xi$ on $\bS^2$. Then $\xi$ admits the expansion
		\[
		\xi=\sum_{\ell=0}^\infty \sum_{m=-\ell}^{\ell} \widehat{\xi}_{\ell,m} Y_{\ell,m},
		\]
		which is convergent in the $L_{2}(\Omega\times\bS^2)$ sense, where $Y_{\ell,m},\ \ell\in\N_{0}, m=-\ell,\dots,\ell$, are spherical harmonic functions and the elements of the set $\mathcal{V}:=\{\widehat{\xi}_{\ell,m}: \ell\in\N_{0}, m=-\ell,\dots,\ell\}$ are complex-valued, centered Gaussian random variables satisfying:
\begin{itemize}
\item The subset $\mathcal{V}^{\prime}:= \{\widehat{\xi}_{\ell,m}: \ell\in\N_{0}, m=0,\dots,\ell\}\subset\mathcal{V}$ consists of {\bf{independent}}, complex-valued
			Gaussian random variables.
			
\item The elements of $\mathcal{V}^{\prime}$ with $m>0$ have the property that $\Re \widehat{\xi}_{\ell,m}$ and $\Im \widehat{\xi}_{\ell,m}$ are independent and $\mathcal{N}(0,\calC_\ell/2)$ distributed.
			
\item The elements of $\mathcal{V}^{\prime}$ with $m=0$ are real-valued, and the elements $\Re \widehat{\xi}_{\ell,0}$ are $\mathcal{N}(0,\calC_\ell)$ distributed.
			
			\item The elements of $\mathcal{V}$ with $m<0$ can be obtained from those of  $\mathcal{V}^{\prime}$ by applying the relations
			\[
			\Re \widehat{\xi}_{\ell,m}= (-1)^m \Re \widehat{\xi}_{\ell,-m}\ \text{and}\ \Im \widehat{\xi}_{\ell,m}= (-1)^{m+1} \Im \widehat{\xi}_{\ell,-m}.
			\]
		\end{itemize}
	\end{prop}
	
	\begin{rem} \label{rem1-1}
		Let $\{\calC_{\ell}:\ell\in\N_{0}\}$ be the angular power spectrum of a centered, strongly isotropic Gaussian random field $\xi$ on $\bS^2$. Let for $\ell\in\N_0$,  
		\begin{align}\label{Z-iid}
			\mathcal{Z}_{\ell}:=\Big\{ Z_{\ell,0}^{(1)},Z_{\ell,1}^{(1)},\dots, Z_{\ell,\ell}^{(1)},
			Z_{\ell,1}^{(2)},Z_{\ell,2}^{(2)},\dots, Z_{\ell,\ell}^{(2)}\Big\},
		\end{align}
be a set of independent, real-valued, standard normally distributed random variables. Then, by Proposition {\rm\ref{Coro}}, the field $\xi$ can be represented, in the $L_{2}(\Omega\times\bS^2)$ sense, as
		\begin{align*}
			\xi:&\simeq\sum_{\ell=0}^{\infty}\Bigl( \sqrt{\calC_{\ell}}Z_{\ell,0}^{(1)}Y_{\ell,0}\notag\\
			&+\sqrt{2\calC_{\ell}}\sum_{m=1}^{\ell}\Bigl(Z_{\ell,m}^{(1)}\Re Y_{\ell,m}+Z_{\ell,m}^{(2)}\Im Y_{\ell,m}\Big)\Big),
		\end{align*}
		where $Z_{\ell,m_1}^{(1)}$, $Z_{\ell,m_2}^{(2)}\in\mathcal{Z}_{\ell}$, $\ell\in\N_0$, $m_1=0,\dots,\ell$, $m_2=1,\dots,\ell$.
	\end{rem}

	\subsection{Brownian motion}\label{sub2.3}
	
	The standard Brownian motion (or the Wiener process) $B(t),\ t\geq0$, with variance $1$ at $t=1$, is a centered Gaussian process on $\RR_{+}=[0,\infty)$ satisfying $B(0)=0$ and
	\[
	\bE\Big[\big\vert B(t)-B(s)\big\vert^2\Big]=\vert t-s\vert,\quad s,t\geq0.
	\]
	Let $B_{\tau}(t),\ t\geq\tau\geq0$, be a real-valued, time-delayed Brownian motion with $B_{\tau}(\tau)=0$ and variance $1$ at $t=\tau+1$. Note that it is easy to show that $B_{\tau}(t),\ t\geq\tau\geq0$, is a centered Gaussian process on $[\tau,\infty)$ satisfying
	\begin{equation}\label{eq:Btau}
		\bE\Big[\big\vert B_{\tau}(t)-B_{\tau}(s)\big\vert^2\Big]=\vert t-s\vert,\quad t,s\geq\tau.
	\end{equation}
	Putting $s=\tau$ in \eqref{eq:Btau} we have
	\begin{equation}\label{eq:Btau2}
		\bE\Big[\big( B_{\tau}(t)\big)^2\Big]=\vert t-\tau\vert,\quad t\geq\tau\geq0.
	\end{equation}
The case when $\tau=0$, $B_{0}(t),\ t\geq0$, represents the standard Brownian motion. We will for brevity write $B_{0}(t)$ as $B(t)$ if no confusion arises.


We define a time-delayed Brownian motion on the unit sphere $\bS^2$ as in \cite[Sections 4.1.1 and 4.1.2] {DaPrato}. 
Let $W_\tau(t)$, $\tau\geq0$, be an $L_2(\bS^2)$-valued time-delayed Brownian motion. 
Then, for each real-valued $f$ in $L_2(\mathbb{S}^2)$, the stochastic process $\langle W_\tau(t),f \rangle_{L_2(\bS^2)}$ is a real-valued, time-delayed Brownian motion. 

For arbitrary real-valued $f, g \in L_2(\bS^2)$, $t, s \ge \tau$,
\[
\bE \big[\langle W_\tau(t),f\rangle_{L_2(\bS^2)} \langle W_\tau(s),f\rangle_{L_2(\bS^2)}\big] = 
(t \wedge s) \bE \big[\langle W_\tau(\tau+1), f\rangle_{L_2(\bS^2)}^2\big]
\]
and
\[
\bE [\langle W_{\tau}(t),f\rangle_{L_2(\bS^2)} \langle W_{\tau}(s),g\rangle_{L_2(\bS^2)}] = (t \wedge s )
\bE [\langle Qf,g\rangle_{L_2(\bS^2)}],
\]
where $Q$ is the covariance operator for 
the law of $W_{\tau}(1+\tau)$, see \cite[Section 2.3.1] {DaPrato}. 

\begin{defin}\label{BRo}   
An $L_2(\bS^2)$-valued stochastic process $W_\tau(t)$, for $t \ge \tau$,
is called a Q-Wiener process if
\begin{enumerate}
\item $W_\tau(\tau) = 0$,
\item $W_\tau$ has continuous trajectories,
\item $W_\tau$ has independent increments,
\item $W_\tau(t)-W_\tau(s)$ follows the normal distribution $\mathcal{N}(0,(t-s)Q)$, $t\geq s\geq\tau$.
\end{enumerate}
\end{defin}


For the complete orthonormal basis 
$Y_{\ell,m}$, there is a sequence of non-negative numbers $\mathcal{A}_{\ell}$,
\[
Q Y_{\ell,m} = \mathcal{A}_{\ell} Y_{\ell,m},\ 
\ell=0,1,2,\ldots; m=-\ell,\ldots,\ell.
\]

\begin{prop}
If $Q$ is of trace class, i.e., ${\rm Tr}(Q)<\infty$, then the following condition
holds true
\begin{equation}\label{eq:condAell}
		\sum_{\ell=0}^{\infty} (2\ell+1)\calA_{\ell} < \infty.
\end{equation}    
\end{prop}
\begin{proof}
By definition,
\begin{equation}
{\rm Tr}(Q) = 
\sum_{\ell=0}^\infty \sum_{m=-\ell}^\ell
    \langle Q Y_{\ell,m}, Y_{\ell,m}\rangle_{L_2(\bS^2)} = 
    	\sum_{\ell=0}^{\infty} (2\ell+1)\calA_{\ell} < \infty.
     \end{equation}
\end{proof}
In what follows we will assume that the condition \eqref{eq:condAell} remains valid.

Let for $\ell\in\N_0$, $t\geq\tau\geq0$,  
	\begin{align}\label{B-iid}
		\mathcal{B}_{\ell,\tau}(t):=\Big\{ \beta_{\ell,0,\tau}^{(1)}(t),\beta_{\ell,1,\tau}^{(1)}(t),\dots, \beta_{\ell,\ell,\tau}^{(1)}(t),
		\beta_{\ell,1,\tau}^{(2)}(t),\beta_{\ell,2,\tau}^{(2)}(t),\dots, \beta_{\ell,\ell,\tau}^{(2)}(t)\Big\},
	\end{align}
	be a set of independent, real-valued time-delayed Brownian motions with variances $1$ when $t=\tau+1$.

Using the above setting, we have the following proposition. Its proof is similar to the one from \cite[Proposition 4.3] {DaPrato}.

\begin{prop}\label{newPropW}
The noise $W_{\tau}(t),\ t\geq\tau\geq0$, admits the following representation
\begin{align}\label{Win1}
W_{\tau}(t)&:= \sum_{\ell=0}^\infty \sum_{m=-\ell}^\ell \widehat{W}_{\ell,m,\tau}(t) Y_{\ell,m},\quad t\geq\tau\geq0,
\end{align}
where $\widehat{W}_{\ell,m,\tau}(t)$, $t\geq\tau\geq0$, $\ell\in\N_{0}$, $m=-\ell,\dots,\ell$, are defined by 





 \noindent
\begin{align}\label{WincomplexFourier}
\widehat{W}_{\ell,m,\tau}&:=
	\begin{cases} 
			\sqrt{\calA_{\ell}}\beta_{\ell,0,\tau}^{(1)} , & m=0,	\\
			\sqrt{\frac{\calA_{\ell}}{2}}\Big(\beta_{\ell,m,\tau}^{(1)}
			- \mi \beta_{\ell,m,\tau}^{(2)}\Big), & m=1,\ldots,\ell, \\
			(-1)^m\sqrt{\frac{\calA_{\ell}}{2}}\Big(\beta_{\ell,|m|,\tau}^{(1)}
			+ \mi \beta_{\ell,|m|,\tau}^{(2)}\Big), & m=-\ell,\ldots,-1,
		\end{cases}
	\end{align}
	where for $\ell\in\N_{0}, m_1=0,\dots,\ell$, $m_2=1,2,\dots,\ell$, $\beta_{\ell,m_1,\tau}^{(1)},\beta_{\ell,m_2,\tau}^{(2)}\in\mathcal{B}_{\ell,\tau}$ and $\mathcal{B}_{\ell,\tau}$ is defined by \eqref{B-iid}.
\end{prop}



\indent
It follows, using \eqref{WincomplexFourier}, that the expansion \eqref{Win1} can be written, in the $L_{2}(\Omega\times\bS^2)$ sense, as 
	\begin{align}\label{Win}
		W_{\tau}(t)&\simeq\sum_{\ell=0}^\infty\Bigl( \sqrt{\calA_{\ell}}\beta_{\ell,0,\tau}^{(1)}(t)Y_{\ell,0}\notag\\
		&+\sqrt{2\calA_{\ell}}\sum_{m=1}^{\ell}\Bigl(\beta_{\ell,m,\tau}^{(1)}(t)\Re Y_{\ell,m}+\beta_{\ell,m,\tau}^{(2)}(t)\Im Y_{\ell,m}\Big)\Big),\quad t\geq\tau\geq0.
	\end{align}

Since by \eqref{eq:condAell}, $\sum_{\ell=0}^{\infty}(2\ell+1)\calA_{\ell}<\infty$, we obtain, using \eqref{Win}, 
	\begin{align*}
		&\Vert W_{\tau}(t)\Vert_{L_{2}(\Omega\times\bS^2)}^2=\sum_{\ell=0}^\infty\Bigl( \calA_{\ell}\bE\big[(\beta_{\ell,0,\tau}^{(1)}(t))^2\big]\Vert Y_{\ell,0}\Vert_{L_{2}(\bS^2)}^2\\
		&+2\calA_{\ell}\sum_{m=1}^{\ell}\Bigl(\bE\big[(\beta_{\ell,m,\tau}^{(1)}(t))^2\big]\Vert\Re Y_{\ell,m}\Vert_{L_{2}(\bS^2)}^2+\bE\big[(\beta_{\ell,m,\tau}^{(2)}(t))^2\big]\Vert\Im Y_{\ell,m}\Vert_{L_{2}(\bS^2)}^2\Big)\Big)\\
		&=(t-\tau)\sum_{\ell=0}^{\infty} (2\ell+1)\calA_{\ell}<\infty,
	\end{align*}
where the second step uses \eqref{eq:Btau2}. Thus the noise $W_\tau$ is well-defined in $L_{2}(\Omega\times\bS^2)$.

Note that the sequence $\{\calA_\ell:\ell\in\N_{0}\}$ is called the angular power spectrum of $W_{\tau}$. 	

\subsection{Stochastic integrals of Mittag-Leffler functions}\label{Sec3}
	
	In this subsection, we provide some tools and technical results from the theory of stochastic integrals. 
	
It is well-known (see \cite{Kuo2006}) that for a non-random, measurable function $g$ in $L_{2}([s,t])$, and a real-valued Brownian motion $B(t)$, stochastic integrals $\int_{s}^{t}g(u)dB(u)$, are defined as Wiener integrals. We can similarly define stochastic integrals, in the $L_{2}(\Omega\times\bS^2)$ sense, of the form $\int_{s}^{t}g(u)dW_{\tau}(u)$, $t>s\geq\tau$.
	
The following definition states that an $L_{2}(\bS^2)-$valued stochastic integral $\int_{s}^{t}g(u)dW_{\tau}(u)$ can be defined as an expansion in spherical harmonics with random coefficients of the form $\int_s^t g(u)d\widehat{W}_{\ell,m,\tau}(u)$.

\begin{defin}\label{def1}
Let $ W_{\tau}$ be an $L_2(\bS^2)$-valued time-delayed Brownian motion on $\bS^2$. Let $\widehat{W}_{\ell,m,\tau}(t)$, $t\geq\tau$, $\ell\in\N_{0}$, $m=-\ell,\dots,\ell$, be the Fourier coefficients for $ W_{\tau}$ defined by \eqref{WincomplexFourier}. For $t>s\geq\tau$, the stochastic integral $\int_s^t g(u)dW_{\tau}(u)$ for a measurable function $g$ in $L_{2}([s,t])$ is defined, in the $L_{2}(\Omega\times\bS^2)$ sense, by
		\[
		\int_s^t g(u)dW_{\tau}(u):= \sum_{\ell=0}^\infty \sum_{m=-\ell}^\ell\left( \int_s^t g(u)d\widehat{W}_{\ell,m,\tau}(u)\right) Y_{\ell,m}.
		\]
	\end{defin}
	
	The following result gives expressions for stochastic integrals of the form $\int_s^t g(u)dW_{\tau}(u)$ in terms of the angular power spectrum of $W_{\tau}$.
	\begin{prop}\label{Prop1}
		Let $ W_{\tau}$ be an $L_2(\bS^2)$-valued time-delayed Brownian motion. Let $\{\calA_\ell:\ell\in\N_{0}\}$ be the angular power spectrum of $W_{\tau}$. Then, for $t>s\geq\tau$, the stochastic integral $\int_s^t g(u)dW_{\tau}(u)$ given by Definition {\rm\ref{def1}} satisfies
		\[
		\bE	\left[ \bigg\|\int_s^t g(u)dW_{\tau}(u)  \bigg\|_{L_{2}(\bS^2)}^2\right]= \sum_{\ell=0}^{\infty}(2\ell+1)\calA_{\ell} \int_s^t |g(u)|^2 du.
		\]
	\end{prop}
	
	\begin{proof}
		By Definition \ref{def1} and Parseval's identity, we get
		\begin{align*}
			\bE	\left[ \bigg\|\int_s^t g(u)dW_{\tau}(u)  \bigg\|_{L_{2}(\bS^2)}^2\right]&=\sum_{\ell=0}^{\infty}\sum_{m=-\ell}^{\ell} \bE\bigg[\bigg|\int_s^t g(u)d\widehat{W}_{\ell,m,\tau}(u)\bigg|^2\bigg]\\
			&=\bigg(\int_s^t |g(u)|^2 du\bigg)\sum_{\ell=0}^{\infty}\calA_{\ell}\sum_{m=-\ell}^{\ell}1\\
			&= \bigg(\int_s^t |g(u)|^2 du\bigg)\sum_{\ell=0}^{\infty}(2\ell+1)\calA_{\ell},
		\end{align*}
where the second step uses It\^{o}'s isometry (see \cite{Kuo2006}) and the last step uses the properties of $Y_{\ell,m}$, which completes the proof.
\end{proof}
	
Let
	\begin{align}\label{EqML}
		E_{\alpha,\beta}(z):=\sum_{k=0}^{\infty}\dfrac{z^{k}}{\Gamma(\alpha k+\beta)},\quad \alpha,\beta>0,\ z\in\bC,
	\end{align}
 be the two-parametric Mittag-Leffler function \cite{Gorenflo2014}. For $\beta=1$, it reduces to the classical Mittag-Leffler function, i.e., $E_{\alpha,1}(z):=E_{\alpha}(z)$. Here $\Gamma(\cdot)$ stands for the gamma function and $\bC$ is the set of complex numbers.
	
The following upper bound holds true \cite[Theorem 1.6] {Podlubny}
\begin{align}\label{EMitagg}
			E_{\alpha,\beta}(-z)\leq \dfrac{C}{1+z},\quad \alpha<2,\ \beta\in\RR,\ C>0,\ z\geq0.
\end{align}

For $\alpha\in(0,1]$, let $M_\alpha$ be defined as
	\begin{align}\label{C2}
		M_\alpha:=
		\begin{cases} 
			(\Gamma(1+\alpha))^2\vert 2\alpha-1\vert^{-1}, & \alpha\neq \frac{1}{2}, \\
			(\Gamma(3/2))^2, & \alpha=\frac{1}{2}.
		\end{cases}
	\end{align}
	
	For $\ell\in\N_{0}$, $t\geq0$, and $\alpha\in(0,1]$, let $\sigma_{\ell,t,\alpha}^2$ be defined as
	\begin{align}\label{var}
		\sigma_{\ell,t,\alpha}^2:=\int_0^t (E_{\alpha}(-\lambda_{\ell}(t-u)^\alpha))^2 du=\int_0^t (E_{\alpha}(-\lambda_{\ell}r^\alpha))^2 dr,
	\end{align}
	where $E_{\alpha}(\cdot)$ is defined in \eqref{EqML} and  $\lambda_{\ell}$ is given by \eqref{lam}. 
	\begin{rem}\label{rem1}
		It is well-known that the Mittag-Leffler function $E_{\alpha}(-\lambda_{\ell}t^\alpha)$, $\alpha\in(0,1]$, $t>0$, is positive and decreasing, with
		$E_{\alpha}(0)=1$. Thus, for
		$t\geq0$, $\ell\in\N_{0}$, there holds $E_{\alpha}(-\lambda_{\ell}t^\alpha)\in(0,1]$,
		$(E_{\alpha}(-\lambda_{\ell}t^\alpha))^2\leq E_{\alpha}(-\lambda_{\ell}t^\alpha)$ and the integral defined in {\rm(\ref{var})} is finite, i.e.,  
		\[
		\sigma_{\ell,t,\alpha}^2=\int_0^t (E_{\alpha}(-\lambda_{\ell}r^\alpha))^2 dr\leq \int_0^t E_{\alpha}(-\lambda_{\ell}r^\alpha)dr\leq t<\infty.
		\]
	\end{rem}
	\begin{prop}\label{PropVar}
		Let $\alpha\in(0,1]$ and $\calI_{\ell,m,\alpha}^{(j)}(t)$, $j=1,2$, $\ell\in\N_{0}$, $m=0,\dots,\ell$, be stochastic integrals defined as 
		\begin{align}\label{Int}
			\mathcal{I}_{\ell,m,\alpha}^{(j)}(t):=\int_{0}^{t}E_{\alpha}(-\lambda_{\ell}(t-u)^{\alpha})d\beta_{\ell,m}^{(j)}(u),\quad t>0,
		\end{align}
		where for each $j=1,2$, the process $\beta_{\ell,m}^{(j)}(u)$ is an independent real-valued Brownian motion with variance $1$ at $u=1$. Then, for $m=0,\dots,\ell$, $\ell\in \N_{0}$, each stochastic integral in \eqref{Int} 
		is normally distributed with mean zero and variance $\sigma_{\ell,t,\alpha}^2$, where $\sigma_{\ell,t,\alpha}^2$ is given by {\rm \eqref{var}}.
	\end{prop}
	
	\begin{proof}
		Note that for $t>0$ and $\lambda_{\ell}\geq0$,  $E_{\alpha}(-\lambda_{\ell}t^\alpha)$ is a bounded measurable function on $\RR_{+}$. Then, by It\^o's isometry (see \cite{Kuo2006}), for $j=1,2$, there holds
		\begin{align*}
			\bE\left[(\calI_{\ell,m,\alpha}^{(j)}(t))^2\right]&=\bE	\left[ \bigg|\int_0^t E_{\alpha}(-\lambda_{\ell}(t-u)^\alpha)d\beta_{\ell,m}^{(j)}(u)  \bigg|^2\right]\notag\\
			&= \bE\bigg[\int_0^t |E_{\alpha}(-\lambda_{\ell}(t-u)^\alpha)|^2 du\bigg]\notag\\
			&=\int_0^{t} (E_{\alpha}(-\lambda_{\ell}r^\alpha))^2 dr=\sigma_{\ell,t,\alpha}^2.	
		\end{align*}
		By \cite{Kuo2006} (see Theorem 2.3.4),  we conclude that each stochastic integral defined in \eqref{Int} is a Gaussian random variable with mean zero and variance $\sigma_{\ell,t,\alpha}^2$ given by \eqref{var}, thus completing the proof.
	\end{proof}

	\begin{prop}
		For $t>0$ and $\ell\in\N$ we have
		\begin{align}\label{N-int40}
			\sigma_{\ell,t,\alpha}^2\leq \begin{cases} 
				\lambda_{\ell}^{-\frac{1}{\alpha}}+M_{\alpha}t^{1-2\alpha}\lambda_{\ell}^{-2}, & \alpha\in(0,\frac{1}{2}), \\
				\lambda_{\ell}^{-\frac{1}{\alpha}}(1+M_{\alpha}), & \alpha\in(\frac{1}{2},1],\\
				\lambda_{\ell}^{-2}	(1+M_{\frac{1}{2}}\ln(\lambda_{\ell}^2t)), & \alpha=\frac{1}{2},
			\end{cases}
		\end{align}	
		where $M_\alpha$ is given by \eqref{C2}.
	\end{prop}

	\begin{proof}
		Using \eqref{var} with the substitution $s=\lambda_{\ell}^{\frac{1}{\alpha}}r$ we get
		\begin{align}\label{N-Sig}
			\sigma_{\ell,t,\alpha}^2=\lambda_{\ell}^{-\frac{1}{\alpha}}\int_{0}^{\lambda_{\ell}^{\frac{1}{\alpha}}t}(E_{\alpha}(-s^\alpha))^2ds.
		\end{align}

		\noindent
		Suppose that $t\lambda_{\ell}^{\frac{1}{\alpha}}>1$ and write, using \eqref{N-Sig},
		\begin{align}\label{N-int35}
			\sigma_{\ell,t,\alpha}^2&=\lambda_{\ell}^{-\frac{1}{\alpha}}\Big(\int_{0}^{1}(E_{\alpha}(-s^\alpha))^2ds+\int_{1}^{\lambda_{\ell}^{\frac{1}{\alpha}}t}(E_{\alpha}(-s^\alpha))^2ds\Big)\notag\\
			&\leq\lambda_{\ell}^{-\frac{1}{\alpha}}\Big(1+(\Gamma(1+\alpha))^2\int_{1}^{\lambda_{\ell}^{\frac{1}{\alpha}}t}s^{-2\alpha}ds\Big),
		\end{align}
		where the second step uses Remark \ref{rem1} and the upper bound (see \cite{Simon2014}, Theorem 4), 
		\begin{align}\label{N-Simon}
			E_{\alpha}(-s^\alpha)&\leq \dfrac{1}{1+(\Gamma(1+\alpha))^{-1}s^\alpha}\leq\Gamma(1+\alpha)s^{-\alpha},\quad s>0.
		\end{align}	
		\noindent
		Note that if $\alpha\in(0,\frac{1}{2})$, then	
		\begin{align}\label{N-int37}
			\int_{1}^{\lambda_{\ell}^{\frac{1}{\alpha}}t}s^{-2\alpha} ds&= \dfrac{1}{1-2\alpha}\Big(\Big(\lambda_{\ell}^{\frac{1}{\alpha}}t\Big)^{1-2\alpha}-1\Big)\notag\\
			&\leq \dfrac{1}{1-2\alpha}\Big(\lambda_{\ell}^{\frac{1}{\alpha}}t\Big)^{1-2\alpha}\notag\\
			&\leq \dfrac{1}{1-2\alpha}t^{1-2\alpha}\lambda_{\ell}^{\frac{1}{\alpha}-2}.
		\end{align}	
		Similarly if $\alpha\in(\frac{1}{2},1]$, then	
		\begin{align}\label{N-int38}
			\int_{1}^{\lambda_{\ell}^{\frac{1}{\alpha}}t}s^{-2\alpha} ds&= \dfrac{1}{2\alpha-1}\Big(1-\Big(\lambda_{\ell}^{\frac{1}{\alpha}}t\Big)^{1-2\alpha}\Big)\notag\\
			&\leq \dfrac{1}{2\alpha-1}.
		\end{align}	
		For the critical case $\alpha=\frac{1}{2}$,
		\begin{align}\label{N-int39}
			\int_{1}^{\lambda_{\ell}^{\frac{1}{\alpha}}t}s^{-2\alpha} ds&=\ln(\lambda_{\ell}^{\frac{1}{\alpha}}t).
		\end{align}	
		Thus the result follows by using the estimates \eqref{N-int37}, \eqref{N-int38}, and \eqref{N-int39} with \eqref{N-int35}.
	\end{proof}
In subsequent sections, some results are required for establishing the pathwise solution of \eqref{pde}. In particular, a formal justification for taking the fractional derivative of stochastic integrals defined by \eqref{Int} is required. Since there is currently no existing result in the literature on this matter, we embark on proving the following theorem.
\begin{theo}
Let $B(t)$ be a real-valued Brownian motion with variance $1$ at $t=1$. Let $\Upsilon(t,s): \RR_{+}\times\RR_{+}\to \RR$ be a deterministic function such that both $\Upsilon(t,s)$ and  ${}_{s} D_{t}^{1-\alpha} \Upsilon(t,s)$, $\alpha\in(0,1)$, are continuous in $t$ and $s$ in some region of the 
	$ts$-plane. Suppose that for $\alpha\in(0,1)$
	the following stochastic integrals are well-defined  
	\[
	\int_{0}^{t}\Upsilon(s,t)dB(s)\ \text{and}\ \int_{0}^{t} \left({}_{s} D_{t}^{1-\alpha} \Upsilon(t,s)\right)dB(s).
	\]
Assume moreover that 
\begin{align}\label{con:Fubini}
    \int_{0}^{t} (t-s)^{\alpha-1}\Bigg(\int_{0}^{s}|\Upsilon(s,r)|^2 dr\Bigg)^{\frac{1}{2}}ds< \infty,
\end{align}
then for all $t>0$ there holds
\begin{align}\label{DK}
D_{t}^{1-\alpha}\Bigg(\int_{0}^{t} \Upsilon(t,r) d B(r)\Bigg)dt \simeq \Bigg( \int_{0}^{t} {}_{r} D_{t}^{1-\alpha} \Upsilon(t,r) d B(r)\Bigg)dt +\big(\lim_{r\to t} {}_{r}D_{t}^{-\alpha}\Upsilon(t,r)\big)dB(t),
\end{align}
where $	{}_{x}D_{t}^{-p} \Upsilon(t,x)=  \frac{1}{\Gamma(p)}\int_{x}^{t} (t-s)^{p-1}\Upsilon(s,x)ds$, $p>0$, is a $p$-fold integral, see~\cite[Section 2.3.2] {Podlubny}.
\end{theo}
\begin{proof}
Using \eqref{D:Frac} and \eqref{con:Fubini} with the stochastic Fubini theorem (see \cite[Section 4.5] {DaPrato}) we get 
	\begin{align*}
		D_{t}^{1-\alpha}\Bigg(\int_{0}^{t} \Upsilon(t,r) d B(r)\Bigg)dt&\simeq\dfrac{1}{\Gamma(\alpha)}d\Bigg(\int_{0}^{t} \Bigg(\dfrac{1}{(t-s)^{1-\alpha}} \int_{0}^{s} \Upsilon(s,r) d B(r)\Bigg)ds\Bigg)\\
		&\simeq d \Bigg(\int_{0}^{t} \Bigg(\dfrac{1}{\Gamma(\alpha)} \int_{r}^{t} \dfrac{\Upsilon(s,r)}{(t-s)^{1-\alpha}}ds\Bigg) d B(r)\Bigg)\\
		&\simeq d \Big(\int_{0}^{t} \widetilde{\Upsilon}(t,r) d B(r)\Big),
	\end{align*}
where $d$ denotes the $t$-differential and
	\[
	\widetilde{\Upsilon}(t,x):= \dfrac{1}{\Gamma(\alpha)}\int_{x}^{t} \dfrac{\Upsilon(s,x)}{(t-s)^{1-\alpha}}ds.
	\]
By Proposition 1.3 in \cite{Hess2023}, we could write
	\begin{align*}
		D_{t}^{1-\alpha}\Bigg(\int_{0}^{t} \Upsilon(t,r) d B(r)\Bigg) dt&\simeq  \Big(\int_{0}^{t} \dfrac{\partial}{\partial t} \widetilde{\Upsilon}(t,r) d B(r)\Big) dt+\lim_{r\to t} \widetilde \Upsilon(t,r)dB(t)\\
		&\simeq 	\int_{0}^{t} {}_{r} D_{t}^{1-\alpha} \Upsilon(t,r) d B(r)dt +\big(\lim_{r\to t} {}_{r}D_{t}^{-\alpha}\Upsilon(t,r)\big)dB(t),
	\end{align*} 
which completes the proof.
\end{proof}
It is worth noting that if we have $\Upsilon(t-r)$ instead of the function $\Upsilon(t,r)$, then \eqref{DK} takes the form 
\begin{align}\label{eq:D}
	D_{t}^{1-\alpha}\Bigg(\int_{0}^{t} \Upsilon(t-r) d B(r)\Bigg)dt\simeq \Bigg(\int_{0}^{t} D_{r}^{1-\alpha} \Upsilon(r) d B(r)\Bigg)dt+ \big(\lim_{r\to 0} D_{r}^{-\alpha}\Upsilon(r)\big)dB(t).
\end{align}

\begin{lem}\label{frD}
Let $t>\tau>0$. For $\ell\geq1$ and $\alpha\in(0,1)$, there holds
    \begin{align}\label{DE}
	D_{t}^{1-\alpha}\int_{0}^{t-\tau} E_{\alpha}( -\lambda_{\ell} (t-\tau-r)^\alpha) d B(r)\simeq \int_{\tau}^{t} D_{r}^{1-\alpha} E_{\alpha}( -\lambda_{\ell} r^\alpha) d B_{\tau}(r),
\end{align}
where $\lambda_\ell$ is given by \eqref{lam} and  $B_\tau$ is a real-valued, time-delayed Brownian motion starting at time $\tau>0$. 
\end{lem}

\begin{proof}
By Remark \eqref{rem1},
\begin{align*}
    \int_{\tau}^{t} (t-s)^{\alpha-1}\Bigg(\int_{0}^{s}|E_{\alpha}(-\lambda_{\ell}r^\alpha)|^2 dr\Bigg)^{\frac{1}{2}}ds&\le \sqrt{t}\int_{\tau}^{t}(t-s)^{\alpha-1}ds\\
    &\le \sqrt{t}(t-\tau)^{\alpha}/\alpha\\
    &< \infty,
\end{align*}
and by Proposition \ref{PropVar}, the stochastic integral
\[
\int_{0}^{t-\tau} E_{\alpha}( -\lambda_{\ell} (t-\tau-r)^\alpha) d B(r)
\]
is well-defined.

Let us show that the stochastic integral 
\[
\int_{\tau}^{t} D_{r}^{1-\alpha} E_{\alpha}( -\lambda_{\ell} r^\alpha) d B_{\tau}(r)
\]
is well-defined. 
 Using It\^o's isometry and Lemma 3.10 in \cite{Podlubny} we obtain 
\begin{align*}
\bE\Bigg[\Bigg\vert\int_{\tau}^{t} D_{r}^{1-\alpha} E_{\alpha}( -\lambda_{\ell} r^\alpha) d B_{\tau}(r)\Bigg\vert^2\Bigg]&= \int_{\tau}^{t} \big| D_{r}^{1-\alpha} E_{\alpha}( -\lambda_{\ell} r^\alpha) \big|^2 dr\\
&= \dfrac{1}{\lambda_{\ell}^2} \int_{\tau}^t \Bigg( \dfrac{-\lambda_{\ell}}{\Gamma(\alpha)}\dfrac{d}{dr}	\int_{0}^r \dfrac{E_{\alpha}( -\lambda_{\ell} s^\alpha)}{(r-s)^{1-\alpha}} \Bigg)^2 dr\\
&=\dfrac{1}{\lambda_{\ell}^2}\int_{\tau}^t \Big(\dfrac{d}{dr}(E_{\alpha}( -\lambda_{\ell} r^\alpha)-1)\Big)^2 dr\\
&\leq C\lambda_{\ell}^{-2}\tau^{-1}<\infty,
\end{align*}
where the last second step used the relations (4.3.1) and (5.1.14) in \cite{Gorenflo2014} and the upper bound \eqref{EMitagg}. 

By the relation \eqref{eq:D} we obtain
\begin{align}\label{FracUellm}
	D_{t}^{1-\alpha}\Bigg(\int_{0}^{t-\tau} E_{\alpha}( -\lambda_{\ell} (t-\tau-r)^\alpha) d B(r)\Bigg)dt\simeq& \Bigg(\int_{\tau}^{t} D_{r}^{1-\alpha} E_{\alpha}( -\lambda_{\ell} r^\alpha) d B_{\tau}(r)\Bigg)dt\nonumber\\
 &+ \lim_{r\to 0} \big(D_{r}^{-\alpha}E_{\alpha}( -\lambda_{\ell} r^\alpha)\big)dB_{\tau}(t).
\end{align}

Using equations (2.88) and (1.100) in \cite{Podlubny}, 
\begin{align*}
	\lim_{r\to 0} D_{r}^{-\alpha}E_{\alpha}( -\lambda_{\ell} r^\alpha)&= \lim_{r\to 0}\dfrac{1}{\Gamma(\alpha)} \int_{0}^{r} (r-s)^{\alpha-1} E_{\alpha}( -\lambda_{\ell} s^\alpha)ds\\
	&= \lim_{r\to 0} r^{\alpha} E_{\alpha,1+\alpha}( -\lambda_{\ell} r^\alpha)=0.
\end{align*}

Thus, equation \eqref{FracUellm} becomes
\begin{align*}
	D_{t}^{1-\alpha}\int_{0}^{t-\tau} E_{\alpha}( -\lambda_{\ell} (t-\tau-r)^\alpha) d B(r)\simeq \int_{\tau}^{t} D_{r}^{1-\alpha} E_{\alpha}( -\lambda_{\ell} r^\alpha) d B_{\tau}(r),
\end{align*}
which completes the proof.
\end{proof}

\section{Solution of the homogeneous problem}\label{Sec4}
Consider a random field $U^{H}(t)\in L_2(\Omega\times\bS^2),\ t\in (0,\infty)$. We say that the field $U^H$ satisfies the equation \eqref{to}, in the $L_2(\Omega\times\bS^2)$ sense, if for a given $t>0$, there holds
\[
	\sup_{t > 0} \left\| \frac{d}{dt} U^H(t) - D^{1-\alpha}_t 
	\Delta_{\bS^2} U^H(t) \right\|^2_{L_2(\Omega\times\bS^2)} = 0.
\]	
This section derives the solution  $U^{H}(t)$, $t\in (0,\infty)$, to the homogeneous equation \eqref{to} under the random initial condition $U^{H}(0)=\xi$, where $\xi$ is a centered, strongly isotropic Gaussian random field on $\bS^2$.
	
Note that since $U^{H}\in L_2(\Omega\times\bS^2)$, it follows that the field $U^{H}$ takes the series expansion 

\begin{align*}
		U^{H}(t) = \sum_{\ell=0}^\infty \sum_{m=-\ell}^{\ell} \widehat{U^{H}}_{\ell,m}(t) Y_{\ell,m}, \quad t\in(0,\infty),
\end{align*}
where the random coefficients $\widehat{U^{H}}_{\ell,m}$, $\ell\in\N_{0}$, $m=-\ell,\dots,\ell$, are given by
\begin{align*}
	\widehat{U^{H}}_{\ell,m}(t)=\int_{\bS^2}U^{H}(\bsx,t)\overline{Y_{\ell,m}(\bsx)}\mu(d\bsx),\quad t\in(0,\infty).
	\end{align*}
By multiplying both sides of \eqref{to} by $\overline{Y_{\ell,m}}$ and integrating over $\bS^2$, we obtain, with the help of \eqref{beltrami}, the following set of ordinary differential equations 
	\begin{equation}\label{ord}
		d \widehat{U^{H}}_{\ell,m} + \lambda_{\ell} D_t^{1-\alpha} \widehat{U^{H}}_{\ell,m} dt = 0, \quad \widehat{U^{H}}_{\ell,m}(0) = \widehat{\xi}_{\ell,m}.
	\end{equation}
We denote by $\widetilde{U^{H}}_{\ell,m}(z)$ the Laplace transform of $\widehat{U^H}_{\ell,m}$ with respect to $t$, i.e.,
\begin{align}\label{Laplace}
		\widetilde{U^H}_{\ell,m}(z)=\calL(\widehat{U^H}_{\ell,m}(t)):=\int_{0}^{\infty}e^{-tz}\widehat{U^H}_{\ell,m}(t)dt,
\end{align}
assuming that the integral converges. The condition for its
convergence will be given in Proposition~\ref{Theo1}.
Since 
	$\calL(D_{t}^{1-\alpha}\widehat{U^H}_{\ell,m}(t))=z^{1-\alpha}\widetilde{U^H}_{\ell,m}(z)$ and $\calL(\frac{d}{dt}\widehat{U^H}_{\ell,m}(t))=z\widetilde{U^H}_{\ell,m}(z)-\widehat{U^H}_{\ell,m}(0)$, then by taking the Laplace transform of both sides of \eqref{ord}, we get
	\begin{equation*}
		z \widetilde{U^H}_{\ell,m}(z) - \widehat{\xi}_{\ell,m} + \lambda_{\ell} z^{1-\alpha} \widetilde{U^H}_{\ell,m} (z) = 0.
	\end{equation*}
	Now solving for $\widetilde{U^H}_{\ell,m} (z)$, we arrive at 	\begin{equation}\label{Lap}
		\widetilde{U^H}_{\ell,m} (z) = \frac{ \widehat{\xi}_{\ell,m}}{ z+ \lambda_{\ell} z^{1-\alpha}}. 
	\end{equation}
	Recall that (see \cite{Podlubny}, equation 1.80)
	\[
	\calL\{ E_\alpha(-\lambda_{\ell} t^\alpha)\} = \frac{1} { z+ \lambda_{\ell} z^{1-\alpha}},
	\]
	where $E_\alpha(\cdot)$ is the Mittag-Leffler function (see equation \eqref{EqML}).
	
	\indent	
	By taking the inverse Laplace transform in \eqref{Lap}, the solution of \eqref{to} becomes
 \begin{align}\label{homo}
     U^{H}(t) = \sum_{\ell=0}^\infty \sum_{m=-\ell}^{\ell} E_\alpha(-\lambda_{\ell} t^\alpha) \widehat{\xi}_{\ell,m} Y_{\ell,m}, \quad t\in[0,\infty).
 \end{align}	

In order to prove that $U^H$ is the solution of \eqref{to} we need to prove a truncation field $U_L^H$ of $U^H$ is a solution and then pass it to the limit as the truncation degree $L \to \infty$. Also, we need the uniform convergence of $\frac{d}{dt} U^H_L$ to $\frac{d}{dt}U^H$ in some sense. We begin with the following lemma.

Let $V_{L}^{H}(t)$ be defined as
\begin{align}\label{VL}
 V_{L}^{H}(t):= \sum_{\ell=1}^{L}\sum_{m=-\ell}^\ell
			(-\lambda_{\ell}) t^{\alpha-1}E_{\alpha,\alpha}(-\lambda_{\ell} t^{\alpha})\widehat{\xi}_{\ell,m} Y_{\ell,m}= 
   \frac{d}{dt} U_L^H(t),  
\end{align}
where $\widehat{\xi}_{\ell,m}$ are the Fourier coefficients of $\xi$. 

\begin{lem}\label{unif-sure}
Let $\xi$, be a centered, $2$-weakly isotropic Gaussian random field on $\bS^2$. Let $\{\calC_{\ell}:\ell\in\N_{0}\}$, the angular power spectrum of $\xi$, satisfy \eqref{eq:condCell}. 
Let $t_0>0$ be given. For $t\in[t_0,\infty)$, then $V_{L}^{H}$ is convergent to
\begin{equation}\label{eq:defVH}
V^{H}(t):= \sum_{\ell=1}^{\infty}\sum_{m=-\ell}^\ell
			(-\lambda_{\ell}) t^{\alpha-1}E_{\alpha,\alpha}(-\lambda_{\ell} t^{\alpha})\widehat{\xi}_{\ell,m} Y_{\ell,m}
\end{equation}
in the following sense
\[
\sup_{t \geq t_0} \bE\Big[ \left\| V_L^H(t) - V^H(t) \right\|^2_{L_2(\bS^2)} \Big]
\to 0,\quad  \text{ as } L \to \infty.
\]
\end{lem}

\begin{proof}
For $t\geq t_0$, $L<M$, by Parseval's formula, we have 
\begin{align}
  \bE \Big[\|V^H_L(t) - V^H_M(t) \|^2_{L_2(\bS^2)}\Big]
  &= \bE \left[\sum_{\ell=L+1}^M
  \sum_{m=-\ell}^{\ell}
  \lambda^2_{\ell} t^{2\alpha-2}
  (E_{\alpha,\alpha}(-\lambda_{\ell} t^\alpha))^2
  |\widehat{\xi}_{\ell,m}|^2  \right] \nonumber\\
  &= \sum_{\ell=L+1}^M
  \sum_{m=-\ell}^{\ell}
  \lambda^2_{\ell} t^{2\alpha-2}
  (E_{\alpha,\alpha}(-\lambda_{\ell} t^\alpha))^2
  \bE\big[|\widehat{\xi}_{\ell,m}|^2\big] \nonumber\\
  &=\sum_{\ell=L+1}^M
\lambda^2_{\ell} t^{2\alpha-2}
  (E_{\alpha,\alpha}(-\lambda_{\ell} t^\alpha))^2
  (2\ell+1) \calC_{\ell} 
\text{   by using \eqref{eq:defCell} }\nonumber \\
&\le C^2 t^{-2}
\sum_{\ell=L+1}^M (2\ell+1) \calC_{\ell} 
\text{   by using \eqref{EMitagg} } \nonumber\\
&\le C^2 t_0^{-2}
\sum_{\ell=L+1}^M (2\ell+1) \calC_{\ell} \label{eq:cauchyCond}.
\end{align}
For a given $\epsilon>0$, using the condition \eqref{eq:condCell}, there is 
$L_0$ independent of $t$ 
such that for $L,M  \ge L_0$, 
the RHS of \eqref{eq:cauchyCond} is smaller than $\epsilon$. So
$\{V_L^H\}$ is a Cauchy sequence in $L_2(\Omega \times \bS^2)$ and hence it
is convergent. We define the limit of $V_L^H$ as $L \to \infty$ as in
\eqref{eq:defVH}. 
Note that by the condition \eqref{eq:condCell} we obtain
\[
\bE \Big[\|V^H(t)\|^2_{L_2(\bS^2)}\Big]\le C^2 t_0^{-2}
\sum_{\ell=1}^{\infty} (2\ell+1) \calC_{\ell}<\infty,
\]
which guarantees that $V^H$ is well-defined in the $L_2(\Omega\times\bS^2)$ sense.
\end{proof}

Now we prove the following lemma in which we adapt the techniques from the proofs of \cite[Theorems 7.11 and 7.17]{Rudin1976}.
\begin{lem}\label{dUH}
Let $T> t_0>0$ be fixed. 
For $\widehat{t} \in [t_0, T]$, we have
\[
\lim_{L \to \infty} 
\bE \Bigg[\left\| \frac{dU^H}{dt}(\widehat{t}) - V^H_L(\widehat{t}) \right\|^2_{L_2(\bS^2)}\Bigg] = 0.
\]
\end{lem}
\begin{proof}
Let us define for $t\neq \widehat{t}$
\[
\phi(t) := 
\frac{U^H(t)-U^H(\widehat{t})}
{t-\widehat{t}}
\]
and
\[
\phi_L(t) := 
\frac{U^H_L(t)-U^H_L(\widehat{t})}
{t-\widehat{t}}.
\]
Then we have 
\begin{equation}\label{eq:limphi}
\lim_{t \to \widehat{t}} \phi(t) = 
\frac{dU^H}{dt}(\wt) = V(\wt).
\end{equation}
We also have 
\begin{align}
\lim_{t \to \widehat{t}} \phi_L(t) &= 
\sum_{\ell=0}^L
\lim_{t \to \wt}
\frac{E_{\alpha}(-\lambda_{\ell} t^\alpha) - E_{\alpha}(-\lambda_{\ell} \wt^\alpha)}{t-\wt} 
\sum_{m=-\ell}^{\ell}
\widehat{\xi}_{\ell,m} Y_{\ell,m} \nonumber\\
&=\sum_{\ell=0}^L
\frac{d}{d\hat{t}} E_{\alpha}(-\lambda_{\ell}\hat{t}^\alpha)
\sum_{m=-\ell}^{\ell}
\widehat{\xi}_{\ell,m} Y_{\ell,m}\nonumber\\
&=\sum_{\ell=1}^L 
(-\lambda_{\ell}) \hat{t}^{\alpha-1} E_{\alpha,\alpha}(-\lambda_{\ell} \hat{t}^\alpha)\sum_{m=-\ell}^{\ell}
\widehat{\xi}_{\ell,m} Y_{\ell,m}
=V^H_L(\widehat{t}).    \label{eq:limphiL}
\end{align}
We remark, for a fixed $\omega\in\Omega$, that $V_L^H(\wt)<\infty$ for $\wt\geq t_0>0$.

Since by Lemma \ref{unif-sure},  $V^H_L(t)$ and $V^H_L(\wt)$ converge uniformly to $V^H(t)$ and $V^H(\wt)$ respectively in $L_2(\Omega \times \bS^2)$, for any given $\epsilon>0$, we can choose $L_0$ independent of 
$t$ such that for 
all $L,M \ge L_0$, 
\begin{equation}
    \bE\big[\| U^H_L(\wt) - U^H_M(\wt)-U^H_L(t)+U^H_M(t)\|^2_{L_2(\bS^2)}\big] 
    \le \frac{|\wt-t|^2}{(T-t_0)^2} \epsilon.
\end{equation}
Therefore, 
\begin{equation}
\bE \big[\| \phi_L(t) - \phi_M(t)\|^2_{L_2(\bS^2)}\big] \le \epsilon.
\end{equation}
This shows that $\{\phi_L\}$ is a Cauchy sequence in
$L_2(\Omega \times \bS^2)$ and hence for 
$t \in [t_0,T]$ and $t \ne \wt$,
\begin{equation}\label{eq:phiL}
\bE \big[\| \phi(t) - \phi_L(t)\|^2_{L_2(\bS^2)}\big] \le \epsilon.
\end{equation}
Now, by taking the limit as $t \to \wt$ in \eqref{eq:phiL},
with the help of \eqref{eq:limphi} and \eqref{eq:limphiL} we obtain
\[
\bE \Bigg[\left\| \frac{dU^H}{dt}(\wt) - V^H_L(\wt) \right\|^2_{L_2(\bS^2)}\Bigg]
\le \epsilon,
\]
which completes the proof.
\end{proof}

Let $G_L$ be defined as 
\begin{align}\label{GL}
    G_L(t) := \sum_{\ell=1}^L (-\lambda_{\ell}) E_{\alpha}(-\lambda_{\ell} t^\alpha)
\sum_{m=-\ell}^\ell \widehat{\xi}_{\ell,m} Y_{\ell,m}.
\end{align}

Using \eqref{GL} with the Parseval formula, \eqref{eq:defCell} and the upper bound \eqref{N-Simon} then for $L<M$ we have
\begin{align}\label{eq:cauchyCond2}
   \bE\|G_L(t) - G_M(t)\|^2_{L_2(\bS^2)}&=\sum_{\ell=L+1}^M (\lambda_{\ell})^2
   (E_{\alpha}(-\lambda_{\ell} t^\alpha))^2 (2\ell+1 ) \calC_{\ell}\notag\\
   &\leq \big(\Gamma(1+\alpha)t^{-\alpha}\big)^2\sum_{\ell=L+1}^M (2\ell+1 ) \calC_{\ell} \nonumber\\
   &\leq \big(\Gamma(1+\alpha)t_0^{-\alpha}\big)^2\sum_{\ell=L+1}^M (2\ell+1 ) \calC_{\ell}.
\end{align}
For a given $\epsilon>0$, using the condition \eqref{eq:condCell}, there is 
$L_0$ independent of $t$ 
such that for $L,M  \ge L_0$, the RHS of \eqref{eq:cauchyCond2} can be smaller than $\epsilon$. We then let
\begin{align}\label{G}
    G(t) := \sum_{\ell=1}^\infty (-\lambda_{\ell}) E_{\alpha}(-\lambda_{\ell} t^\alpha)
\sum_{m=-\ell}^\ell \widehat{\xi}_{\ell,m} Y_{\ell,m},
\end{align}
which is convergent in $L_2(\Omega\times\bS^2)$.

\begin{lem}\label{lem:Dalpha}
Let $\xi$ be a centered, $2$-weakly isotropic Gaussian random field on $\bS^2$. Let $\{\calC_{\ell}:\ell\in\N_{0}\}$, the angular power spectrum of $\xi$, satisfy \eqref{eq:condCell}. Let $t_0>0$ be given. For
$t\geq t_0$ there holds
\[\lim_{L \to \infty} \bE\big[ \| D_{t}^{1-\alpha} G_L(t) - D_{t}^{1-\alpha} G(t)\|^2_{L_2(\bS^2)}\big] = 0,
\]
where $G_L$ and $G$ are given by \eqref{GL} and \eqref{G} respectively.
\end{lem}
\begin{proof}

Using Lemma 3.10 in \cite{Gorenflo2014} we have
\begin{align*}
   D_t^{1-\alpha} G_L(t) &= V_L^H(t),
\end{align*}
which implies, by Lemma \ref{unif-sure}, that $\{D_t^{1-\alpha} G_L\}$ is a Cauchy sequence in $L_2(\Omega\times \bS^2)$
, i.e., 
we can choose $L_0$ independent of 
$t$ such that for all $L,M \ge L_0$, there holds
\[
\bE \big[\|D_{t}^{1-\alpha} G_L(t) - D_{t}^{1-\alpha} G_M(t) \|^2_{L_2(\bS^2)} \big]\le \epsilon.
\]
Hence,
\[
\bE \big[\| D_{t}^{1-\alpha} G_L(t) - D_{t}^{1-\alpha} G(t) \|^2_{L_2(\bS^2)}\big] \le \epsilon,
\]
which completes the proof.
\end{proof}
\begin{lem}\label{lem:DalphaDelta}
For $t \in [t_0,T]$, we have 
\[
\lim_{L\to \infty}  
\bE \Big[\left\| V_L^H(t) - D_{t}^{1-\alpha} \Delta_{\bS^2} U^H(t)\right\|^2_{L_2(\bS^2)}\Big] = 0.
\]   
\end{lem}
\begin{proof}
Using Theorem 4.1.2 (see~\cite[Satz 4.1.2] {Berens_et_al1968}), we can show that
\[
\lim_{L\to \infty} \bE\big[\|\Delta_{\bS^2} U^H(t) - G_L(t)\|_{L_2(\bS^2)}^2\big] = 0.
\]
By Lemma~\ref{lem:Dalpha}, noting that $D_{t}^{1-\alpha} G_L(t) = V_L^H(t)$, 
we obtain the conclusion of the lemma.
\end{proof}
\begin{prop}\label{Theo1}
Let the angular power spectrum $\{\calC_{\ell}:\ell\in\N_{0}\}$ of the isotropic Gaussian random field $\xi$ satisfy assumption \eqref{eq:condCell}. 
Then the random field $U^{H}$ defined by \eqref{homo} satisfies \eqref{to}, in the $L_2(\Omega\times\bS^2)$ sense, under the initial condition $U^{H}(0)=\xi$. 
In particular, for a given $t_0>0$, 
we have
\[
\sup_{t \geq t_0} \bE \Bigg[\left\| \frac{d}{dt} U^H(t) - D^{1-\alpha}_t 
\Delta_{\bS^2} U^H(t) \right\|^2_{L_2(\bS^2)}\Bigg] = 0.
\]
\end{prop}
\begin{proof}
Let us fix $L\geq1$, then we have
\begin{equation*}
\begin{split}
\frac{d}{dt} U^H_L(t)  &=  \sum_{\ell=1}^L \sum_{m=-\ell}^{\ell}
\frac{d}{dt} E_{\alpha}(-\lambda_{\ell} t^{\alpha}) \widehat{\xi}_{\ell,m}
Y_{\ell,m}     \\
&= \sum_{\ell=1}^L \sum_{m=-\ell}^{\ell}(-\lambda_{\ell} ) t^{\alpha-1} E_{\alpha,\alpha}(-\lambda_{\ell}t^\alpha)
\widehat{\xi}_{\ell,m}
Y_{\ell,m}= V^H_L(t).
\end{split}
\end{equation*}
Using Lemma 3.10 in \cite{Gorenflo2014} we have
\begin{align*}
D_t^{1-\alpha} \Delta_{\bS^2} U^H_L(t)
	& = \sum_{\ell=0}^L\sum_{m=-\ell}^\ell
			(-\lambda_{\ell}) D_t^{1-\alpha} E_{\alpha}(-\lambda_{\ell} t^{\alpha}) \widehat{\xi}_{\ell,m} Y_{\ell,m} \\
&=\sum_{\ell=1}^L\sum_{m=-\ell}^\ell
			\frac{d}{dt} (E_{\alpha}(-\lambda_{\ell} t^{\alpha})-1)
			\widehat{\xi}_{\ell,m} Y_{\ell,m}\\
			&= \sum_{\ell=1}^L\sum_{m=-\ell}^\ell
			(-\lambda_{\ell}) t^{\alpha-1}E_{\alpha,\alpha}(-\lambda_{\ell} t^{\alpha})\widehat{\xi}_{\ell,m} Y_{\ell,m} = V_L^H(t).   
\end{align*}
Combining the above two equations, then for any fixed $\omega\in\Omega$ we obtain
\begin{equation}\label{eq:finite}
\frac{d}{dt} U^H_L (t)- D^{1-\alpha}  \Delta_{\bS^2} U^H_L (t) 
 =  V_L^H(t) - V_L^H(t) = 0.
\end{equation}
It follows from \eqref{eq:finite} that
\[
\bE\Bigg[ \left\|\frac{d}{dt} U^H_L(t)-D_{t}^{1-\alpha}  \Delta_{\bS^2} U_{L}^H(t)\right\|^2_{L_2(\bS^2)}\Bigg] = 0.
\]
By the triangle inequality 
\begin{align*}
&\bE\Bigg[\left\| \frac{d}{dt} U^H(t) - D_{t}^{1-\alpha}  \Delta_{\bS^2} U^H(t)   \right\|^2_{L_2(\bS^2)}\Bigg]
\le 2\bE \Bigg[
\left\| \frac{d}{dt} U^H(t) - \frac{d}{dt} U^H_L(t) \right\|^2_{L_2(\bS^2)}\Bigg]\\
&+  2\bE\Bigg[ \left\|D_{t}^{1-\alpha}  \Delta_{\bS^2} U^H_L(t)-D_{t}^{1-\alpha}  \Delta_{\bS^2} U^H(t)\right\|^2_{L_2(\bS^2)}\Bigg].
\end{align*}
By Lemma~\ref{dUH}, we can choose $L$ so that 
\begin{equation}
\bE \Bigg[\left\| \frac{d}{dt} U^H(t) - \frac{d}{dt} U^H_L(t) \right\|^2_{L_2(\bS^2)}\Bigg] \le 
\frac{\epsilon}{4}.
\end{equation}
By Lemma~\ref{lem:DalphaDelta}, we can choose $L$ so that
\begin{equation}
\bE \Big[\left\|D_{t}^{1-\alpha}  \Delta_{\bS^2} U^H_L(t)-D_{t}^{1-\alpha}  \Delta_{\bS^2} U^H(t)\right\|^2_{L_2(\bS^2)}\Big]\le \frac{\epsilon}{4}.    
\end{equation}
So the proposition is proved.
\end{proof}
The following result shows that the solution $U^{H}$ to the equation (\ref{to}) is a centered, 2-weakly isotropic Gaussian random field on $\bS^2$.
	Here we assume that the angular power spectrum  
	$\{\calC_{\ell}:\ell\in\N_{0}\}$ of $\xi$ decays algebraically with order $\kappa_1>2$, i.e., there exist constants $\widetilde{D},\widetilde{C}>0$ such that
	\begin{align}\label{New-Cl}
		\calC_{\ell}\leq\begin{cases} 
			\widetilde{D}, & \ell=0, \\
			\widetilde{C}\ell^{-\kappa_1}, & \ell\geq1,\ \kappa_1>2.
		\end{cases}
	\end{align}
	For $\kappa_1>2$ and $\widetilde{C}>0$, we let
	\begin{align}\label{Ckappa}
		\widetilde{C}_{\kappa_1}:=\bigg(\widetilde{C}\Big(\dfrac{2}{\kappa_1-2}+\dfrac{1}{\kappa_1-1}\Big)\bigg)^{1/2}.
	\end{align}
	
	\begin{prop}\label{Prop4}
		Let the field $U^{H}$, given in {\rm(\ref{homo})}, be the solution to the equation {\rm(\ref{to})} under the initial condition $U^{H}(0)= \xi$, where $\xi$ is a centered, $2$-weakly isotropic Gaussian random field on $\bS^2$. Let $\{\calC_{\ell}:\ell\in\N_{0}\}$, the angular power spectrum of $\xi$, satisfy \eqref{New-Cl}. For a fixed $t\in(0,\infty)$, $U^{H}(t)$ is a centred, $2$-weakly isotropic Gaussian random field on $\bS^2$, and its random coefficients
		\begin{align}\label{xile}
			\widehat{U^{H}}_{\ell,m}(t)=E_\alpha(-\lambda_{\ell} t^\alpha) \widehat{\xi}_{\ell,m},
		\end{align} 
		satisfy for $\ell,\ell^\prime\in\N_{0}$, $m=-\ell,\dots,\ell$ and $m^\prime=-\ell^\prime,\dots,\ell^\prime$, 
		\begin{align}\label{var-xile}
			\bE \left[\widehat{U^{H}}_{\ell,m}(t)\overline{\widehat{U^{H}}_{\ell^\prime,m^\prime}(t)}\right]=(E_\alpha(-\lambda_{\ell} t^\alpha))^2 \calC_{\ell}\delta_{\ell\ell^\prime}\delta_{mm^\prime},
		\end{align}
		where $\lambda_{\ell}$ is given in \eqref{lam} and $\delta_{\ell\ell^\prime}$ is the Kronecker delta function.
	\end{prop}
\begin{proof}
Consider the field $U^{H}$ defined by \eqref{homo}.
Since $\xi$ is centred we have $\bE[U^{H}(t)]=0$. Let $t\in(0,\infty)$ and $\bsx,\bsy\in\bS^2$, then by \eqref{homo}, we can write	
\begin{align*}
			\bE \left[U^{H}(\bsx,t)U^{H}(\bsy,t)\right]&=\sum_{\ell=0}^\infty\sum_{\ell^\prime=0}^\infty \sum_{m=-\ell}^{\ell}
			\sum_{m\prime=-\ell^\prime}^{\ell^\prime} E_\alpha(-\lambda_{\ell} t^\alpha)E_\alpha(-\lambda_{\ell^\prime} t^\alpha)\\ &\times\bE\left[\widehat{\xi}_{\ell,m}\overline{\widehat{\xi}_{\ell^\prime,m^\prime}}\right] Y_{\ell,m}(\bsx)\overline{Y_{\ell^\prime,m^\prime}(\bsy)}.
		\end{align*}
Since  $\xi$ is a centred, 2-weak isotropic Gaussian random field we have $\bE\left[\widehat{\xi}_{\ell,m}\overline{\widehat{\xi}_{\ell^\prime,m^\prime}}\right]=\calC_{\ell}  \delta_{\ell \ell^\prime} \delta_{m m^\prime}$, and by the addition theorem (see equation \eqref{addition}) and Proposition \ref{Coro}, we have
		\begin{align*}
			\bE\left[U^{H}(\bsx,t)U^{H}(\bsy,t)\right]&=\sum_{\ell=0}^\infty(E_\alpha(-\lambda_{\ell} t^\alpha))^2\calC_{\ell}\sum_{m=-\ell}^{\ell} Y_{\ell,m}(\bsx)\overline{Y_{\ell,m}(\bsy)}\\
			&=\sum_{\ell=0}^\infty(E_\alpha(-\lambda_{\ell} t^\alpha))^2(2\ell+1)\calC_\ell P_{\ell}(\bsx\cdot\bsy),
		\end{align*}
		where $P_{\ell}(\cdot)$, $\ell\in\N_{0}$, is the Legendre polynomial of degree $\ell$. As $P_{\ell}(\bsx\cdot\bsy)$ depends only on the inner product of $\bsx$ and $\bsy$, we conclude that the covariance function $\bE\left[U^{H}(\bsx,t)U^{H}(\bsy,t)\right]$ is rotationally invariant. Note that for $\ell,\ell^\prime\in\N_{0}$, $m=-\ell,\dots,\ell$ and $m^\prime=-\ell^\prime,\dots,\ell^\prime$, by the 2-weak isotropy of $\xi$, and Proposition \ref{Coro}, 
		\begin{align*}
			\bE \left[\widehat{U^{H}}_{\ell,m}(t)\overline{\widehat{U^{H}}_{\ell^\prime,m^\prime}(t)}\right]&=E_\alpha(-\lambda_{\ell} t^\alpha)E_\alpha(-\lambda_{\ell^\prime} t^\alpha) \bE\left[\widehat{\xi}_{\ell,m}\overline{\widehat{\xi}_{\ell^\prime,m^\prime}}\right]\\
			&=(E_\alpha(-\lambda_{\ell} t^\alpha))^2 \calC_{\ell}\delta_{\ell\ell^\prime}\delta_{mm^\prime}.
		\end{align*}
		To prove that $U^{H}(t)$ is Gaussian, first note that its variance can be bounded, using \eqref{New-Cl} and that $0< E_{\alpha}(\cdot)\leq1$, by \begin{align}\label{Vpf}
			Var\left[U^{H}(t)\right]&=\sum_{\ell=0}^\infty(E_\alpha(-\lambda_{\ell} t^\alpha))^2(2\ell+1)\calC_\ell\notag\\
			&\leq  \widetilde{D}+ \widetilde{C}\sum_{\ell=1}^\infty (2\ell+1)\ell^{-\kappa_1}\notag\\
			&\leq \widetilde{D}+ \widetilde{C}\int_{1}^\infty \big(2x^{1-\kappa_1}+x^{-\kappa_1}\big)dx\notag\\
			&\leq  \widetilde{D}+ \widetilde{C}\Big(\frac{2}{\kappa_1-2}+\frac{1}{\kappa_1-1}\Big)\notag\\
			&\leq \widetilde{D}+ \widetilde{C}_{\kappa_1}^2<\infty.
		\end{align}
		Now let $T_{\ell}(t)$, $\ell\in\N_0$, be defined as  
		\begin{align*}
			T_{\ell}(t):=\sum_{m=-\ell}^{\ell}\widehat{U^{H}}_{\ell,m}(t)Y_{\ell,m}= \sum_{m=-\ell}^{\ell}E_\alpha(-\lambda_{\ell} t^\alpha) \widehat{\xi}_{\ell,m}Y_{\ell,m},
		\end{align*}
		and note that since $\widehat{\xi}_{\ell,m}$ are centred, independent Gaussian random variables (by the 2-weak isotropy of $\xi$), then $T_{\ell}$
		is a Gaussian random variable with mean zero and variance, using \eqref{var-xile}, \eqref{addition}, and properties of spherical harmonics, given by
		\begin{align}\label{var-T}
			Var[T_{\ell}]=(E_\alpha(-\lambda_{\ell} t^\alpha))^2(2\ell+1)\calC_\ell.
		\end{align}
		Let $U^H_N(t)$, $N\geq1$, be defined as
		\[
	U^H_N(t):=\sum_{\ell=0}^{N}T_{\ell}(t).
		\]
		It is easy to see, using \eqref{var-T}, that
		\begin{align}\label{var-Un}
			Var[U^H_N(t)]=\sum_{\ell=0}^{N}Var[T_{\ell}(t)]=\sum_{\ell=0}^{N}(E_\alpha(-\lambda_{\ell} t^\alpha))^2(2\ell+1)\calC_\ell,
		\end{align}
		and, using \eqref{Vpf}, 
		\begin{align}\label{Lim}
			\lim_{N\to\infty}Var[U^H_N(t)]=\sum_{\ell=0}^{\infty}(E_\alpha(-\lambda_{\ell} t^\alpha))^2(2\ell+1)\calC_\ell<\infty.
		\end{align}
		The characteristic function $\varphi_{U^H_N(t)}(r)$ of $U^H_N(t)$ can be written, using \eqref{var-T} and \eqref{var-Un}, as
		\begin{align*}
			\varphi_{U^H_N(t)}(r)=\bE\Big[e^{\mi r U^H_N(t)}\Big]&= \prod_{\ell=0}^{N}\bE\Big[e^{\mi r T_{\ell}(t)}\Big]
			=\prod_{\ell=0}^{N}  \varphi_{T_{\ell}(t)}(r)\\
			&=e^{-\frac{1}{2}r^2 \sum_{\ell=0}^{N}(E_\alpha(-\lambda_{\ell} t^\alpha))^2(2\ell+1)\calC_\ell},   
		\end{align*}
		since $T_{\ell},\ \ell\geq0$, are independent Gaussian random variables.
		
		\indent
		Hence,
		\[
		\lim_{N\to\infty} \varphi_{U^H_N(t)}(r)= e^{-\frac{1}{2}r^2 \sum_{\ell=0}^{\infty}(E_\alpha(-\lambda_{\ell} t^\alpha))^2(2\ell+1)\calC_\ell},
		\]
		which is the characteristic function of some Gaussian random variable with mean zero and variance $\sum_{\ell=0}^{\infty}(E_\alpha(-\lambda_{\ell} t^\alpha))^2(2\ell+1)\calC_\ell<\infty$ (by \eqref{Lim}). Then by the continuity theorem (see \cite{Durrett}, Theorem 3.3.6),
		we conclude that the field $U^{H}$ is Gaussian, thus completing the proof. 
	\end{proof}

	\begin{rem} \label{rem4}
		Let $\{\calC_{\ell}:\ell\in\N_{0}\}$ be the angular power spectrum of $\xi$. Then the Fourier coefficients $\widehat{U^H}_{\ell,m}$ of $U^{H}$ can be written, using \eqref{xile} and Proposition {\rm\ref{Coro}}, as
		\begin{align}\label{UH-complexFourier}
			\widehat{U^H}_{\ell,m}(t)&:= 
			\begin{cases} 
				\sqrt{\calC_{\ell}}E_\alpha(-\lambda_{\ell} t^\alpha)Z_{\ell,0}^{(1)}, & m=0,	\\
				\sqrt{\frac{\calC_{\ell}}{2}}E_\alpha(-\lambda_{\ell} t^\alpha)\Big(Z_{\ell,m}^{(1)}
				- \mi Z_{\ell,m}^{(2)}\Big), & m=1,\ldots,\ell, \\
				(-1)^m\sqrt{\frac{\calC_{\ell}}{2}}E_\alpha(-\lambda_{\ell} t^\alpha)\Big(Z_{\ell,|m|}^{(1)}
				+ \mi Z_{\ell,|m|}^{(2)}\Big), & m=-\ell,\ldots,-1,
			\end{cases}
		\end{align}
		where $Z_{\ell,m_1}^{(1)},Z_{\ell,m_2}^{(2)}\in\mathcal{Z}_{\ell},\ \ell\in\N_{0}, m_1=0,\dots,\ell$, $m_2=1,2,\dots,\ell$ and $\mathcal{Z}_{\ell}$ is defined by \eqref{Z-iid}. 
		
		\indent
		Moreover, the solution $U^{H}(t)$, $t\in(0,\infty)$, can be represented, using \eqref{UH-complexFourier}, as
		\begin{align}\label{New-HomSol}
			U^{H}(t)&= \sum_{\ell=0}^{\infty}\sum_{m=-\ell}^\ell \widehat{U^H}_{\ell,m}(t) Y_{\ell,m}\notag\\
			&=\sum_{\ell=0}^{\infty}\Bigg(E_\alpha(-\lambda_{\ell} t^\alpha)\Big( \sqrt{\calC_{\ell}}Z_{\ell,0}^{(1)}Y_{\ell,0}\notag\\
			&+\sqrt{2\calC_{\ell}}\sum_{m=1}^{\ell}\Bigl(Z_{\ell,m}^{(1)}\Re Y_{\ell,m}(\bsx)+Z_{\ell,m}^{(2)}\Im Y_{\ell,m}\Big)\Big)\Bigg),
		\end{align}
where the above expansion is convergent in $L_2(\Omega\times\bS^2)$.
	\end{rem}

\section{Solution of the inhomogeneous problem}\label{Sec5}
Consider a random field $U^{I}(t)\in L_2(\Omega\times\bS^2), t\in (0,\infty)$.
We say that $U^I$ is a solution to  \eqref{pde} if the equality in \eqref{UIintform} holds in the sense of $L_2(\Omega\times\bS^2)$, that is
\[
\sup_{t \geq \tau} \bE \left\| U^{I}(t)-  \frac{1}{\Gamma(\alpha)}\int_\tau^t\frac{\Delta_{\bS^2} U^{I}(s)}{(t-s)^{1-\alpha}}ds- W_{\tau}(t)\right\|^2_{L_2(\bS^2)} = 0.
\]
This section derives the solution 
$U^{I}(t)\in L_2(\Omega\times\bS^2)$, $t\in(0,\infty)$, to the inhomogeneous equation \eqref{pde}, under the condition $U^{I}(t)=0$ for $t\in(0,\tau]$. 
	
Note that since $U^{I}\in L_2(\Omega\times\bS^2)$, we can write 
	\begin{align*}
		U^{I}(t) = \sum_{\ell=0}^\infty \sum_{m=-\ell}^{\ell} \widehat{U^I}_{\ell,m}(t) Y_{\ell,m}, \quad t\in(0,\infty),
	\end{align*}
	where the random coefficients $\widehat{U^I}_{\ell,m}$, $\ell\in\N_{0}$, $m=-\ell,\dots,\ell$, are given by
	\begin{align*}
		\widehat{U^I}_{\ell,m}(t)=\int_{\bS^2}U^{I}(\bsx,t)\overline{Y_{\ell,m}(\bsx)}\mu(d\bsx),\quad t\in(0,\infty).
	\end{align*}
	By multiplying both sides of equation \eqref{pde} by $\overline{Y_{\ell,m}}$ and integrating over $\bS^2$, we obtain a set of ordinary differential equations 
	\begin{equation}\label{Eqv1}
		d \widehat{U^I}_{\ell,m} + \lambda_{\ell} D_t^{1-\alpha} \widehat{U^I}_{\ell,m} dt = d \widehat{W}_{\ell,m,\tau}, \quad \widehat{U^I}_{\ell,m}(0) =0,
	\end{equation}
	where $\widehat{W}_{\ell,m,\tau}$ are the Fourier coefficients of $W_\tau$.
	
	\indent
	If $\widehat{W}_{\ell,m,\tau}$ were deterministic differentiable functions in time, then we could write \eqref{Eqv1} as
	\begin{equation}\label{Nn}
		\frac{d}{dt} \widehat{U^I}_{\ell,m} + \lambda_{\ell}D_t^{1-\alpha} \widehat{U^I}_{\ell,m}  = \frac{d}{dt} \widehat{W}_{\ell,m,\tau}, \quad \widehat{U^I}_{\ell,m}(0) = 0.
	\end{equation}
	
	\noindent
	By taking the Laplace transform of \eqref{Nn} (see equation \eqref{Laplace}), we get
	\begin{equation}\label{lap}
		z \widetilde{U^I}_{\ell,m}(z) - \widehat{U^I}_{\ell,m}(0) + \lambda_{\ell} z^{1-\alpha} \widetilde{U^I}_{\ell,m} (z) = e^{-z\tau} \widetilde{W}_{\ell,m,0}(z),
	\end{equation}
	where $\widetilde{U^I}_{\ell,m}(z)$ is the Laplace transform of $\widehat{U^I}_{\ell,m}(t)$ and $e^{-z\tau} \widetilde{W}_{\ell,m,0}(z)$ is the Laplace transform of $\frac{d}{dt}\widehat{W}_{\ell,m,\tau}(t)$.

	\indent
	Now solving equation \eqref{lap} for $\widetilde{U^I}_{\ell,m} (z)$, we obtain
	\begin{equation}\label{Lap2}
		\widetilde{U^I}_{\ell,m} (z) = e^{-z\tau}\frac{ \widetilde{W}_{\ell,m,0}(z) }{ z+ \lambda_{\ell} z^{1-\alpha}}.
	\end{equation}

	\noindent
	Using the convolution theorem (see \cite{Dyke}, Theorem 3.2), we could write   
	\[
	\calL^{-1}\Big\{ \frac{ \widetilde{W}_{\ell,m,0}(z) }{ z+ \lambda_{\ell} z^{1-\alpha}}\Big\}= 	\int_{0}^{t} E_{\alpha}( -\lambda_{\ell}(t-s)^\alpha) \frac{d}{ds} \widehat{W}_{\ell,m,0}(s)ds,
	\]
	where $\calL^{-1}(\cdot)$ denotes the inverse Laplace transform.
	
	\indent
	Now by taking the inverse Laplace transform in \eqref{Lap2} with the help of the second shift theorem (see \cite{Dyke}, Theorem 2.4) we obtain
	\begin{align*}
		\widehat{U^I}_{\ell,m}(t) &= \mathcal{H}(t-\tau)\int_{0}^{t-\tau}  E_{\alpha}( -\lambda_{\ell}(t-\tau-s)^\alpha) \frac{d}{ds} \widehat{W}_{\ell,m,0}(s) ds,
	\end{align*}
	where $\mathcal{H}(\cdot)$ is the Heaviside unit step function (i.e, $\mathcal{H}(t)=1$, for $t\geq0$, and zero for $t<0$).
	
	\indent	
	Since $\widehat{W}_{\ell,m,0}(\cdot)$ are $1$-dimensional, real-valued Brownian motions which are continuous but nowhere differentiable, we have
	to write instead 
	\begin{equation*}
		\widehat{U^I}_{\ell,m}(t) = \mathcal{H}(t-\tau)\int_{0}^{t-\tau} E_{\alpha}( -\lambda_{\ell} (t-\tau-s)^\alpha) d\widehat{W}_{\ell,m,0}(s),
	\end{equation*}
	where the integral is in the It\^{o} sense. 
	
	\indent
Thus the solution of equation \eqref{pde} takes the form, for $t>\tau$,
\begin{align}\label{NewSol2}
		U^{I}(t) = \sum_{\ell=0}^\infty \sum_{m=-\ell}^{\ell}\int_{0}^{t-\tau} E_{\alpha}( -\lambda_{\ell} (t-\tau-s)^\alpha) d \widehat{W}_{\ell,m,0}(s)Y_{\ell,m},
\end{align}
and $U^{I}(t)=0$ for $t\leq\tau$.

By the condition \eqref{eq:condAell} and Remark \ref{rem1}, for every $\ell\in\N_{0}$, we have $\sigma_{\ell,t,\alpha}^2\leq t$, which  guarantees the convergence of \eqref{NewSol2} in the $L_2(\Omega\times\bS^2)$ space.

For $L\in\N$, let 
\begin{equation}\label{VLIint}
V_L^{I}(t):=D_{t}^{1-\alpha}\Delta_{\bS^2} U_L^I(t)
=  \frac{1}{\Gamma(\alpha)} \frac{d}{dt}
\int_{\tau}^t \frac{\Delta_{\bS^2} U_L^I(s)}
{(t-s)^{1-\alpha}}ds.
\end{equation}

\begin{rem}
By Lemma \ref{frD}, we can express $V_L^{I}$ in the form
\begin{align}\label{VI}
V_L^{I}(t)&= \sum_{\ell=0}^L \sum_{m=-\ell}^{\ell} -\lambda_{\ell}D_{t}^{1-\alpha}\int_{\tau}^{t} E_{\alpha}( -\lambda_{\ell} (t-\tau-r)^\alpha) d \widehat{W}_{\ell,m,\tau}(r)Y_{\ell,m}\nonumber\\
&= \sum_{\ell=0}^L \sum_{m=-\ell}^{\ell} -\lambda_{\ell}\int_{\tau}^{t} D_{r}^{1-\alpha}E_{\alpha}( -\lambda_{\ell} r^\alpha) d \widehat{W}_{\ell,m,\tau}(r)Y_{\ell,m}\nonumber\\
&= \sum_{\ell=0}^L \sum_{m=-\ell}^{\ell} \int_{\tau}^{t} \Big(\dfrac{-\lambda_{\ell}}{\Gamma(\alpha)}\dfrac{d}{dr}\int_{0}^{r}\dfrac{E_{\alpha}( -\lambda_{\ell} s^\alpha)}{(r-s)^{1-\alpha}}ds\Big) d \widehat{W}_{\ell,m,\tau}(r)Y_{\ell,m}\nonumber\\
&= \sum_{\ell=0}^L \sum_{m=-\ell}^{\ell} \int_{\tau}^{t} \dfrac{d}{dr}\Big(E_{\alpha}( -\lambda_{\ell} r^\alpha)-1\Big)d \widehat{W}_{\ell,m,\tau}(r)Y_{\ell,m}\nonumber\\
&=\sum_{\ell=0}^L \sum_{m=-\ell}^{\ell}\int_{\tau}^{t} (-\lambda_{\ell})r^{\alpha-1} E_{\alpha,\alpha}( -\lambda_{\ell} r^\alpha) d \widehat{W}_{\ell,m,\tau}(r)Y_{\ell,m},
\end{align}
where the third step used Lemma 3.10 in \cite{Gorenflo2014}.
\end{rem}

\begin{lem}\label{cachu-UI}
Let $W_{\tau}$ be an $L_2(\bS^2)$-valued time-delayed Brownian motion. Let $\{\calA_{\ell}:\ell\in\N_{0}\}$, the angular power spectrum of $W_\tau$, satisfy \eqref{eq:condAell}.
Then, as $L\to\infty$, $V_{L}^{I}$ is convergent 
 to
\begin{equation}\label{eq:defVI}
		V^{I}(t):= \sum_{\ell=0}^{\infty}\sum_{m=-\ell}^\ell\int_{\tau}^{t} (-\lambda_{\ell})s^{\alpha-1} E_{\alpha.\alpha}( -\lambda_{\ell} s^\alpha) d \widehat{W}_{\ell,m,\tau}(s)Y_{\ell,m},\quad t\geq \tau,
\end{equation}
in the following sense
 \[
	\sup_{t \geq \tau} \bE \left\| V_L^I(t) - V^I(t) \right\|^2_{L_2(\bS^2)} 
	\to 0,\quad  \text{ as } L \to \infty.
\]
\end{lem}
\begin{proof}
For $t\geq \tau$, $L<M$, by Parseval's formula and It\^o's isometry, we have 
\begin{align}
	\bE \|V^I_L(t) - V^I_M(t) \|^2_{L_2(\bS^2)}
		&= \bE \left[\sum_{\ell=L+1}^M
		\sum_{m=-\ell}^{\ell} 		\lambda^2_{\ell} \int_{\tau}^{t} (-\lambda_{\ell})s^{\alpha-1} E_{\alpha,\alpha}( -\lambda_{\ell} s^\alpha) d \widehat{W}_{\ell,m,\tau}(s)Y_{\ell,m}
		 \right]\nonumber\\
		&=\sum_{\ell=L+1}^M
		\lambda^2_{\ell}(2\ell+1) \calA_{\ell} \int_{\tau}^{t} s^{2\alpha-2}
		(E_{\alpha,\alpha}(-\lambda_{\ell} s^\alpha))^2 ds
		\text{   by using \eqref{eq:defCell} }\nonumber \\
		&\le C^2 \int_{\tau}^{t} s^{-2}ds
		\sum_{\ell=L+1}^M (2\ell+1) \calA_{\ell}\nonumber \\
  &\le C^2\tau^{-1}\sum_{\ell=L+1}^M (2\ell+1) \calA_{\ell} \label{eq:cauchyCondI},
\end{align}
where the third step used \eqref{EMitagg}.

For a given $\epsilon>0$, using the condition \eqref{eq:condAell}, there is 
	$L_0$ independent of $t$ 
such that for $L,M  \ge L_0$, 
	the RHS of \eqref{eq:cauchyCondI} is smaller than $\epsilon$. So
	$\{V_L^I\}$ is a Cauchy sequence in $L_2(\Omega \times \bS^2)$ and hence it
	is convergent. We define the limit of $V_L^I$ as $L \to \infty$ as in
	\eqref{eq:defVI}. 
Note that 
by the condition \eqref{eq:condAell} we obtain
 \begin{align}\label{VIconv}
   \bE \Big[\|V^I(t)\|^2_{L_2(\bS^2)}\Big]\le C^2 \tau^{-1}
\sum_{\ell=0}^{\infty} (2\ell+1) \calA_{\ell}<\infty,  
 \end{align}
which guarantees that $V^I$ is well-defined in the $L_2(\Omega\times\bS^2)$ sense.
\end{proof}

\begin{rem}\label{RemVI}
Let $T>\tau>0$, then by \eqref{eq:cauchyCondI} we obtain
\[
\int_{\tau}^{T} \bE \Big[\|V_L^I(s)\|^2_{L_2(\bS^2)}\Big]ds\le C^2\tau^{-1} T \sum_{\ell=0}^{\infty} (2\ell+1) \calA_{\ell}<\infty,
\]
which implies, by Lemma \ref{cachu-UI} and the dominated convergence theorem, that $\int_{\tau}^T V_L^{I}(s)ds$ is convergent as $L\to\infty$, 
to $\int_{\tau}^t V^{I}(s)ds$, in the $L_{2}(\Omega\times\bS^2)$ sense.  
\end{rem}
For $L\in\N$, $t>\tau$, we define
\begin{align}\label{VLI}
U_L^I(t):&= \int_{\tau}^t V_L^{I}(s)ds+ 
W^{(L)}_\tau(t),
\end{align}
where 
\begin{align}\label{truncW}
W^{(L)}_\tau(t):=\sum_{\ell=0}^L \sum_{m=-\ell}^{\ell}\widehat{W}_{\ell,m,\tau}(t)Y_{\ell,m},
\end{align}
and $V_L^{I}$ is defined by \eqref{VI}. 

\begin{lem}\label{LemdUI}
Let $W_{\tau}$ be an $L_2(\bS^2)$-valued time-delayed Brownian motion. Let $\{\calA_{\ell}:\ell\in\N_{0}\}$, the angular power spectrum of $W_\tau$, satisfy the condition \eqref{eq:condAell}. 
	For $t\in[\tau,T]$, then $U_{L}^{I}$ is convergent 
 as $L\to\infty$, to $U^I(t):=\int_{\tau}^{t}V^{I}(s)ds+W_{\tau}(t)$,
	in the following sense
	\[
	\sup_{t \in[\tau,T]} \bE\Big[ \left\| U_L^I(t) - U^I(t) \right\|^2_{L_2(\bS^2)} \Big]
	\to 0,\quad  \text{ as } L \to \infty.
	\]
\end{lem}
\begin{proof}
    For $t\geq \tau$, $L<M$, we have, using \eqref{VLI},
\begin{align}\label{pfduI}
		\bE\Big[ \|U^I_L(t) - U^I_M(t) \|^2_{L_2(\bS^2)}\Big]
		&= \bE \Big[\Big\Vert\int_{\tau}^{t}V_L^{I}(s)ds-\int_{\tau}^{t}V^I_M(s)ds\nonumber\\
  &+ \sum_{\ell=L+1}^M \sum_{m=-\ell}^{\ell} \widehat{W}_{\ell,m,\tau}(s)Y_{\ell,m} \Big\Vert^2_{L_2(\bS^2)}\Big] \nonumber\\
&\le \bE \Big[\Big\Vert\int_{\tau}^{t}V_L^{I}(s)ds-\int_{\tau}^{t}V^I_M(s)ds\Big\Vert^2_{L_2(\bS^2)}\Big]\nonumber\\
&+(t-\tau)
\sum_{\ell=L+1}^M (2\ell+1) \calA_{\ell},
\end{align}
where the second step used Parseval's formula.

By Remark \ref{RemVI}, there is $L_0$ independent of $t$ 
such that for $L,M  \ge L_0$, 
the term 
\[
\bE \Big[\Big\Vert\int_{\tau}^{t}V_L^{I}(s)ds-\int_{\tau}^{t}V^I_M(s)ds\Big\Vert^2_{L_2(\bS^2)}\Big]\leq\epsilon/2,
\]
for $\epsilon>0$. Similarly, using the condition \eqref{eq:condAell}, the second term of \eqref{pfduI} can be smaller than $\epsilon/2$. So
	$\{U_L^I\}$ is a Cauchy sequence in $L_2(\Omega \times \bS^2)$ and hence it
	is convergent.
 Thus the proof is complete.
\end{proof}

\begin{prop}
Let the angular power spectrum $\{\calA_{\ell}:\ell\in\N_{0}\}$ of $W_\tau$ satisfy assumption \eqref{eq:condAell}.
Then the random field $U^{I}(t),\ t>\tau$, defined by \eqref{NewSol2}, satisfies \eqref{pde} in the $L_2(\Omega\times\bS^2)$ sense.
In particular, for $t\in[\tau,T]$, there holds
\begin{align}\label{UIVI}
    \sup_{t \in [\tau,T]} \bE \left\| U^{I}(t)-  \frac{1}{\Gamma(\alpha)}\int_\tau^t\frac{\Delta_{\bS^2} U^{I}(s)}{(t-s)^{1-\alpha}}ds- W_{\tau}(t)\right\|^2_{L_2(\bS^2)} = 0.
\end{align}
\end{prop}

\begin{proof}
By Lemma \ref{LemdUI} there holds
\[
\sup_{t \in[\tau,T]} \bE \left\| U^I(t)-\int_{\tau}^{t}V^{I}(s)ds-W_{\tau}(t)\right\|^2_{L_2(\bS^2)}=0.
\]
Also, using \eqref{VLIint} we obtain
\[
\int_{\tau}^{t}V^{I}(s)ds=\frac{1}{\Gamma(\alpha)}\int_\tau^t\frac{\Delta_{\bS^2} U^{I}(s)}{(t-s)^{1-\alpha}}ds,
\]
which is convergent in $L_2(\Omega\times\bS^2)$ due to \eqref{VIconv}.
Thus the proof is complete. 
\end{proof}

The following result shows that the inhomogeneous solution $U^{I}(t)$, $t\in(\tau,\infty)$, to the equation (\ref{pde}) is a centered, 2-weakly isotropic Gaussian random field.
	Here we assume that the angular power spectrum  
	$\{\calA_{\ell}:\ell\in\N_{0}\}$ of $W_{\tau}$, decays algebraically with order $\kappa_2>2$, i.e., there exist constants $\widetilde{K},\widetilde{A}>0$ such that
	\begin{align}\label{New-Al}
		\calA_{\ell}\leq\begin{cases} 
			\widetilde{K}, & \ell=0, \\
			\widetilde{A}\ell^{-\kappa_2}, & \ell\geq1,\ \kappa_2>2.
		\end{cases}
	\end{align}
	Let $\widetilde{A}_{\kappa_2}$ and $\gamma_\alpha(\kappa_2)$ be defined as
	\begin{align}\label{Akappa}
		\widetilde{A}_{\kappa_2}:=\bigg(\widetilde{A}\Big(\dfrac{2}{\kappa_2-2}+\dfrac{1}{\kappa_2-1}\Big)\bigg)^{1/2}
	\end{align}
	and
	\begin{align}\label{C1}
		\gamma_\alpha(\kappa_2):=
		\begin{cases} 
			\kappa_2+2, & \alpha\in(0,\frac{1}{2}), \\
			\kappa_2+\frac{2}{\alpha}-2, & \alpha\in(\frac{1}{2},1],\\
			\kappa_2, & \alpha=\frac{1}{2}.
		\end{cases}
	\end{align}
	

	\begin{prop}\label{covUI}
		Let the field $U^{I}$ be defined as in {\rm(\ref{NewSol2})}, 
		with the condition $U^{I}(t)= 0$, $t\in(0,\tau]$. Let $\widehat{W}_{\ell,m,0}$ be the Fourier coefficients of $W_0$ with variances $\{\calA_{\ell}:\ell\in\N_{0}\}$ given in \eqref{New-Al}. For a fixed $t\in(\tau,\infty)$, $U^{I}(t)$ is a centred, $2$-weakly isotropic Gaussian random field on $\bS^2$, and its random coefficients 
		\begin{align}\label{zetalm}
			\widehat{U^{I}}_{\ell,m}(t)=\int_{0}^{t-\tau}E_\alpha(-\lambda_{\ell} (t-\tau-s)^\alpha)d \widehat{W}_{\ell,m,0}(s),
		\end{align}
		satisfy for $\ell,\ell^\prime\in\N_{0}$, $m=-\ell,\dots,\ell$ and $m^\prime=-\ell^\prime,\dots,\ell^\prime$, 
		\begin{align}\label{var-zet}
			\bE \left[ \widehat{U^{I}}_{\ell,m}(t)\overline{ \widehat{U^{I}}_{\ell^\prime,m^\prime}(t)}\right]=\calA_{\ell}\sigma_{\ell,t-\tau,\alpha}^2\delta_{\ell\ell^\prime}\delta_{mm^\prime},
		\end{align}
		where $\lambda_{\ell}$ is given in \eqref{lam}, $\delta_{\ell\ell^\prime}$ is the Kronecker delta function, and $\sigma_{\ell,t,\alpha}^2$ is given by \eqref{var}.
	\end{prop}
\begin{proof}
Consider the field $U^{I}$ defined by \eqref{NewSol2}. Since $W_0$ is centred we have $\bE[U^{I}(t)]=0$. Let $t\in(\tau,\infty)$ and $\bsx,\bsy\in\bS^2$, then by using \eqref{NewSol2} we write
\begin{align*}
			&\bE \left[U^{I}(\bsx,t)U^{I}(\bsy,t)\right]=\sum_{\ell=0}^\infty\sum_{\ell^\prime=0}^\infty \sum_{m=-\ell}^{\ell}
			\sum_{m\prime=-\ell^\prime}^{\ell^\prime}Y_{\ell,m}(\bsx)\overline{Y_{\ell^\prime,m^\prime}(\bsy)}\\ 
			&\times\bE\left[\left(\int_{0}^{t-\tau} E_\alpha(-\lambda_{\ell} (t-\tau-s)^\alpha)d\widehat{W}_{\ell,m,0}(s)\right)\left(\overline{\int_{0}^{t-\tau}E_\alpha(-\lambda_{\ell^\prime} (t-\tau-v)^\alpha)d\widehat{W}_{\ell^\prime,m^{\prime},0}(v)}\right)\right].
		\end{align*}
		Using It\^{o}’s isometry (see \cite{Kuo2006}) and the addition theorem (see \eqref{addition}) we can write
		\begin{align*}
			\bE\left[U^{I}(\bsx,t)U^{I}(\bsy,t)\right]&=\sum_{\ell=0}^\infty\calA_{\ell}\int_{0}^{t-\tau} (E_\alpha(-\lambda_{\ell} (t-\tau-s)^\alpha))^2ds\sum_{m=-\ell}^{\ell} Y_{\ell,m}(\bsx)\overline{Y_{\ell,m}(\bsy)}\\
			&=\sum_{\ell=0}^\infty(2\ell+1)\calA_\ell P_{\ell}(\bsx\cdot\bsy)\int_{0}^{t-\tau} (E_\alpha(-\lambda_{\ell} (t-\tau-s)^\alpha))^2ds\\
			&= \sum_{\ell=0}^\infty(2\ell+1)\calA_\ell\sigma_{\ell,t-\tau,\alpha}^2 P_{\ell}(\bsx\cdot\bsy),
		\end{align*}
		where $P_{\ell}(\cdot)$, $\ell\in\N_{0}$, is the Legendre polynomial of degree $\ell$. As $P_{\ell}(\bsx\cdot\bsy)$ depends only on the inner product of $\bsx$ and $\bsy$, we conclude that the covariance function $\bE\left[U^{I}(\bsx,t)U^{I}(\bsy,t)\right]$ is rotationally invariant. Note that for $\ell,\ell^\prime\in\N_{0}$, $m=-\ell,\dots,\ell$ and $m^\prime=-\ell^\prime,\dots,\ell^\prime$, by the 2-weak isotropy of $W_0$ and It\^{o}’s isometry, 
		\begin{align*}
			\bE \left[ \widehat{U^{I}}_{\ell,m}(t)\overline{ \widehat{U^{I}}_{\ell^\prime,m^\prime}(t)}\right]=\calA_{\ell}\sigma_{\ell,t-\tau,\alpha}^2\delta_{\ell\ell^\prime}\delta_{mm^\prime}.
		\end{align*}
		To prove that $U^{I}(\bsx,t)$ is Gaussian, we note that its variance can be written, using \eqref{New-Al}, as
		\begin{align}\label{Ga-InHom}
			Var\left[U^{I}(\bsx,t)\right]&=\sum_{\ell=0}^\infty(2\ell+1)\calA_\ell \sigma_{\ell,t-\tau,\alpha}^2\notag\\
			&= \widetilde{K} \sigma_{0,t-\tau,\alpha}^2 +\sum_{\ell=1}^\infty(2\ell+1)\calA_\ell \sigma_{\ell,t-\tau,\alpha}^2\notag\\
			&	\leq (t-\tau)\big(\widetilde{K}+ \sum_{\ell=1}^\infty(2\ell+1)\calA_\ell\big)\notag\\
			&\leq (t-\tau)\big( \widetilde{K}+\widetilde{A}_{\kappa_2}^2\big)<\infty,
		\end{align}
		where $\widetilde{A}_{\kappa_2}$ is given by \eqref{Akappa} and the last step uses \eqref{var} and steps similar to \eqref{Vpf}.
		
		\indent	
Note that $R_{\ell}(t):=\sum_{m=-\ell}^{\ell} \widehat{U^{I}}_{\ell,m}(t)Y_{\ell,m}$ is a centred, Gaussian random variable with mean zero and variance given, using \eqref{var-zet}, by  $(2\ell+1)\calA_{\ell}\sigma_{\ell,t-\tau,\alpha}^2$.
		Define $U^{I}_{N}(t):=\sum_{\ell=0}^{N}R_{\ell}(t)$, $N\geq1$, then by \eqref{Ga-InHom} and similar steps as in \eqref{var-Un} and \eqref{Lim}, with direct application of the continuity theorem (see \cite{Durrett}, Theorem 3.3.6),
		we conclude that the field $U^{I}$ is Gaussian, thus completing the proof. 
	\end{proof}
	
	\begin{rem}
		Let $\{\calA_{\ell}:\ell\in\N_{0}\}$ be the angular power spectrum of $W_{0}$. Then for $t>\tau$ the Fourier coefficients $ \widehat{U^I}_{\ell,m}(t)$ of the field $U^{I}(t)$ can be written, using \eqref{WincomplexFourier} and \eqref{NewSol2}, as
		\begin{align}\label{complexFourier}
			\widehat{U^I}_{\ell,m}(t)&:=
			\begin{cases} 
				\sqrt{\calA_{\ell}}\mathcal{I}_{\ell,0,\alpha}^{(1)}(t-\tau), & m=0,	\\
				\sqrt{\calA_{\ell}/2}
				\Big(\mathcal{I}_{\ell,m,\alpha}^{(1)}(t-\tau)
				- \mi \mathcal{I}_{\ell,m,\alpha}^{(2)}(t-\tau)\Big), & m=1,\ldots,\ell, \\
				(-1)^m\sqrt{\calA_{\ell}/2}
				\Big(\mathcal{I}_{\ell,|m|,\alpha}^{(1)}(t-\tau)
				+ \mi \mathcal{I}_{\ell,|m|,\alpha}^{(2)}(t-\tau)\Big), & m=-\ell,\ldots,-1,
			\end{cases}
		\end{align}
		where $\mathcal{I}_{\ell,m,\alpha}^{(j)}(\cdot)$, $j=1,2$, are given in \eqref{Int}. 
		
		\indent
		Moreover, 
		for $t>\tau$, $U^{I}(t)$ can be expressed, using \eqref{complexFourier}, as
		\begin{align}\label{Sol2}
			U^{I}(t)&=
			\sum_{\ell=0}^\infty
			\sum_{m=-\ell}^\ell \widehat{U^I}_{\ell,m}(t) Y_{\ell,m}\notag\\
			&=\sum_{\ell=0}^{\infty}\Big(\sqrt{\calA_{\ell}}\mathcal{I}_{\ell,0,\alpha}^{(1)}(t-\tau)Y_{\ell,0}\notag\\
			&+\sqrt{2\calA_{\ell}}\sum_{m=1}^{\ell}\Bigl(\mathcal{I}_{\ell,m,\alpha}^{(1)}(t-\tau)\Re Y_{\ell,m}+\mathcal{I}_{\ell,m,\alpha}^{(2)}(t-\tau)\Im Y_{\ell,m}\Big)
			\Big),
		\end{align}
where the above expansion is convergent in $L_2(\Omega\times\bS^2)$.
\end{rem}

	\section{Solution of the time-fractional diffusion equation}\label{FullSec}
	This section demonstrates the solution $U(t)$, $t\in(0,\infty)$, to the time-fractional diffusion equation given in \eqref{System}. 
	
	Let the conditions of Propositions {\rm\ref{Theo1}} 
	be satisfied. 
	Then, the solution $U(t)$, $t\in(0,\infty)$, of equation \eqref{System} is given by
	\begin{align}\label{Exact}
		U(t):=U^{H}(t)+U^{I}(t),\quad t\in(0,\infty),
	\end{align}
	where $U^{H}$ is given by \eqref{New-HomSol}
	and $U^{I}(t)=0$ for $t\leq\tau$, while for $t>\tau$, $U^{I}(t)$ is given by \eqref{Sol2}.

	The following result shows that the solution $U(t)$, $t\in(0,\infty)$, to the equation (\ref{System}) is a centred, 2-weakly isotropic Gaussian random field.
	\begin{prop}\label{Fulliso}
Let the field $U$, given by {\rm(\ref{Exact})}, be the solution to the equation {\rm(\ref{System})}. Then, for a fixed $t\in(0,\infty)$, $U(t)$ is a centred, $2$-weakly isotropic Gaussian random field with random coefficients
		\begin{align}\label{gamlm}
			\widehat{U}_{\ell,m}(t):=
			\begin{cases}
				\widehat{U^{H}}_{\ell,m}(t),&t\leq\tau,\\
				\widehat{U^{H}}_{\ell,m}(t)+ \widehat{U^{I}}_{\ell,m}(t),&t>\tau,
			\end{cases}
		\end{align}
		where $ \widehat{U^{H}}_{\ell,m}(t)$ and $ \widehat{U^{I}}_{\ell,m}(t)$ are given by \eqref{xile} and \eqref{zetalm} respectively. 
		
		\indent
		Moreover, for $\ell,\ell^\prime\in\N_{0}$, $m=-\ell,\dots,\ell$ and $m^\prime=-\ell^\prime,\dots,\ell^\prime$, there holds
		\begin{align*}
			\bE \left[ \widehat{U}_{\ell,m}(t)\overline{ \widehat{U}_{\ell^\prime,m^\prime}(t)}\right]=\delta_{\ell\ell^\prime}\delta_{mm^\prime}
			\begin{cases}
				\calC_\ell (E_\alpha(-\lambda_{\ell} t^\alpha))^2,&t\leq\tau,\\
				\calC_\ell (E_\alpha(-\lambda_{\ell} t^\alpha))^2 +\calA_\ell\sigma_{\ell,t-\tau,\alpha}^2,&t>\tau.
			\end{cases}
		\end{align*}
	\end{prop}
	\begin{proof}
		For $t\leq\tau$ the result follows by Proposition \ref{Prop4}. Let $t>\tau$, then by Propositions \ref{Prop4}, \ref{covUI}, and the
		uncorrelatedness of $U^H$ and $U^I$, we obtain
		\begin{align*}
			\bE\left[U(\bsx,t)U(\bsy,t)\right]&=\bE\left[U^{H}(\bsx,t)U^{H}(\bsy,t)\right]+\bE\left[U^{I}(\bsx,t)U^{I}(\bsy,t)\right]\\
			&= \sum_{\ell=0}^\infty(2\ell+1)\left(\calC_\ell (E_\alpha(-\lambda_{\ell} t^\alpha))^2 +\calA_\ell\sigma_{\ell,t-\tau,\alpha}^2\right)P_{\ell}(\bsx\cdot\bsy).
		\end{align*}
		As $P_{\ell}(\bsx\cdot\bsy)$ depends only on the inner product of $\bsx$ and $\bsy$, we conclude that the covariance function $\bE\left[U^{I}(\bsx,t)U^{I}(\bsy,t)\right]$ is rotationally invariant. Since $t>\tau$,  we note that for $\ell,\ell^\prime\in\N_{0}$, $m=-\ell,\dots,\ell$ and $m^\prime=-\ell^\prime,\dots,\ell^\prime$, by Propositions \ref{Prop4} and \ref{covUI},
		\begin{align*}
			\bE \left[ \widehat{U}_{\ell,m}(t)\overline{ \widehat{U}_{\ell^\prime,m^\prime}(t)}\right]=\big((E_\alpha(-\lambda_{\ell} t^\alpha))^2 \calC_{\ell}+\calA_{\ell}\sigma_{\ell,t-\tau,\alpha}^2\big)\delta_{\ell\ell^\prime}\delta_{mm^\prime}.
		\end{align*}
		Now using the results of Propositions \ref{covUI} and \ref{Prop4}, we conclude that the field $U$ is Gaussian, thus completing the proof.
	\end{proof}
	
	\begin{rem}\label{rem5}
		The solution $U$ given by \eqref{Exact}, can be expressed, using \eqref{New-HomSol} and \eqref{Sol2}, as 
		\begin{align*}
			U(t)= \sum_{\ell=0}^{\infty}\big(U^{H}_{\ell}(t)+U^{I}_{\ell}(t)\big),\quad t\in(0,\infty),
		\end{align*}
which is convergent in $L_2(\Omega\times\bS^2)$, where $U^{H}_{\ell}(t)$ is defined as 
		\begin{align}\label{tru-HomSol}
			U^{H}_{\ell}(t)&=E_\alpha(-\lambda_{\ell} t^\alpha)\Big( \sqrt{\calC_{\ell}}Z_{\ell,0}^{(1)}Y_{\ell,0}\notag\\
			&+\sqrt{2\calC_{\ell}}\sum_{m=1}^{\ell}\Bigl(Z_{\ell,m}^{(1)}\Re Y_{\ell,m}+Z_{\ell,m}^{(2)}\Im Y_{\ell,m}\Big)\Big),
		\end{align}
		and $U^{I}_{\ell}(t)=0$ for $t\leq\tau$, while for $t>\tau$,
		\begin{align}\label{tru-inho}
			U^{I}_{\ell}(t)&= \sqrt{\calA_{\ell}}\mathcal{I}_{\ell,0,\alpha}^{(1)}(t-\tau)Y_{\ell,0}\notag\\
			&+\sqrt{2\calA_{\ell}}\sum_{m=1}^{\ell}\Bigl(\mathcal{I}_{\ell,m,\alpha}^{(1)}(t-\tau)\Re Y_{\ell,m}+\mathcal{I}_{\ell,m,\alpha}^{(2)}(t-\tau)\Im Y_{\ell,m}
			\Big),
		\end{align}
		where $\mathcal{I}_{\ell,m,\alpha}^{(j)}(\cdot)$, $j=1,2$, are given in \eqref{Int}. 
	\end{rem}
	
	\section{Approximation to the solution}\label{Sec6}
	
	The results in the previous sections give a series representation of the solution $U(t)$, $t\in(0,\infty)$, of the time-fractional diffusion equation \eqref{System}. This section provides an approximation to the solution $U(t)$ by truncating its expansion at a degree 
	(truncation level) $L\in\N$. Then we give an upper bound for the approximation error of $U_{L}$. 
	
	\begin{defin}
		The approximation $U_{L}$ of degree $L\in\N$ to the solution $U$ given in {\rm(\ref{Exact})} is defined as
		\begin{align}\label{Approx}
			U_{L}(t):&= U_L^H(t)+U_L^I(t)\notag\\
			&= \sum_{\ell=0}^{L}\Big(U^{H}_{\ell}(t)+U^{I}_{\ell}(t)\Big),\quad t\in(0,\infty),
		\end{align}
		where $U^{H}_{\ell}(t)$ is given by \eqref{tru-HomSol}
		and  for $t\leq\tau$, $U^{I}_{\ell}(t)=0$, while for $t>\tau$, $U^{I}_{\ell}(t)$ is given by \eqref{tru-inho}.
	\end{defin}
	\begin{prop}\label{theo3}
		Let $\{\calC_{\ell}:\ell\in\N_{0}\}$, the angular power spectrum of $ \xi$, satisfy \eqref{New-Cl}. Let $Q^{H}_{L}(t)$, $t\in(0,\infty)$, be defined as
		\begin{align*}
			Q^{H}_{L}(t):&=\Big\Vert U^H(t)-U_L^H(t) \Big\Vert_{L_{2}(\Omega\times\bS^2)}=	\Big\Vert \sum_{\ell=L+1}^{\infty}U^{H}_{\ell}(t)\Big\Vert_{L_{2}(\Omega\times\bS^2)},
		\end{align*} 
		where $U^{H}_{\ell}$ is given in \eqref{tru-HomSol}. Then, the following estimates hold true:
		\begin{enumerate}[\rm a)]
			\item\label{itm:AB} for $0< t\leq  \lambda_{L}^{-\frac{1}{\alpha}}$, $L\geq\ell_0$,
			\begin{align*}
				Q^{H}_{L}(t)\leq \widetilde{C}_{\kappa_1} L^{-(\kappa_1-2)/2},
			\end{align*}
			where $\widetilde{C}_{\kappa_1}$ is given by \eqref{Ckappa}.
			\item\label{itm:AC1} for $t> \lambda_{L}^{-\frac{1}{\alpha}}$, $L\geq\ell_0$,
			\begin{align*}
				Q^{H}_{L}(t)\leq \psi^{H}_{\alpha}(t)\widetilde{C}_{\kappa_1}  L^{-(2+\kappa_1)/2},
			\end{align*}
		\end{enumerate}
		where $\psi^{H}_{\alpha}(\cdot)$ is defined as
		\begin{align}\label{C1t}
			\psi^{H}_{\alpha}(t):=	\Gamma(1+\alpha)t^{-\alpha}\quad \alpha\in(0,1].
		\end{align}
	\end{prop}
	
	\begin{proof}
		By \eqref{tru-HomSol} we can write
		\begin{align*}
			(Q^{H}_{L}(t))^2&=\bE\Bigg[\Big\Vert\sum_{\ell=L+1}^{\infty}\Big[ E_\alpha(-\lambda_{\ell} t^\alpha)\Bigl( \sqrt{\calC_{\ell}}Z_{\ell,0}^{(1)}Y_{\ell,0}\\
			&+\sqrt{2\calC_{\ell}}\sum_{m=1}^{\ell}\Bigl(Z_{\ell,m}^{(1)}\Re Y_{\ell,m}+Z_{\ell,m}^{(2)}\Im Y_{\ell,m}\Big)\Big)\Big]\Big\Vert_{L_{2}(\bS^2)}^2\Bigg].
		\end{align*}
		Since $Z_{\ell,m_1}^{(1)},Z_{\ell,m_2}^{(2)}\in\mathcal{Z}_{\ell}, \ell\in\N_{0}$, $m_1=0,1,\dots,\ell$, $m_2=1,\dots,\ell$, we have
		\begin{align*}
			(Q^{H}_{L}(t))^2&=\sum_{\ell=L+1}^{\infty}\Big[ (E_\alpha(-\lambda_{\ell} t^\alpha))^2\Bigl( \calC_{\ell} \bE \left[\big(Z_{\ell,0}^{(1)}\big)^2\right]\Vert Y_{\ell,0}\Vert_{L_{2}(\bS^2)}^2\notag\\
			&+2\calC_{\ell}\sum_{m=1}^{\ell}\Bigl(\bE \left[\big(Z_{\ell,m}^{(1)}\big)^2\right]\Vert \Re Y_{\ell,m}\Vert_{L_{2}(\bS^2)}^2+\bE \left[\big(Z_{\ell,m}^{(2)}\big)^2\right]\Vert \Im Y_{\ell,m}\Vert_{L_{2}(\bS^2)}^2\Big)\Big)\Big]\notag\\
			&=\sum_{\ell=L+1}^{\infty}(E_\alpha(-\lambda_{\ell} t^\alpha))^2\Big[ \calC_{\ell}\Vert Y_{\ell,0}\Vert_{L_{2}(\bS^2)}^2\notag\\
			&+2\calC_{\ell} \sum_{m=1}^{\ell}\left(\Vert \Re Y_{\ell,m}\Vert_{L_{2}(\bS^2)}^2+\Vert \Im Y_{\ell,m}\Vert_{L_{2}(\bS^2)}^2\right)	\Big].
		\end{align*}
		By the properties of the spherical harmonics we have
		\begin{align}\label{Ylm}
			\Vert Y_{\ell,0}\Vert_{L_{2}(\bS^2)}^2+2\sum_{m=1}^{\ell}\Big(\Vert \Re Y_{\ell,m}\Vert_{L_{2}(\bS^2)}^2+&\Vert \Im Y_{\ell,m}\Vert_{L_{2}(\bS^2)}^2\Big)=\sum_{m=-\ell}^{\ell}\Vert Y_{\ell,m}\Vert_{L_{2}(\bS^2)}^2\notag\\
			&=(2\ell+1)P_{\ell}(1)=(2\ell+1),
		\end{align}
		and hence
		\begin{align}\label{U1P_1}
			(Q^{H}_{L}(t))^2=\sum_{\ell=L+1}^{\infty}(E_\alpha(-\lambda_{\ell} t^\alpha))^2 (2\ell+1)\calC_{\ell}.
		\end{align}
		Now, assume $0< t\leq \lambda_{L}^{-\frac{1}{\alpha}}$ for $L\geq\ell_0$, then since $E_{\alpha}(-s^\alpha)\leq1$, equation \eqref{U1P_1} can be bounded by
		\begin{align}\label{N-up}
			(Q^{H}_{L}(t))^2&\leq \sum_{\ell=L+1}^{\infty} (2\ell+1)\calC_{\ell}\notag\\
			&\leq \widetilde{C} \sum_{\ell=L+1}^{\infty}(2\ell+1)\ell^{-\kappa_1}\notag\\
			&\leq \widetilde{C}\int_{L}^{\infty}(2x+1)x^{-\kappa_1}dx\notag\\
			&= \widetilde{C}\bigg(\dfrac{2}{\kappa_1-2}+\dfrac{L^{-1}}{\kappa_1-1}\bigg)L^{2-\kappa_1}\notag\\
			&\leq \widetilde{C}\bigg(\dfrac{2}{\kappa_1-2}+\dfrac{1}{\kappa_1-1}\bigg)L^{2-\kappa_1}.
		\end{align}
		Now, let us consider part \ref{itm:AC1}. Since $E_{\alpha}(-\lambda_{\ell}t^\alpha)$ is positive and decreasing for $t>0$,
		then for $t> \lambda_{L}^{-\frac{1}{\alpha}}$, using \eqref{N-Simon}, \eqref{U1P_1} can be bounded by
		\begin{align}\label{NewI1}
			(Q^{H}_{L}(t))^2&\leq(E_\alpha(-\lambda_{L} t^\alpha))^2\sum_{\ell=L+1}^{\infty} (2\ell+1)\calC_{\ell}\notag\\
			&\leq (\Gamma(1+\alpha))^2t^{-2\alpha}\lambda_{L}^{-2}\sum_{\ell=L+1}^{\infty} (2\ell+1)\calC_{\ell}\notag\\
			&\leq \big(\psi^{H}_{\alpha}(t)\big)^2\lambda_{L}^{-2}\widetilde{C}_{\kappa_1}^2L^{2-\kappa_1},
		\end{align}
		thus completing the proof.
	\end{proof}
	
	In the next proposition, we derive an upper bound for the approximation errors of the truncated solution $U_{L}^I(t)$, $t>\tau$.
	
	Let $\psi^{I}_{\alpha}(t)$, $t>0$, be defined as
	\begin{align}\label{int41}
		\psi^{I}_{\alpha}(t):=
		\begin{cases} 
			\sqrt{1+M_{\alpha}t^{1-2\alpha}}, & \alpha\in(0,\frac{1}{2}), \\
			\sqrt{1+M_{\alpha}}, & \alpha\in(\frac{1}{2},1],\\
			K(t), & \alpha=\frac{1}{2},
		\end{cases}
	\end{align}
	where $M_{\alpha}$ is given by \eqref{C2} and 
	\begin{align}
		K(t):= \begin{cases} 
			\sqrt{1+M_{\frac{1}{2}}(2+\ln(t))}, & t>1, \\
			\sqrt{1+2M_{\frac{1}{2}}}, & t\leq1.
		\end{cases}
	\end{align}

	\begin{prop}\label{theo4}
		Let $\{\calA_{\ell}:\ell\in\N_{0}\}$, the angular power spectrum of $W_\tau$, satisfy \eqref{New-Al}. Let $Q^{I}_{L}(t)$, $t>\tau$, be defined as
		\begin{align}
			Q^{I}_{L}(t):&=\Big\Vert U^I(t)-U_L^I(t) \Big\Vert_{L_{2}(\Omega\times\bS^2)}=	\Big\Vert \sum_{\ell=L+1}^{\infty}U^{I}_{\ell}(t)\Big\Vert_{L_{2}(\Omega\times\bS^2)},
		\end{align}
		where for $t>\tau$, $U^{I}_{\ell}(t)$ is given by \eqref{tru-inho}. Then, the following estimates hold true:
		\begin{enumerate}[\rm i)]
			\item\label{itm:B1} for $\tau< t\leq\tau+ \lambda_{L}^{-\frac{1}{\alpha}}$, $L\geq\ell_0$,
			\begin{align*}
				Q^{I}_{L}(t)\leq \widetilde{A}_{\kappa_2}L^{-\frac{(\kappa_2+\frac{2}{\alpha}-2)}{2}},
			\end{align*}
			where $\widetilde{A}_{\kappa_2}$ is given by \eqref{Akappa}. 
			\item\label{itm:B2} For $t>\tau+ \lambda_{L}^{-\frac{1}{\alpha}}$, $L\geq\ell_0$, 
			\begin{align*}
				Q^{I}_{L}(t)\leq \psi^{I}_{\alpha}(t-\tau)\widetilde{A}_{\kappa_2}L^{-\frac{\gamma_\alpha(\kappa_2)}{2}},
			\end{align*}
			where $\gamma_\alpha(\kappa_2)$ is given in \eqref{C1} and $\psi^{I}_{\alpha}(\cdot)$ is given in \eqref{int41}.
		\end{enumerate}
	\end{prop}
	\begin{proof}
		By \eqref{tru-inho} the term $(Q^{I}_{L}(t))^2$ can be written, for $t>\tau$, as
		\begin{align*}
			(Q^{I}_{L}(t))^2&=\Big\Vert \sum_{\ell=L+1}^{\infty}U_{\ell}^{I}(t)\Big\Vert_{L_{2}(\Omega\times\bS^2)}^2=\Big\Vert  \sum_{\ell=L+1}^{\infty}\Big[\sqrt{\calA_{\ell}}\mathcal{I}_{\ell,0,\alpha}^{(1)}(t-\tau)Y_{\ell,0}\\
			&+\sqrt{2\calA_{\ell}}\sum_{m=1}^{\ell}\Bigl(\mathcal{I}_{\ell,m,\alpha}^{(1)}(t-\tau)\Re Y_{\ell,m}+\mathcal{I}_{\ell,m,\alpha}^{(2)}(t-\tau)\Im Y_{\ell,m}\Big)\Big]\Big\Vert_{L_{2}(\Omega\times\bS^2)}^2.
		\end{align*}
		By Proposition \ref{PropVar}, for $\ell\in\N_{0}, m_1=0,\dots,\ell,m_2=1,\dots,\ell$, 
  we know that $\{\big(\mathcal{I}_{\ell,m_1,\alpha}^{(1)}(\cdot),\mathcal{I}_{\ell,m_2,\alpha}^{(2)}(\cdot)\big)\}$ is a sequence of independent, 
  real-valued Gaussian random variables with mean zero and variances $\sigma_{\ell,t-\tau,\alpha}^2$, $t>\tau$. It follows that
		\begin{align}\label{upQ}
			(Q^{I}_{L}(t))^2&=\sum_{\ell=L+1}^{\infty}\Big[ \Bigl( \calA_{\ell} \bE \left[\big(\calI_{\ell,0,\alpha}^{(1)}(t-\tau)\big)^2\right]\Vert Y_{\ell,0}\Vert_{L_{2}(\bS^2)}^2\notag\\
			&+2\calA_{\ell}\sum_{m=1}^{\ell}\Bigl(\bE \Big[\big(\calI_{\ell,m,\alpha}^{(1)}(t-\tau)\big)^2\Big]\Vert \Re Y_{\ell,m}\Vert_{L_{2}(\bS^2)}^2\notag\\&\qquad\qquad\qquad+\bE \Big[\big(\calI_{\ell,m,\alpha}^{(2)}(t-\tau)\big)^2\Big]\Vert \Im Y_{\ell,m}\Vert_{L_{2}(\bS^2)}^2\Big)\Big)\Big]\notag\\
			&=\sum_{\ell=L+1}^{\infty}\sigma_{\ell,t-\tau,\alpha}^2\Big[ \calA_{\ell}\Vert Y_{\ell,0}\Vert_{L_{2}(\bS^2)}^2\notag\\
			&+2\calA_{\ell} \sum_{m=1}^{\ell}\left(\Vert \Re Y_{\ell,m}\Vert_{L_{2}(\bS^2)}^2+\Vert \Im Y_{\ell,m}\Vert_{L_{2}(\bS^2)}^2\right)	\Big]\notag\\
			&= \sum_{\ell=L+1}^{\infty} \calA_{\ell}(2\ell+1)\sigma_{\ell,t-\tau,\alpha}^2\notag\\
			&\leq \sigma_{L,t-\tau,\alpha}^2\sum_{\ell=L+1}^{\infty} \calA_{\ell}(2\ell+1),
		\end{align}
		since by \eqref{var} $\sigma_{\ell,t,\alpha}^2$ is decreasing in $\ell$.
		
		\indent
		Now let us consider part \ref{itm:B1}. Since $0< (t-\tau)\lambda_{L}^{\frac{1}{\alpha}}\leq1$, then by \eqref{N-Sig}, $\sigma_{L,t-\tau,\alpha}^2$ is bounded by $\lambda_{L}^{-\frac{1}{\alpha}}$. Similar to \eqref{N-up}, the upper bound \eqref{upQ} can be bounded, for $\tau< t\leq \tau+ \lambda_{L}^{-\frac{1}{\alpha}}$, $L\geq\ell_0$, by 
		\begin{align*}
			(Q^{I}_{L}(t))^2&\leq \lambda_{L}^{-\frac{1}{\alpha}}\sum_{\ell=L+1}^{\infty} \calA_{\ell}(2\ell+1)
			\leq \widetilde{A}_{\kappa_2}^2 L^{2-(\kappa_2+\frac{2}{\alpha})},
		\end{align*} 
		which completes the proof of \ref{itm:B1}. 
		
		\indent
		To prove \ref{itm:B2}, note that by using the upper bound \eqref{N-Sig} with the estimates \eqref{N-int37}, \eqref{N-int38}, and \eqref{N-int39},  we obtain  
		\begin{align}\label{int40}
			\sigma_{L,t-\tau,\alpha}^2\leq \begin{cases} 
				\lambda_{L}^{-\frac{1}{\alpha}}+M_{\alpha}(t-\tau)^{1-2\alpha}\lambda_{L}^{-2}, & \alpha\in(0,\frac{1}{2}), \\
				\lambda_{L}^{-\frac{1}{\alpha}}(1+M_{\alpha}), & \alpha\in(\frac{1}{2},1],\\
				\lambda_{L}^{-2}	(1+M_{\frac{1}{2}}\ln(\lambda_{L}^2(t-\tau))), & \alpha=\frac{1}{2},
			\end{cases}
		\end{align}	
		where $M_\alpha$ is given in \eqref{C2}.
		
		\indent	
		By using \eqref{upQ} with the estimate \eqref{int40} and the properties of $\calA_\ell$ given in \eqref{New-Al}, we get 
		\begin{align}\label{UpI2}
			(Q^{I}_{L}(t))^2&\leq \widetilde{A} \sigma_{L,t-\tau,\alpha}^2\sum_{\ell=L+1}^{\infty} (2\ell+1)\ell^{-\kappa_2}\notag\\
			& \leq \big(\psi^{I}_{\alpha}(t-\tau)\widetilde{A}_{\kappa_2}\big)^2L^{-\kappa_\alpha(\kappa_2)},
		\end{align}
		where $\gamma_\alpha(\kappa_2)$ is given in \eqref{C1} and $\psi^{I}_{\alpha}(\cdot)$ is given in \eqref{int41}, thus completing the proof.
	\end{proof}
	By combining the results of Propositions \ref{theo3} and \ref{theo4}, we present the following theorem.
	\begin{theo}\label{The5}
		Let $U(t)$, $t\in(0,\infty)$, be the solution \eqref{Exact} to the equation \eqref{System}, and the conditions of Propositions {\rm\ref{theo3}} and {\rm\ref{theo4}} be satisfied. Let $\widetilde{\kappa}_\alpha:=\min(\kappa_1-2,\kappa_2+\frac{2}{\alpha}-2)$, $\widehat{\kappa}_{\alpha}:=\min(\kappa_1+2,\kappa_2+\frac{2}{\alpha}-2)$, and $\kappa_\alpha:=\min\big(\kappa_1+2,\gamma_\alpha(\kappa_2)\big)$, where $\gamma_\alpha(\kappa_2)$ is given in \eqref{C1}. Let $Q_L(t)$, $t\in(0,\infty)$, be defined as
		\begin{align}\label{fullq}
			Q_L(t):=\Big\Vert U(t)- U_{L}(t)\Big\Vert_{L_{2}(\Omega\times\bS^2)},\quad L\geq1,  
		\end{align}
		where $U_{L}$ is given by \eqref{Approx}. 
		Then, the following estimates hold true:
		\begin{enumerate}[\rm I)]
			\item\label{itm:AA1} For $0< t\leq \lambda_{L}^{-\frac{1}{\alpha}}$,\ $\tau\geq\lambda_{L}^{-\frac{1}{\alpha}}$,\ $L\geq\ell_0$,
			\begin{align*}
				Q_L(t)\leq \widetilde{C}_{\kappa_1}L^{-(\kappa_1-2)/2},
			\end{align*}
			where $\widetilde{C}_{\kappa_1}$ is given by \eqref{Ckappa}.
			\item\label{itm:AA0} For $\lambda_{L}^{-\frac{1}{\alpha}}<t\leq\tau+\lambda_{L}^{-\frac{1}{\alpha}}$, $L\geq\ell_0$,
			\begin{align*}
				Q_L(t)\leq \left((\psi^{H}_{\alpha}(t)\widetilde{C}_{\kappa_1})^2+\widetilde{A}_{\kappa_2}^2\right)^{\frac{1}{2}}L^{-\widehat{\kappa}_{\alpha}/2},
			\end{align*}
			where $\widetilde{A}_{\kappa_2}$ is given by \eqref{Akappa}.
			\item\label{itm:AA2} For $t>\tau+\lambda_{L}^{-\frac{1}{\alpha}}$, $L\geq\ell_0$,
			\begin{align*}
				Q_L(t)\leq \left(\big(\psi^{H}_{\alpha}(t)\widetilde{C}_{\kappa_1}\big)^2+\big(\psi^{I}_{\alpha}(t-\tau)\widetilde{A}_{\kappa_2}\big)^2\right)^{\frac{1}{2}}L^{-\kappa_\alpha/2},
			\end{align*}
			where $\psi^{H}_{\alpha}(\cdot)$ is given by \eqref{C1t} and $\psi^{I}_{\alpha}(\cdot)$ is given by \eqref{int41}. 
		\end{enumerate}
	\end{theo}
	
	\begin{rem}
		In view of Theorem {\rm\ref{The5}} and the functions $\psi_{\alpha}^{H}(t)$ and $\psi_{\alpha}^{I}(t)$ given by \eqref{C1t} and \eqref{int41} respectively, it is easy to see that for $\alpha\in(\frac{1}{2},1]$, both functions $\psi_{\alpha}^{H}(t)$ and $\psi_{\alpha}^{I}(t)$ are bounded when $t\to\infty$ and hence the truncation errors $Q_L(t)$ are bounded when $t\to\infty$ {\rm(see case \ref{itm:AA2})}. 
		Also, it is worth noting that
		when $\alpha\in(\frac{1}{2},1]$, the truncation errors $Q_L(t)$ as functions of $t$ are bounded as well when $t\to0$ {\rm(see case \ref{itm:AA1})}, which improves the upper bound for the truncation errors obtained in {\rm\cite{Anh_et_al}} when $\alpha=1$ and $t$ is sufficiently small.
	\end{rem}

	\section{ Temporal increments and sample H\"{o}lder continuity of solution}\label{Sec7}
	
	In this section we study some properties of the stochastic solution of equation \eqref{System}. In particular, we derive an upper bound, in $L_2$-norm, of the temporal increments of the stochastic solution $U(t)$, $t\in(\tau,\infty)$, to equation \eqref{System}. Then we study the sample properties of the stochastic solution $U$. Under some conditions, we show the existence of a locally H\"{o}lder continuous modification of the random field $U$. We demonstrate how the $\mathbb{P}$-almost sure H\"{o}lder continuity of the solution $U$ depends on the decay of the angular power spectra of the initial condition $\xi$ and the driving noise $W_\tau$.

	\begin{theo}\label{Theo6}
		Let $U$ be the solution \eqref{Exact} to the equation \eqref{System}. Let $\calJ_h$, $h>0$, be defined as
		\begin{align}\label{Jinc}
			\calJ_h(t):= \Big\Vert U(t+h)- U(t)\Big\Vert_{L_{2}(\Omega\times\bS^2)}, \quad t>\tau.
		\end{align}
		Then, for $h>0$, there holds 
		\[
		\calJ_h(t)\leq q(t) h^{\frac{1}{2}},
		\]
		where $q(t)$, $t>\tau$, is given by
		\begin{align*}
			q(t)= \sqrt{C\widetilde{C}_{\kappa_1}^2t^{-1}+(1+C)\widetilde{A}_{\kappa_2}^2},\quad C>0,
		\end{align*}
		$\widetilde{C}_{\kappa_1}$, and $\widetilde{A}_{\kappa_2}$ are defined in \eqref{Ckappa} and \eqref{Akappa} respectively.
	\end{theo}
	\begin{proof}
		Let $t\in(\tau,\infty)$, then by the triangle inequality for $L_{2}(\Omega\times\bS^2)$, we can write 
		\[
		(\calJ_h(t))^2\leq\calJ_h^H(t)+\calJ_h^I(t),
		\]
		where $\calJ_h^H(t)$ and $\calJ_h^I(t)$ are defined as
		\[
		\calJ_h^H(t):=\Big\Vert \sum_{\ell=0}^{\infty} \left(U^{H}_{\ell}(t+h)-U^{H}_{\ell}(t)\right)\Big\Vert_{L_{2}(\Omega\times\bS^2)}^2,
		\]
		\[
		\calJ_h^I(t):=\Big\Vert \sum_{\ell=0}^{\infty}\left(U^{I}_{\ell}(t+h)-U^{I}_{\ell}(t)\right)\Big\Vert_{L_{2}(\Omega\times\bS^2)}^2,
		\]
		where $U^{H}_{\ell}$ and $U^{I}_{\ell}$ are given by \eqref{tru-HomSol} and \eqref{tru-inho} respectively.
		
		\indent
		Using \eqref{tru-HomSol}, the term $\calJ_h^H(t)$ can be written as
		\begin{align*}
			\calJ_h^H(t)&=\bE\Bigg[\Big\Vert\sum_{\ell=0}^{\infty}\Big[ (E_\alpha(-\lambda_{\ell} (t+h)^\alpha)-E_\alpha(-\lambda_{\ell} t^\alpha))^2\Bigl( \sqrt{\calC_{\ell}}Z_{\ell,0}^{(1)}Y_{\ell,0}\\
			&+\sqrt{2\calC_{\ell}}\sum_{m=1}^{\ell}\Bigl(Z_{\ell,m}^{(1)}\Re Y_{\ell,m}+Z_{\ell,m}^{(2)}\Im Y_{\ell,m}\Big)\Big)\Big]\Big\Vert_{L_{2}(\bS^2)}^2\Bigg].
		\end{align*}
		Since 
		$Z_{\ell,m_1}^{(1)},Z_{\ell,m_2}^{(2)}\in\mathcal{Z}_{\ell}$, $\ell\in\N_0, m_1=0,\dots,\ell$, $m_2=1,2,\dots,\ell$, there holds, using \eqref{Ylm},
		\begin{align}\label{Jh1}
			\calJ_h^H(t)&=\sum_{\ell=0}^{\infty}\Big[ \big\vert E_\alpha(-\lambda_{\ell} (t+h)^\alpha)-E_\alpha(-\lambda_{\ell} t^\alpha)\big\vert^2\Bigl( \calC_{\ell} \bE \left[\big(Z_{\ell,0}^{(1)}\big)^2\right]\Vert Y_{\ell,0}\Vert_{L_{2}(\bS^2)}^2\notag\\
			&+2\calC_{\ell}\sum_{m=1}^{\ell}\Bigl(\bE \left[\big(Z_{\ell,m}^{(1)}\big)^2\right]\Vert \Re Y_{\ell,m}\Vert_{L_{2}(\bS^2)}^2+\bE \left[\big(Z_{\ell,m}^{(2)}\big)^2\right]\Vert \Im Y_{\ell,m}\Vert_{L_{2}(\bS^2)}^2\Big)\Big)\Big]\notag\\
			&=\sum_{\ell=1}^{\infty}\big\vert E_\alpha(-\lambda_{\ell} (t+h)^\alpha)-E_\alpha(-\lambda_{\ell} t^\alpha)\big\vert^2(2\ell+1)\calC_{\ell}.
		\end{align}
		Since $0<E_{\alpha}(\cdot)\leq1$, \eqref{Jh1} becomes bounded by
		\begin{align}\label{Jh}
			\calJ_h^H(t)&\leq \sum_{\ell=1}^{\infty}\big\vert E_\alpha(-\lambda_{\ell} (t+h)^\alpha)-E_\alpha(-\lambda_{\ell} t^\alpha)\big\vert(2\ell+1)\calC_{\ell}.
		\end{align}
		Now using the mean-value theorem for the function $E_\alpha(-\lambda_{\ell} t^\alpha)$ (see \cite{Gorenflo2014}, equation 4.3.1 ), then \eqref{Jh} becomes
		\begin{align}\label{Et2}
			\calJ_h^H(t)	&\leq h\sum_{\ell=1}^{\infty}\big\vert -\alpha\lambda_{\ell}t_2^{\alpha-1} E_{\alpha,\alpha+1}^{2}(-\lambda_{\ell} t_2^\alpha)\big\vert(2\ell+1)\calC_{\ell}\notag\\
			& \leq ht^{\alpha-1}\sum_{\ell=1}^{\infty} \lambda_{\ell}\alpha E_{\alpha,\alpha+1}^{2}(-\lambda_{\ell} t^\alpha)(2\ell+1)\calC_{\ell},
		\end{align}
		where $t_2\in(t,t+h)$ and $E_{\alpha,\alpha+1}^{2}(\cdot)$ is the 3-parameter Mittag-Leffler function (see \cite{Gorenflo2014}) and the second step uses that $E_{\alpha,\alpha+1}^{2}(\cdot)$ is decreasing in $t$.

		\indent
		Using the relations (see \cite{Gorenflo2014}, equation 5.1.14) 
		\begin{align}\label{E2-up}
			\alpha E_{\alpha,\beta}^{2}=E_{\alpha,\beta-1}-(1+\alpha-\beta)E_{\alpha,\beta},\quad \alpha>0,\ \beta>1,  
		\end{align}
		and \eqref{EMitagg} we estimate \eqref{Et2} as
		\begin{align*}
			\calJ_h^H(t)	& \leq h t^{\alpha-1}\sum_{\ell=1}^{\infty} \lambda_{\ell}E_{\alpha,\alpha}(-\lambda_{\ell} t^\alpha)(2\ell+1)\calC_{\ell}\\
			& \leq C ht^{-1}\sum_{\ell=1}^{\infty} (2\ell+1)\calC_{\ell}\Big(\frac{\lambda_{\ell}t^{\alpha}}{1+\lambda_{\ell} t^\alpha}\Big)\\
			&\leq C ht^{-1}\sum_{\ell=1}^{\infty} (2\ell+1)\calC_{\ell}\\
			& \leq C ht^{-1}\widetilde{C}_{\kappa_1}^2.
		\end{align*}
		Now we treat the term $\calJ_h^I(t)$. 
		By using \eqref{tru-inho} we have
		\begin{align}\label{Int-s}
			\calJ_h^I(t)&=\sum_{\ell=0}^{\infty}\Big[ \Bigl( \calA_{\ell} \bE \left[\big(\calI_{\ell,0,\alpha}^{(1)}(t-\tau+h)-\calI_{\ell,0,\alpha}^{(1)}(t-\tau)\big)^2\right]\Vert Y_{\ell,0}\Vert_{L_{2}(\bS^2)}^2\notag\\
			&+2\calA_{\ell}\sum_{m=1}^{\ell}\Bigl(\bE \left[\big(\calI_{\ell,m,\alpha}^{(1)}(t-\tau+h)-\calI_{\ell,m,\alpha}^{(1)}(t-\tau)\big)^2\right]\Vert \Re Y_{\ell,m}\Vert_{L_{2}(\bS^2)}^2\notag\\
			&+\bE \left[\big(\calI_{\ell,m,\alpha}^{(2)}(t-\tau+h)-\calI_{\ell,m,\alpha}^{(2)}(t-\tau)\big)^2\right]\Vert \Im Y_{\ell,m}\Vert_{L_{2}(\bS^2)}^2\Big)\Big)\Big],
		\end{align}
		since by  Proposition \ref{PropVar}, $\calI_{\ell,m_1,\alpha}^{(1)},\calI_{\ell,m_2,\alpha}^{(2)}$,  $m_1=0,1,\dots,\ell$, $m_2=1,\dots,\ell$, are independent.
		
		\indent
		Let $\Xi_{\ell,m_1,\alpha}^{(1)}(h)$ and $\Xi_{\ell,m_2,\alpha}^{(2)}(h)$, $m_1=0,1,\dots,\ell$, $m_2=1,\dots,\ell$, be defined as
		\[
		\Xi_{\ell,m_1,\alpha}^{(1)}(h):= \bE \left[\big(\calI_{\ell,m_1,\alpha}^{(1)}(t-\tau+h)-\calI_{\ell,m_1,\alpha}^{(1)}(t-\tau)\big)^2\right]
		\]
		and 
		\[
		\Xi_{\ell,m_2,\alpha}^{(2)}(h):=\bE\left[\big(\calI_{\ell,m_2,\alpha}^{(2)}(t-\tau+h)-\calI_{\ell,m_2,\alpha}^{(2)}(t-\tau)\big)^2\right].
		\]
		\noindent
		Since by Proposition \ref{PropVar}, $\beta_{\ell,m_1}^{(1)}\in\mathcal{B}_{\ell,0}$, $\ell\in\N_0, m_1=0,\dots,\ell$, there holds
		\begin{align}\label{Pf:incovu}
			\Xi_{\ell,m_1,\alpha}^{(1)}(h)&= \bE \Bigg[\Bigg(\int_{t-\tau}^{t-\tau+h}E_{\alpha}(-\lambda_{\ell}(t-\tau+h-u)^{\alpha})d\beta_{\ell,m_1}^{(1)}(u)\notag\\
			&+\int_{0}^{t-\tau}\Big(E_{\alpha}(-\lambda_{\ell}(t+h-\tau-u)^{\alpha})-E_{\alpha}(-\lambda_{\ell}(t-\tau-u)^{\alpha})\Big)d\beta_{\ell,m_1}^{(1)}(u) \Bigg)^2\Bigg]\notag\\
			&\leq \bE \Bigg[\Bigg(\int_{t-\tau}^{t-\tau+h}E_{\alpha}(-\lambda_{\ell}(t-\tau+h-u)^{\alpha})d\beta_{\ell,m_1}^{(1)}(u)\Bigg)^2\Bigg]\notag\\
			&+  \bE \Bigg[\Bigg(\int_{0}^{t-\tau}\Big(E_{\alpha}(-\lambda_{\ell}(t+h-\tau-u)^{\alpha})-E_{\alpha}(-\lambda_{\ell}(t-\tau-u)^{\alpha})\Big)d\beta_{\ell,m_1}^{(1)}(u)\Bigg)^2\Bigg]\notag\\
			&\leq\int_{t-\tau}^{t-\tau+h}\Big(E_{\alpha}(-\lambda_{\ell}(t-\tau+h-u)^{\alpha}\Big)^2du\notag\\
			&+ \int_{0}^{t-\tau}\Big(E_{\alpha}(-\lambda_{\ell}(t+h-\tau-u)^{\alpha})-E_{\alpha}(-\lambda_{\ell}(t-\tau-u)^{\alpha})\Big)^2du,
		\end{align}
		where the last step uses It\^o's isometry (see \cite{Kuo2006}).
		
		\indent
		Since $0<E_{\alpha}(\cdot)\leq1$, we can write \eqref{Pf:incovu}, using Remark \ref{rem1}, as 
		\begin{align}\label{Pf:incov}
			\Xi_{\ell,m_1,\alpha}^{(1)}(h)&\leq  \sigma_{\ell,h,\alpha}^2 + \int_{0}^{t-\tau}\Big\vert E_{\alpha}(-\lambda_{\ell}(t+h-\tau-u)^{\alpha})-E_{\alpha}(-\lambda_{\ell}(t-\tau-u)^{\alpha})\Big\vert du\notag\\
			&\leq \sigma_{\ell,h,\alpha}^2+h\int_{0}^{t-\tau}\big\vert -\alpha\lambda_{\ell}(t-\tau-u)^{\alpha-1} E_{\alpha,\alpha+1}^{2}(-\lambda_{\ell} (t-\tau-u)^\alpha)\big\vert du\notag\\
			&\leq \sigma_{\ell,h,\alpha}^2+h\lambda_{\ell}\int_{0}^{t-\tau}\big\vert (t-\tau-u)^{\alpha-1} E_{\alpha,\alpha}(-\lambda_{\ell} (t-\tau-u)^\alpha)\big\vert du,
		\end{align}
		where the second step uses the mean-value theorem for the function $E_\alpha(-\lambda_{\ell} (t-\tau)^\alpha)$ (see \cite{Gorenflo2014}, equation 4.3.1) and the last step uses \eqref{E2-up}.
		
		\indent
		Using Remark \ref{rem1} and relation (4.4.4) in \cite{Gorenflo2014}, then \eqref{Pf:incov} becomes
		\begin{align}\label{pf:new}
			\Xi_{\ell,m_1,\alpha}^{(1)}(h)&\leq h+h\lambda_{\ell} (t-\tau)^{\alpha} E_{\alpha,\alpha}(-\lambda_{\ell} (t-\tau)^\alpha)\notag\\
			& \leq h+ h  C\frac{\lambda_{\ell}(t-\tau)^\alpha}{1+\lambda_{\ell}(t-\tau)^\alpha}\notag\\
			&\leq h(1+C),
		\end{align}
		where the second step uses \eqref{EMitagg}.
		
		\indent
		Similarly, we obtain, for $m_2=1,\dots,\ell$, that
		\begin{align}\label{Pf:incov2}
			\Xi_{\ell,m_2,\alpha}^{(2)}(h)&\leq h(1+C).
		\end{align}
		
		\noindent	
		Using \eqref{pf:new}, \eqref{Pf:incov2}, and \eqref{Ylm} with \eqref{Int-s}, we get
		\begin{align*}
			\calJ_h^I(t)	& \leq \sum_{\ell=0}^{\infty} h(1+C)(2\ell+1)\calA_{\ell} \\
			& \leq h(1+C)\sum_{\ell=0}^{\infty} (2\ell+1)\calA_{\ell} \leq h(1+C)\widetilde{A}_{\kappa_2}^2.
		\end{align*}
		Thus the result follows by combining the upper bounds $\calJ_h^H(\cdot)$ and $\calJ_h^I(\cdot)$.
	\end{proof}

	Now we demonstrate continuity properties of the stochastic solution $U$  at a given time $t$ with respect to the geodesic distance on $\bS^2$, i.e.  $\tilde{d}(\bsx,\bsy)$, $\bsx,\bsy\in\bS^2$.
	
	Notice that by Proposition  \ref{Fulliso}, the stochastic solution $U$ is known to be mean square continuous (see \cite{MarPec2013}). However, to obtain sample H\"{o}lder continuity of the stochastic solution $U$ we need stronger assumptions on the angular power spectra $\calC_\ell$ and $\calA_\ell$, $\ell\in\N_{0}$. 
	\begin{ass}\label{ass1}
		Assume that the angular power spectra $\{\calC_\ell,\ \ell\in\N_0\}$ and $\{\calA_\ell,\ \ell\in\N_0\}$ satisfy \eqref{New-Cl} and \eqref{New-Al} with $\kappa_1,\kappa_2>2(1+\beta^*)$ for some $\beta^*\in(0,1]$.
	\end{ass}
	It follows from Assumption {\rm\ref{ass1}} that for $\beta^*\in(0,1]$ there exist constants $K^{(1)}_{\beta^*},K^{(2)}_{\beta^*}>0$ such that
	\[
	K^{(1)}_{\beta^*}:=	\sum_{\ell=1}^{\infty}  \ell^{1+2\beta^*}\calC_{\ell}<\infty\ \text{and}\ K^{(2)}_{\beta^*}:=\sum_{\ell=1}^{\infty}  \ell^{1+2\beta^*}\calA_{\ell}<\infty.
	\]
	
The following theorem provides bounds on moments of the variance $Var[U(\bsx,t)-U(\bsy,t)]$, for $\bsx,\bsy\in\bS^2$ in terms of the geodesic distance $\tilde{d}(\bsx,\bsy)$.

	\begin{theo}\label{Hoelder}
		Let $U$, $t\in(0,\infty)$, be the solution given by \eqref{Exact} to the equation \eqref{System}, and the angular power spectra $\calC_\ell$ and $\calA_\ell$ satisfy the Assumption {\rm\ref{ass1}}. Then there exists a constant $K_{\beta^*}>0$ such that
		\[
		Var[U(\bsx,t)-U(\bsy,t)]\leq K_{\beta^*}\; (\tilde{d}(\bsx,\bsy))^{2\beta^*},\quad \beta^*\in(0,1].
		\]
	\end{theo}
	\begin{proof}
		Let $t\leq\tau$. Then by \eqref{Exact} and \eqref{gamlm} there holds 
		\begin{align*}
			\bE\left[U(\bsx,t)U(\bsy,t)\right]&=\bE\left[U^{H}(\bsx,t)U^{H}(\bsy,t)\right]\\
			&= \sum_{\ell=0}^\infty(2\ell+1)\calC_\ell (E_\alpha(-\lambda_{\ell} t^\alpha))^2P_{\ell}(\bsx\cdot\bsy),
		\end{align*}
		where $P_{\ell}(\cdot)$, $\ell\in\N_{0}$, is the Legendre polynomial of degree $\ell$. 
		
		\indent
		Thus for $t\leq\tau$ there holds
		\begin{align}\label{holder-1}
			Var[U(\bsx,t)-U(\bsy,t)]&=Var[U^H(\bsx,t)]+Var[U^H(\bsy,t)]-2\bE[U^H(\bsx,t)U^H(\bsy,t)]\notag\\
			&=2\big(	Var[U^H(\bsx,t)]-\bE[U^H(\bsx,t)U^H(\bsy,t)]\big)\notag\\
			&=2\sum_{\ell=0}^\infty(2\ell+1)\calC_\ell (E_\alpha(-\lambda_{\ell} t^\alpha))^2(1-P_{\ell}(\bsx\cdot\bsy))\notag\\
			&\leq 2\sum_{\ell=0}^\infty(2\ell+1)\calC_\ell\vert1-P_{\ell}(\cos \tilde{d}(\bsx,\bsy))\vert,
		\end{align}
		where the last step uses that $0<E_\alpha(-\lambda_{\ell} t^\alpha)\leq1$. 
		
		\indent
		Since $P_{\ell}(1)=1$ for all $\ell\in\N_0$ and $P_{\ell}^{\prime}(x)\leq P_{\ell}^{\prime}(1)$ for $x\in[-1,1]$, we obtain
		\begin{align}\label{Holder1}
			\vert 1-P_{\ell}(x)\vert=\Big\vert \int_{x}^{1}P_{\ell}^{\prime}(y)dy\Big\vert\leq \vert1-x\vert\Big(\frac{\ell(\ell+1)}{2}\Big),\quad x\in[-1,1].
		\end{align}
		Also, we observe that 
		\begin{align}\label{Holder2}
			\vert 1-P_{\ell}(x)\vert\leq 2,\quad x\in[-1,1].
		\end{align}
		Then, for some $\beta^*\in(0,1]$ we can write, using \eqref{Holder1} and \eqref{Holder2},
		\begin{align}\label{Leg}
			\vert 1-P_{\ell}(x)\vert&=\vert 1-P_{\ell}(x)\vert^{\beta^*} \vert 1-P_{\ell}(x)\vert^{1-\beta^*}\notag\\
			&\leq \Big(\vert1-x\vert\frac{\ell(\ell+1)}{2}\Big)^{\beta^*} 2^{1-\beta^*}\notag\\
			&\leq 2^{1-2\beta^*}\vert1-x\vert^{\beta^*}(\ell(\ell+1))^{\beta^*}.
		\end{align}
		Thus \eqref{holder-1} becomes bounded, using \eqref{Leg}, by
		\begin{align*}
			Var[U(\bsx,t)-U(\bsy,t)]&\leq  2^{2-2\beta^*}\vert 1-\cos \tilde{d}(\bsx,\bsy)\vert^{\beta^*}\sum_{\ell=0}^\infty(2\ell+1)\calC_\ell\lambda_{\ell}^{\beta^*}\\
			&\leq 2^{2-2\beta^*}(\tilde{d}(\bsx,\bsy))^{2\beta^*}\sum_{\ell=1}^\infty(2\ell+1)\calC_\ell(\ell(\ell+1))^{\beta^*}\\
			&\leq 2^{4-\beta^*}(\tilde{d}(\bsx,\bsy))^{2\beta^*}K^{(1)}_{\beta^*},
		\end{align*}
		where the second step uses the inequality  
		\begin{align*}
			\vert 1-\cos r\vert&= \Big\vert \int_{0}^r \sin x dx\Big\vert\leq r \sin r\\
			&\leq r \Big\vert \int_{0}^r \cos x dx\Big\vert\leq r^2.
		\end{align*}
		\noindent
		Similarly, for $t>\tau$, we have
		\begin{align*}
			\bE\left[U(\bsx,t)U(\bsy,t)\right]&= \sum_{\ell=0}^\infty(2\ell+1)\left(\calC_\ell (E_\alpha(-\lambda_{\ell} t^\alpha))^2 +\calA_\ell\sigma_{\ell,t-\tau,\alpha}^2\right)P_{\ell}(\bsx\cdot\bsy)
		\end{align*}
		and by Assumption \ref{ass1} and \eqref{Leg} we obtain
		\begin{align*}
			Var[U(\bsx,t)-U(\bsy,t)]&=2\sum_{\ell=0}^\infty(2\ell+1)\left(\calC_\ell (E_\alpha(-\lambda_{\ell} t^\alpha))^2 +\calA_\ell\sigma_{\ell,t-\tau,\alpha}^2\right)(1-P_{\ell}(\bsx\cdot\bsy))\\
			&\leq 2\sum_{\ell=0}^\infty(2\ell+1)\left(\calC_\ell +(t-\tau)\calA_\ell\right)(1-P_{\ell}(\bsx\cdot\bsy))\\
			&\leq 2^{4-\beta^*}(\tilde{d}(\bsx,\bsy))^{2\beta^*}\big(K^{(1)}_{\beta^*}+(t-\tau)K^{(2)}_{\beta^*}\big),
		\end{align*}
		where the second step uses Remark \ref{rem1}. 
		
		\indent
		Thus for $t>0$, there exists a constant $K_{\beta^*}$ such that  
		\[
		Var[U(\bsx,t)-U(\bsy,t)]\leq K_{\beta^*}\; (\tilde{d}(\bsx,\bsy))^{2\beta^*},
		\]
		where
		\begin{align*}
			K_{\beta^*}:= 2^{4-\beta^*}
			\begin{cases} 
				K^{(1)}_{\beta^*}, & t\leq\tau, \\
				K^{(1)}_{\beta^*}+(t-\tau)K^{(2)}_{\beta^*}, & t>\tau,
			\end{cases}
		\end{align*}
		which completes the proof.
	\end{proof}
	
Under the assumptions of Theorem \ref{Hoelder}, applying Kolmogorov’s continuity criterion (see Corollary 4.5 in \cite{Lanetal16}), for any fixed $t>0$, we deduce the existence of a locally H\"{o}lder continuous version (a modification) for the random field $U$  of order $\gamma^{*}\in(0,\beta^*)$.

	\section{Numerical studies}\label{Num}
	In this section, we present some numerical examples for the solution $U$ of equation \eqref{System}. In particular, we show the evolution of the stochastic solution $U$ of equation \eqref{System} using simulated data inspired by the CMB map. Also, we explain the convergence rates of the truncation
	errors of the approximations $U_{L}(t)$, $t\in(0,\infty)$, and the $L_2$-norm of the temporal increments of the stochastic solution $U(t)$, $t\in(\tau,\infty)$, to the equation \eqref{System}.

	\subsection{Evolution of the solution}\label{Evol}
	In this subsection, we illustrate the evolution of the stochastic solution $U$ of the equation \eqref{System} using simulated data inspired by the CMB map.

	First, in Figure \ref{fig:U400 at 0} we display a realization of the truncated initial condition at time $t=0$ of degree $L=600$ (i.e. $\xi=U_{600}(0)$) whose angular power spectrum $\calC_{\ell}$ has the form given by \eqref{New-Cl} with $\widetilde{D}=\widetilde{C}=1$, $\kappa_1 =2.3$. The Python HEALPy package (see \cite{Gorski2005}) was used to generate such a realization.
	\begin{figure}[ht]
		\centering
		\includegraphics[width=0.55\textwidth]{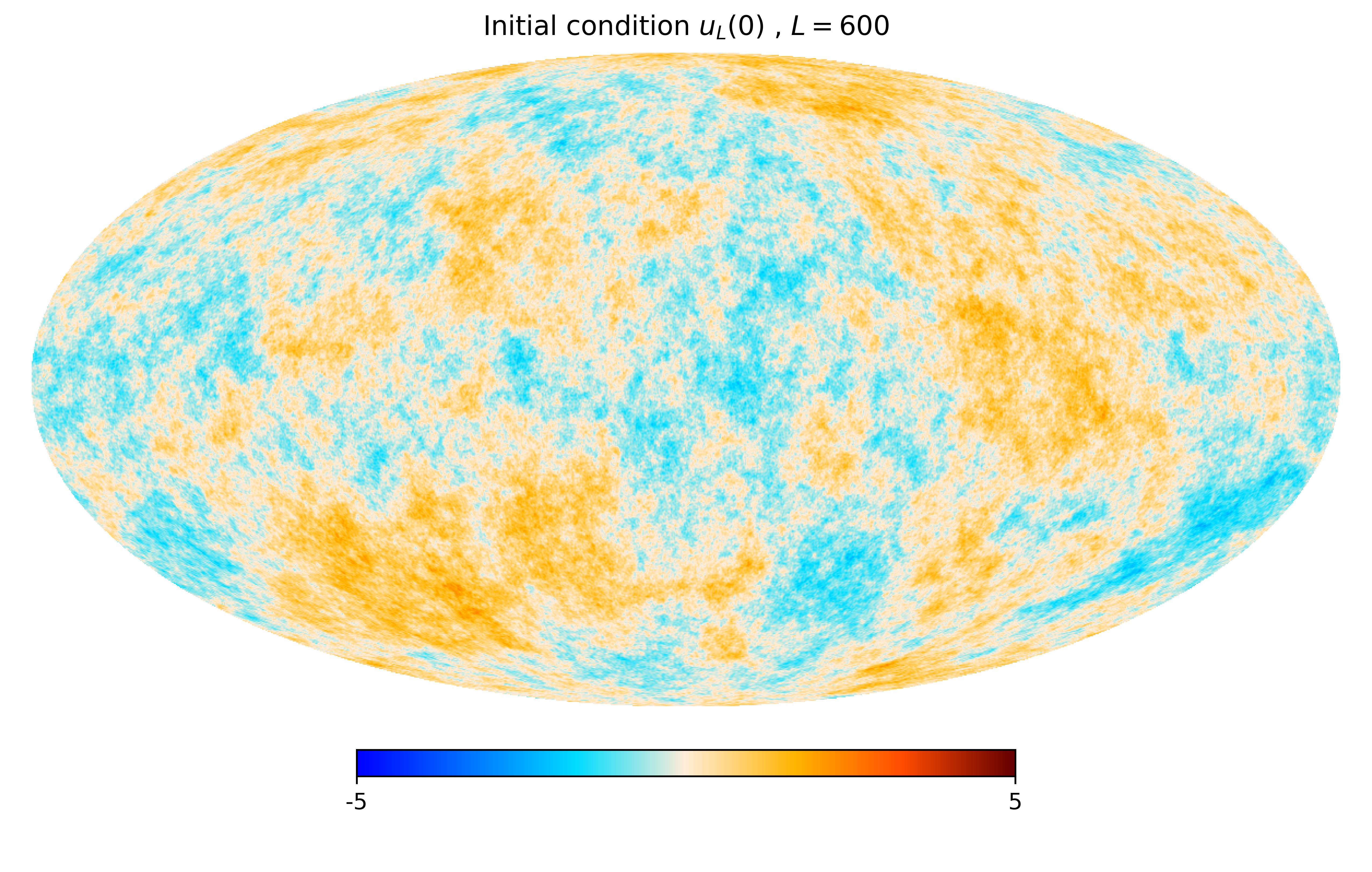}
		\caption{A realization of the truncated initial condition $U_{600}(0)$ with $L=600$, $\alpha=0.5$ and $\kappa_1=2.30$.}
		\label{fig:U400 at 0}
	\end{figure}
	Then, we use the obtained initial realization $U_{600}(0)$ to generate two realizations for truncated solution $U_{600}(t)$ given by \eqref{Approx} at different times, namely $t=\tau$ and $t=10\tau$ with $\tau=10^{-5}$ for the case $\alpha=0.5$, $\kappa_1=2.3$, and $\kappa_2=2.5$. 
	We use the equation \eqref{New-Al} with $\widetilde{K}=\widetilde{A}=10^4$ to compute the angular power spectrum $\calA_{\ell}$ for the time-delayed Brownian motion $W_\tau$. The Fourier coefficients $\widehat{U}_{\ell,m}(t)$ for the solution $U_L(t)$ are computed, using \eqref{gamlm} with \eqref{UH-complexFourier} and \eqref{complexFourier}, by
	\begin{align}\label{Vl0}
		V_{\ell,0}(t):=
		\begin{cases} 
			E_{\alpha}(-\lambda_{\ell}t^{\alpha}) \sqrt{\calC_{\ell}}Z_{\ell,0}^{(1)}, & t\leq\tau, \\
			E_{\alpha}(-\lambda_{\ell}t^{\alpha}) \sqrt{\calC_{\ell}}Z_{\ell,0}^{(1)}+\sqrt{\calA_{\ell}}\mathcal{I}_{\ell,0,\alpha}^{(1)}(t-\tau), & t>\tau,
		\end{cases}
	\end{align}
	and for $m=1,\dots,\ell$,
	\begin{align}\label{Vlm}
		V_{\ell,m}(t):=
		\begin{cases} 
			E_{\alpha}(-\lambda_{\ell}t^{\alpha})\sqrt{\calC_{\ell}/2}
			Z_{\ell,m}, & t\leq\tau, \\
			E_{\alpha}(-\lambda_{\ell}t^{\alpha})\sqrt{\calC_{\ell}/2}
			Z_{\ell,m}+
			\sqrt{\calA_{\ell}/2}
			\mathcal{I}_{\ell,m,\alpha}(t-\tau), & t>\tau,
		\end{cases}
	\end{align}
	where $Z_{\ell,m}:= Z_{\ell,m}^{(1)} - \mi Z_{\ell,m}^{(2)}$ and $\mathcal{I}_{\ell,m,\alpha}(t):= \mathcal{I}_{\ell,m,\alpha}^{(1)}(t)
	-\mi\mathcal{I}_{\ell,m,\alpha}^{(2)}(t)$, $t>0$.
	
	Note that for each realization and given time $t>\tau$, we simulate the stochastic integrals $\calI_{\ell,m,\alpha}^{(j)}(t-\tau)$, $j=1,2$, using Proposition \ref{PropVar}, as independent, normally distributed
	random variables with mean zero and variances $\sigma_{\ell,t-\tau,\alpha}^2$ where $\sigma_{\ell,t,\alpha}^2$ is given by \eqref{var}.
	Figures \ref{HOM1} and \ref{HOM2}, respectively, show two realizations of the truncated solutions $U_{L}(t)$ and $U_{L}^{H}(t)$ with $L=600$, using the initial realization obtained in Figure \ref{fig:U400 at 0}, at times $t=\tau$ and $t=10\tau$ with $\tau=10^{-5}$. Figure \ref{inHOM1} shows a realization of the inhomogeneous solution $U_{L}^I(t)$ at time $t=10\tau$ with $\tau=10^{-5}$ and $\alpha=0.5$, while Figure \ref{fullHOM2} shows a realization of the combined solution $U_{L}(t)$ 
	using the realizations obtained in Figures \ref{HOM2} and \ref{inHOM1}.
	\begin{figure}[ht]
		\centering
		\subfloat[\centering $U_{600}(\tau)$ with $\tau=10^{-5}$. \label{HOM1}]{{\includegraphics[width=0.45\textwidth]{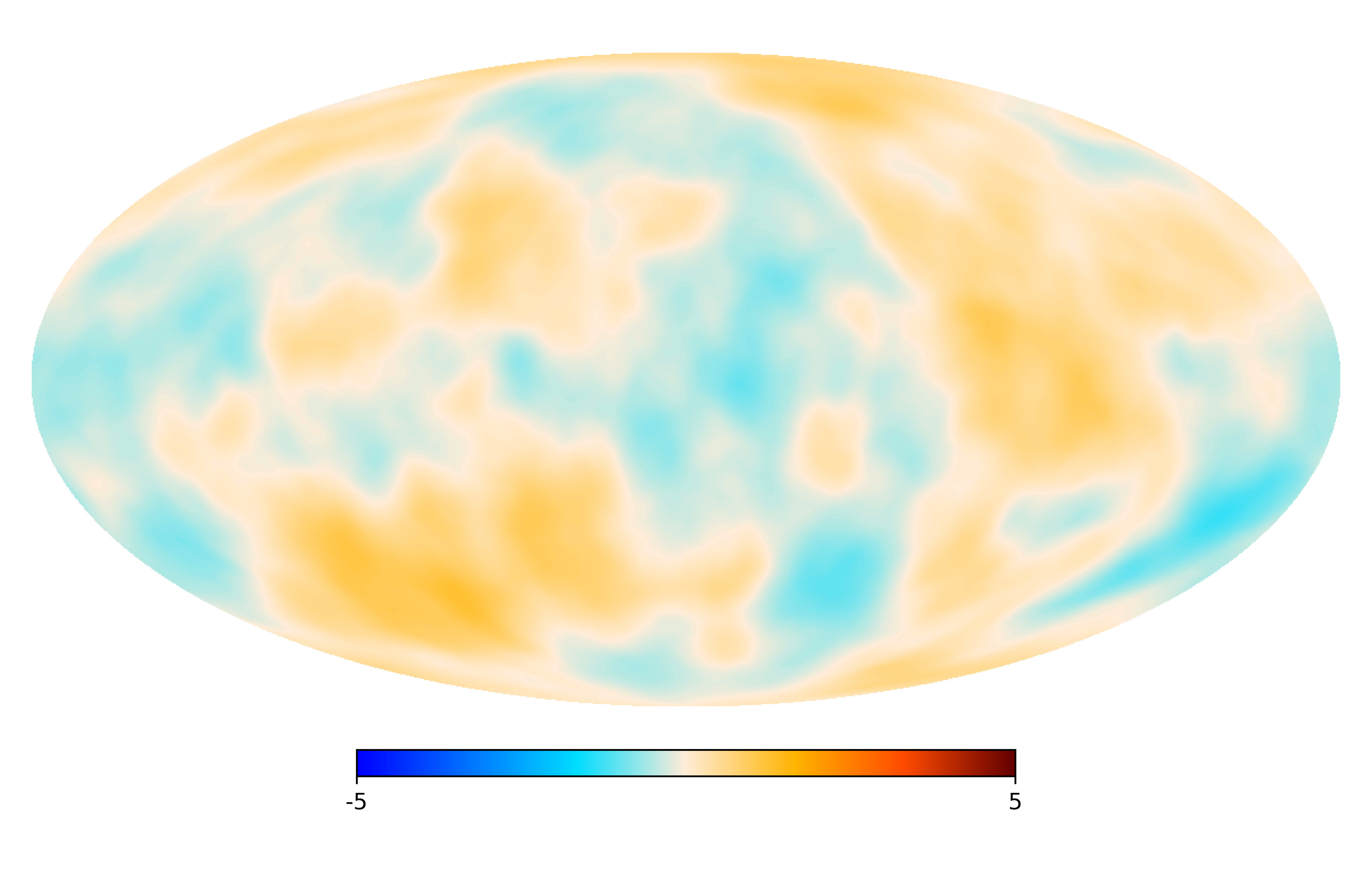} }}%
		\qquad
		\subfloat[\centering $U^H_{600}(10\tau)$ with $\tau=10^{-5}$.\label{HOM2}]{{\includegraphics[width=0.45\textwidth]{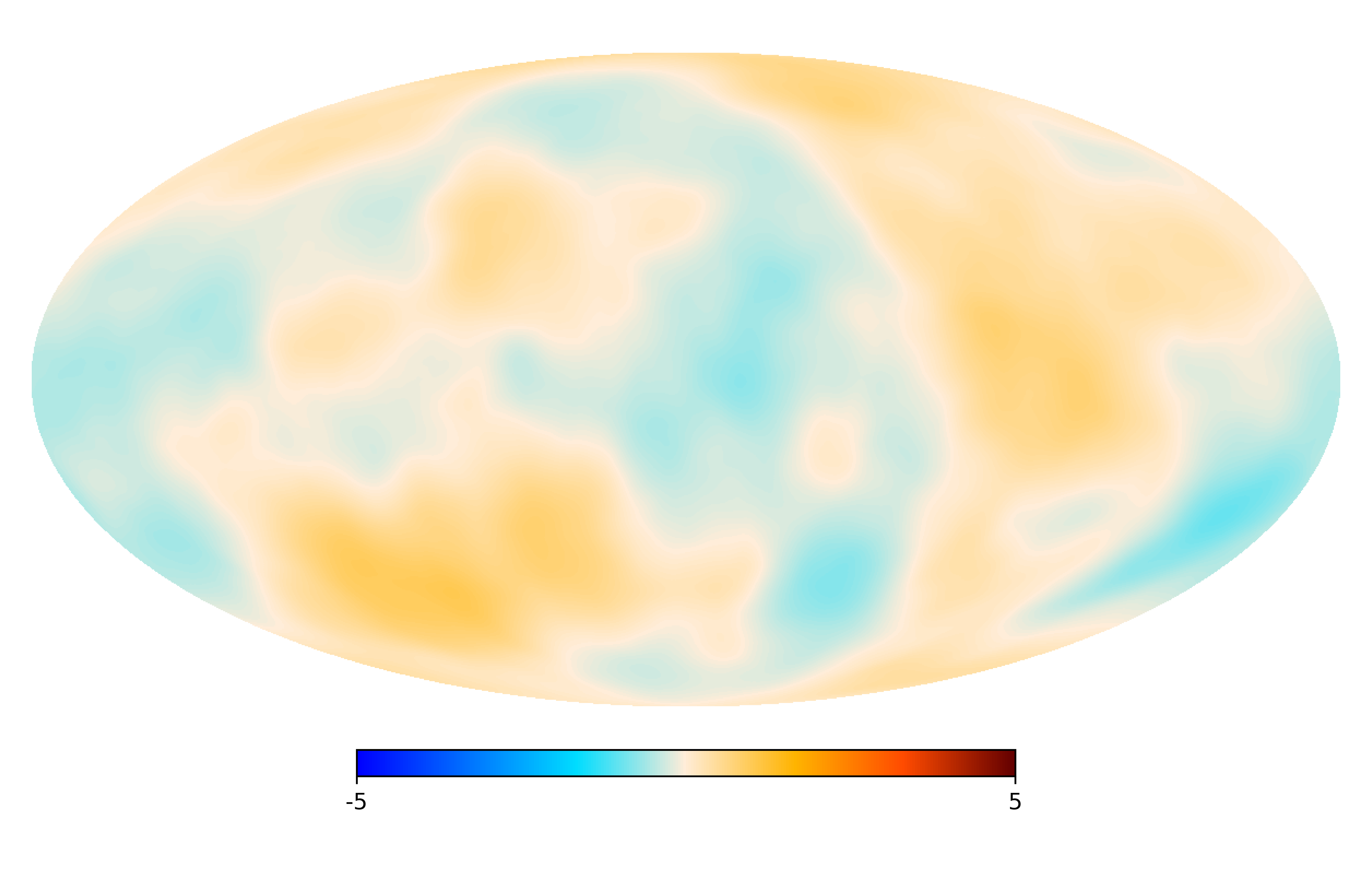} }}%
		\caption{ Truncated homogeneous solutions $U_{600}(\tau)=U^H_{600}(\tau)$ and $U^H_{600}(10\tau)$, using the initial realization in Figure \ref{fig:U400 at 0}, with $\alpha=0.5$, $\kappa_1=2.3$, and $\kappa_2=2.5$.}%
		\label{fig:Moho cases}%
	\end{figure}
	
	\begin{figure}[ht]
		\centering
		\subfloat[\centering $U_{600}^{I}(10\tau)$ with $\tau=10^{-5}$. \label{inHOM1}]{{\includegraphics[width=0.45\textwidth]{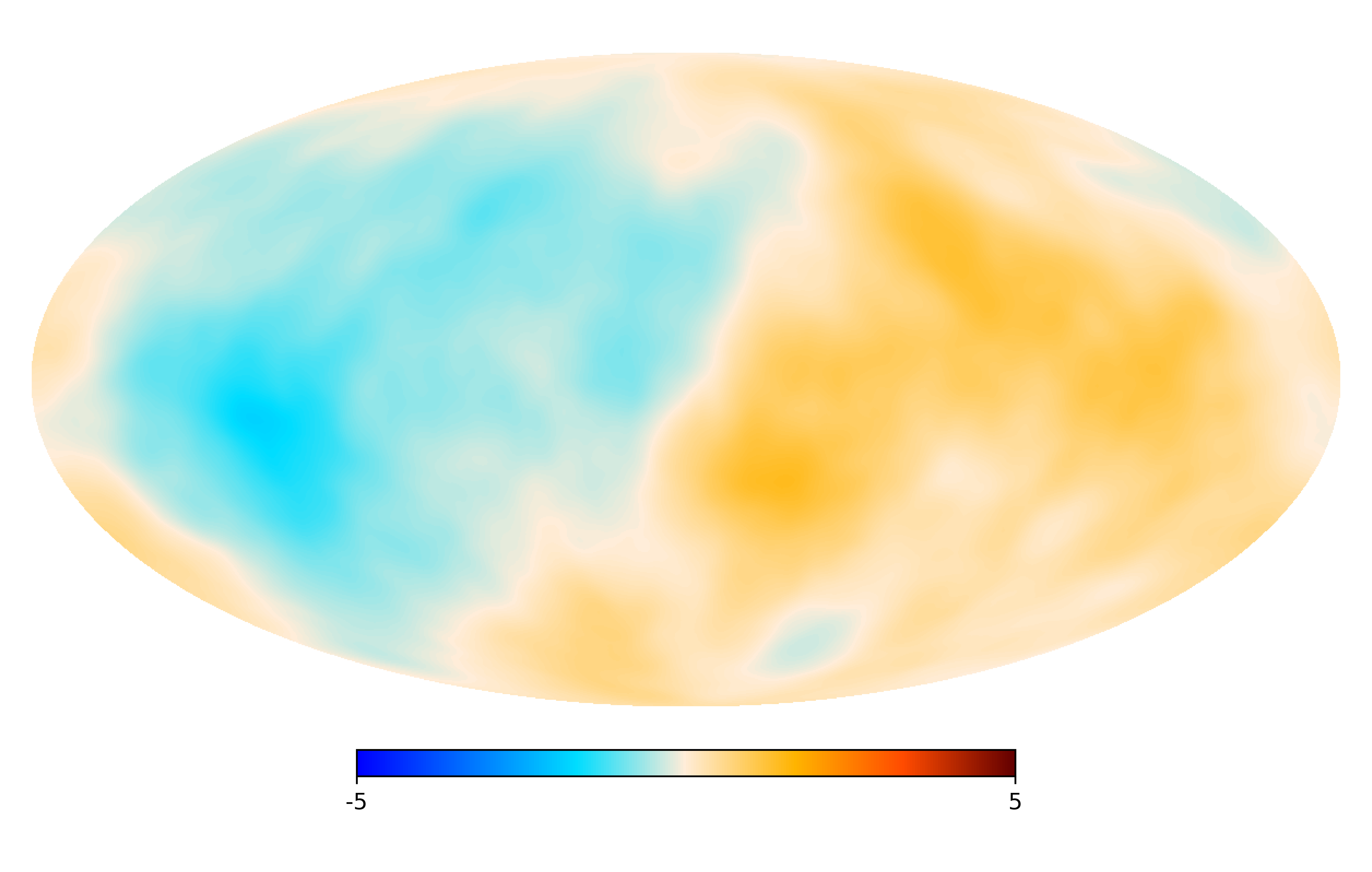} }}%
		\qquad
		\subfloat[\centering $U_{600}(10\tau)=U^H_{600}(10\tau)+U^I_{600}(10\tau)$ with $\tau=10^{-5}$.\label{fullHOM2}]{{\includegraphics[width=0.45\textwidth]{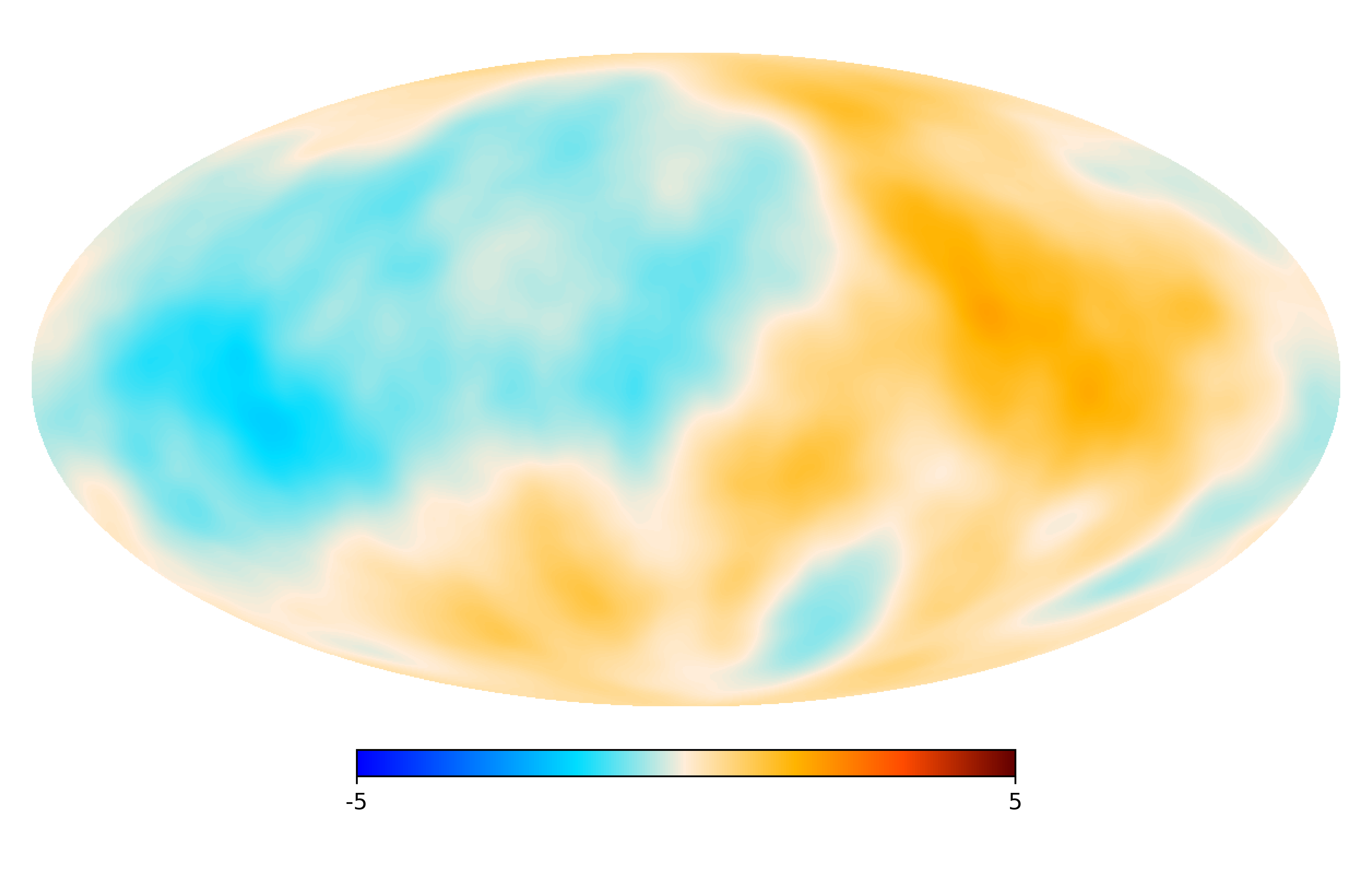} }}%
		\caption{ (a) The truncated inhomogeneous solution $U_{L}^I(10\tau)$ with $\tau=10^{-5}$, $\widetilde{A}=10^4$,  $\alpha=0.5$ and $\kappa_2=2.5$, (b) the truncated combined solution $U_{600}(10\tau)$, using the truncated homogeneous and inhomogeneous solutions in Figures \ref{HOM2} and \ref{inHOM1}.}%
		\label{fig:U400 at 10tau}%
	\end{figure}
	The Figures \ref{HOM1} and \ref{HOM2} reflect the time evolution of the truncated homogeneous stochastic solution $U_{600}^H(t)$ at time $t=\tau$ and  $t=10\tau$ respectively. In this case, the randomness comes only from the initial condition. Figure \ref{inHOM1} reflects the time evolution of the truncated inhomogeneous stochastic solution $U_{600}^I(t)$ at time $t=10\tau$ (in this case the randomness comes only from $W_\tau$). 
	Figure \ref{fullHOM2} reflects the evolution of the same realization obtained in Figure \ref{HOM2} at $t=10\tau$ but for which the noise $W_\tau$ is switched on (in this case the randomness comes from both the initial condition and $W_\tau$, i.e. $U_{600}(10\tau)=U_{600}^H(10\tau)+U_{600}^I(10\tau)$).
	
	\subsection{Simulations of truncation errors of solution}\label{Sim-stage}
	This subsection presents some numerical examples for the truncation errors of the combined stochastic solution $U$ of the equation \eqref{System}. In particular, the numerical examples explain the convergence rates of the truncation errors $Q_{L}(t)$, given by \eqref{fullq}, of the stochastic solution of \eqref{System}. 
	
	To produce numerical results, we use the angular power spectra $\calA_\ell$ and $\calC_\ell$ given by \eqref{New-Cl} and \eqref{New-Al} respectively with $\widetilde{A}=\tilde{K}=10^4$, $\widetilde{C}=\widetilde{D}=1$, $\kappa_1=2.3$, and $\kappa_2=2.5$. Also, we use $U_{\widetilde{L}}$ with $\widetilde{L}=1500$ as a substitution of the solution $U$ to the equation \eqref{System} which is given by \eqref{Exact}. 
	Then the (squared) mean $L_2$-errors are computed, using Parseval's identity,
	by
	\begin{align}\label{Err-system}
		(Q_{L,\widetilde{L}}(t))^2:&=	\Big\Vert U_{\widetilde{L}}(t)-U_L(t)	\Big\Vert_{L_{2}(\Omega\times\bS^2)}^2\notag\\
		&=\bE\Big[\Vert U_{\widetilde{L}}(t)-U_L(t)	\Vert_{L_{2}(\bS^2)}^2\Big]\notag\\ 
		&= \bE\Big[\Big\Vert \sum_{\ell=L+1}^{\widetilde{L}}\sum_{m=-\ell}^{\ell}\widehat{U}_{\ell,m}(\omega,t)Y_{\ell,m}\Big\Vert_{L_2(\bS^2)}^2\Big]\notag\\
		&\approx\dfrac{1}{\widehat{N}}\sum_{j=1}^{\widehat{N}}\sum_{\ell=L+1}^{\widetilde{L}}\sum_{m=-\ell}^{\ell}\Big\vert \widehat{U}_{\ell,m}(\omega_j,t)\Big\vert^2\notag\\
		&=\dfrac{1}{\widehat{N}}\sum_{j=1}^{\widehat{N}}\sum_{\ell=L+1}^{\widetilde{L}}\sum_{m=0}^{\ell}\Big\vert V_{\ell,m}(\omega_j,t)\Big\vert^2,
	\end{align}
	where the fourth line approximates the expectation by the mean of $\widehat{N}=100$ realizations, and 
	$\{V_{\ell,m}:\ell=L+1,\dots,\widetilde{L}; m=0,\dots,\ell\}$ are computed, using \eqref{Vl0} and \eqref{Vlm}.

	\indent
	To illustrate the results of Theorem \ref{The5}, we consider two cases, namely $\alpha=0.5$ and $\alpha=0.75$. Then we compute the root mean $L_2$-errors using equation \eqref{Err-system} with degree up to $L=800$ for different values of $t$. 
	Figure \ref{fig:stage model} shows numerical errors $Q_{L,\widetilde{L}}(t)$ of $\widehat{N}=100$ realizations and the corresponding theoretical errors for different values of $t$ and $\tau$ with $\kappa_1=2.3$ and $\kappa_2=2.5$ for the case $\alpha=0.5$.
	In particular, in Figure \ref{T1} we illustrate the case \ref{itm:AA1} ($0< t\leq\lambda_{L}^{-2}$) in Theorem \ref{The5} with $t=10^{-12}$. 
	In Figure \ref{T2} we illustrate the case \ref{itm:AA0} with $t=\tau+\lambda_{L}^{-2}$ for $\tau=10^{-5}$, while Figure \ref{T3} illustrates the case \ref{itm:AA2} with $t=10\tau$ when $\tau=10^{-5}$. 
	Similarly, we illustrate the results of Theorem \ref{The5} using $\alpha=0.75$. The corresponding results are displayed in Figures \ref{TQ1}-\ref{TQ3}.
	The blue points in each picture in Figures \ref{T1}-\ref{T3} and  \ref{TQ1}-\ref{TQ3} show the
	(sample) mean square approximation errors of $\widehat{N}=100$ realizations of $Q_{L,\widetilde{L}}(t)$. The red line in each figure shows the corresponding theoretical approximation upper bound.
	The numerical results show that the convergence rates of the mean $L_2$-error of $U_{L}$ are consistent with the corresponding theoretical results in Theorem \ref{The5}.
	\begin{figure}[ht]
		\centering
		\subfloat[case \ref{itm:AA1} with $t=10^{-12}$ \label{T1}]{\includegraphics[width=0.5\textwidth]{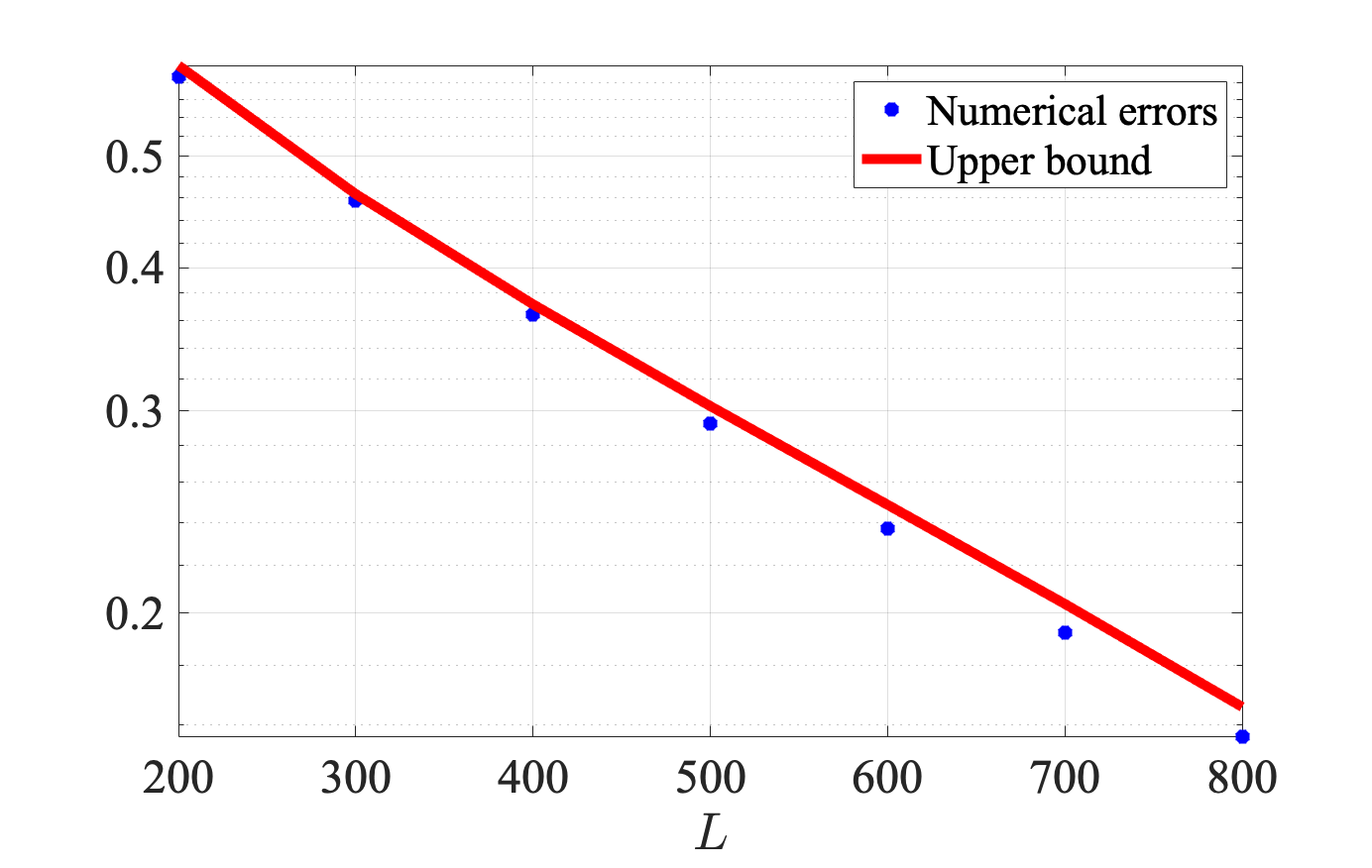}}
		\subfloat[case \ref{itm:AA0} with $t=\tau+\lambda_{800}^{-2}$, $\tau=10^{-5}$\label{T2}]{\includegraphics[width=0.5\textwidth]{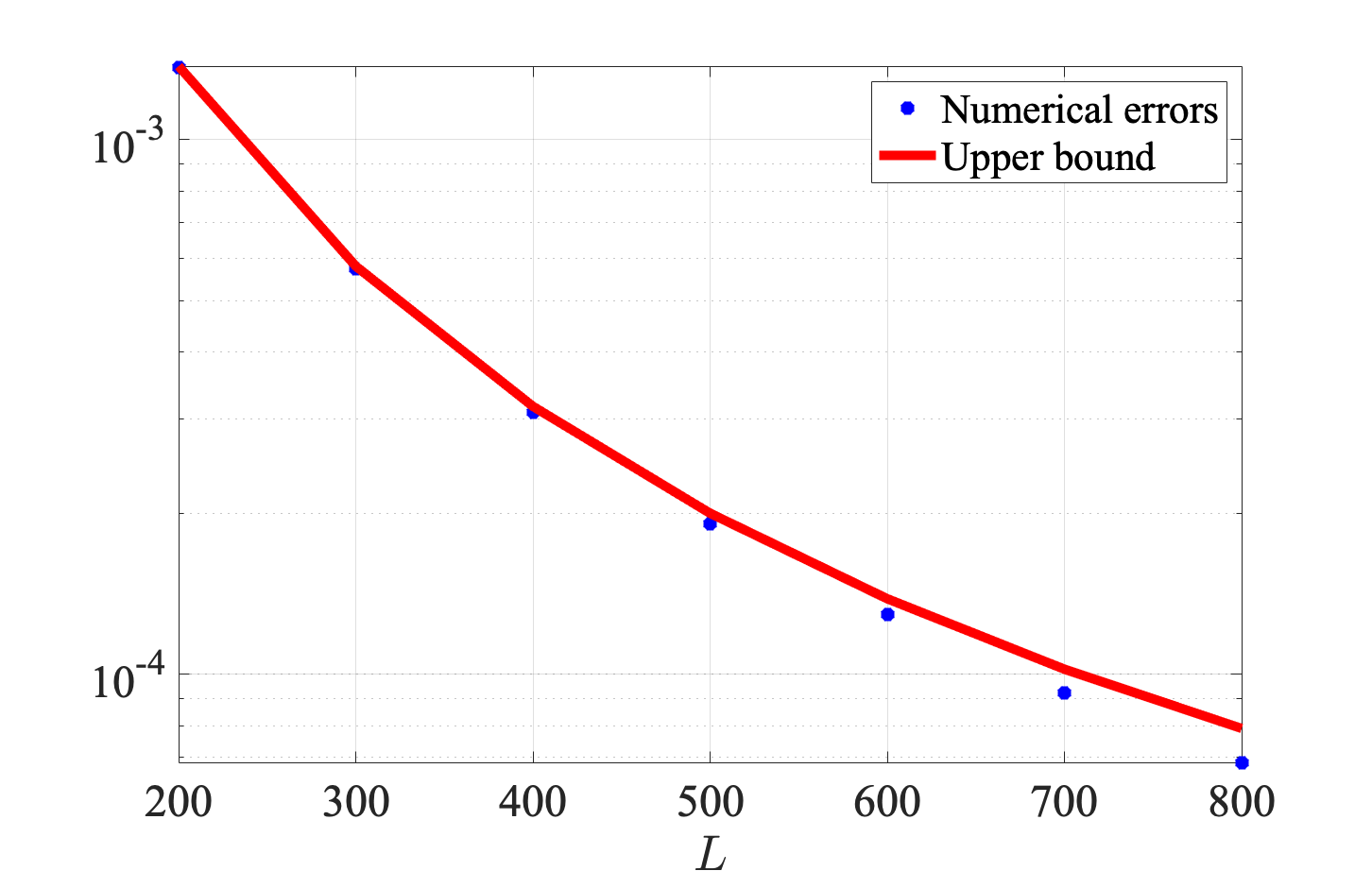}}\qquad
		\subfloat[case \ref{itm:AA2} with $t=10\tau$, $\tau=10^{-5}$\label{T3}]{\includegraphics[width=0.5\textwidth]{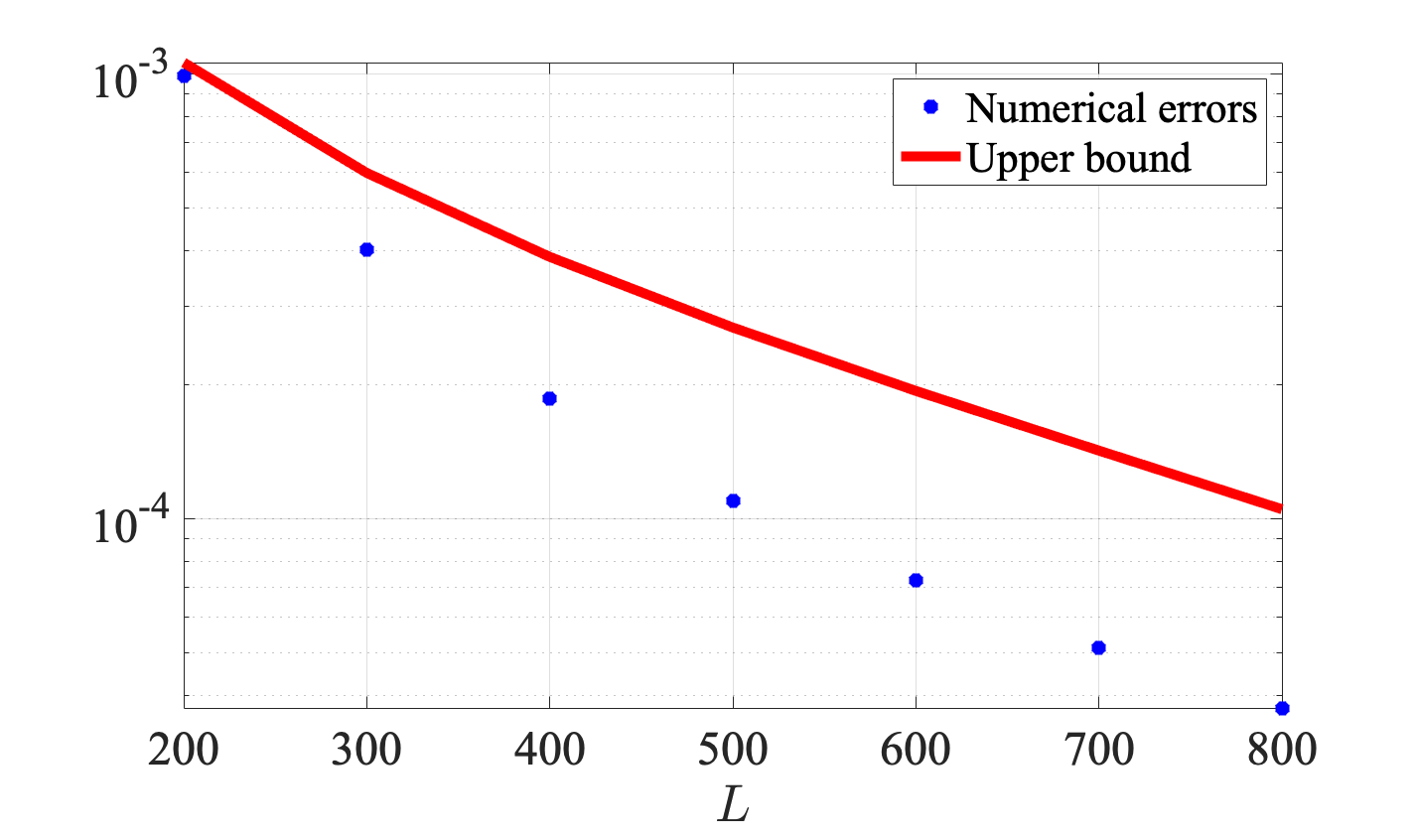}}
		\caption{Numerical errors for the combined solution $U(t)$ with
			$\kappa_1=2.3$, $\kappa_2=2.5$, $\alpha=0.5$.}
		\label{fig:stage model}
	\end{figure}
	\begin{figure}[ht]
		\centering
		\subfloat[case \ref{itm:AA1} with $t=10^{-12}$ \label{TQ1}]{\includegraphics[width=0.5\textwidth]{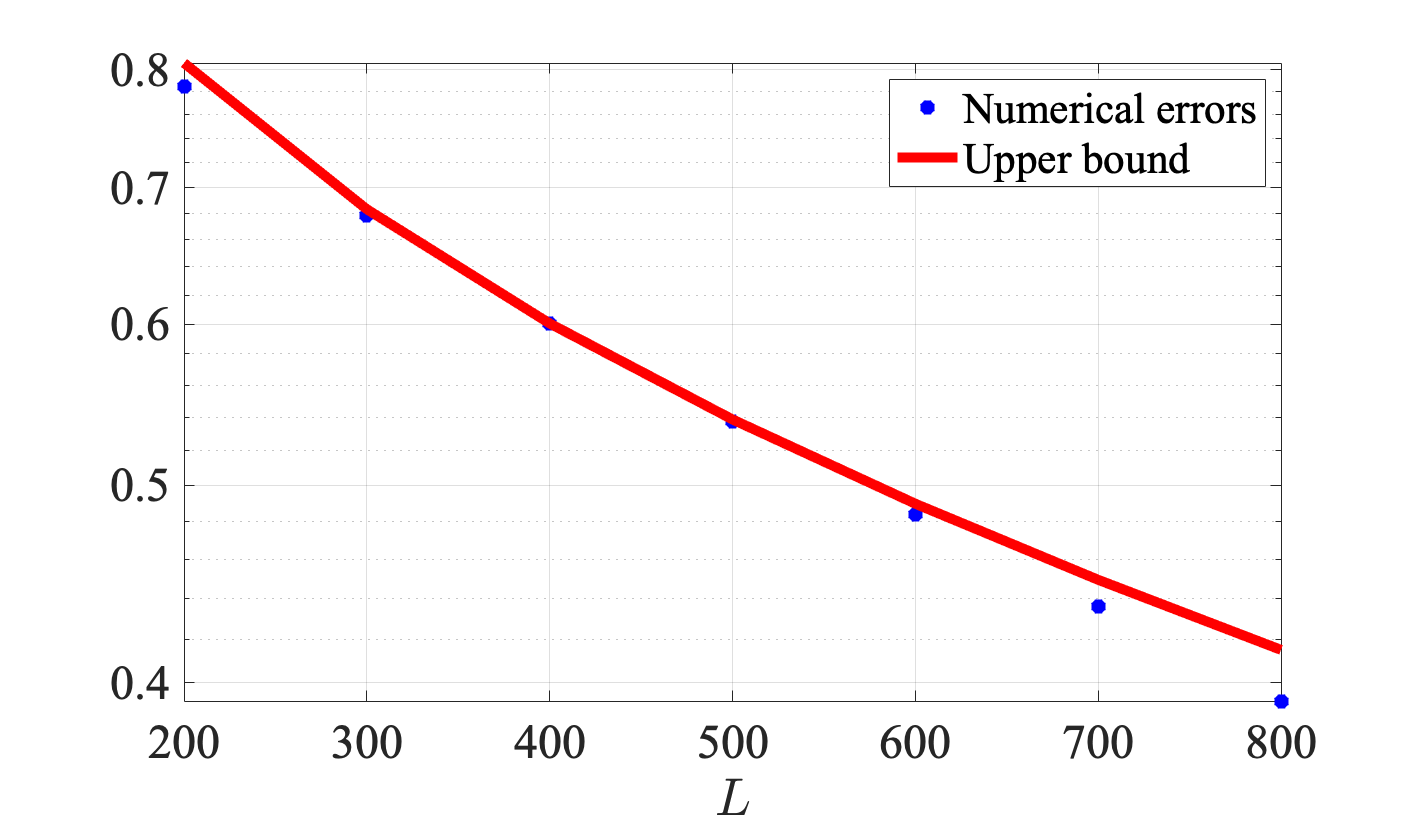}}
		\subfloat[case \ref{itm:AA0} with $t=\tau+\lambda_{800}^{-2}$, $\tau=10^{-5}$\label{TQ2}]{\includegraphics[width=0.5\textwidth]{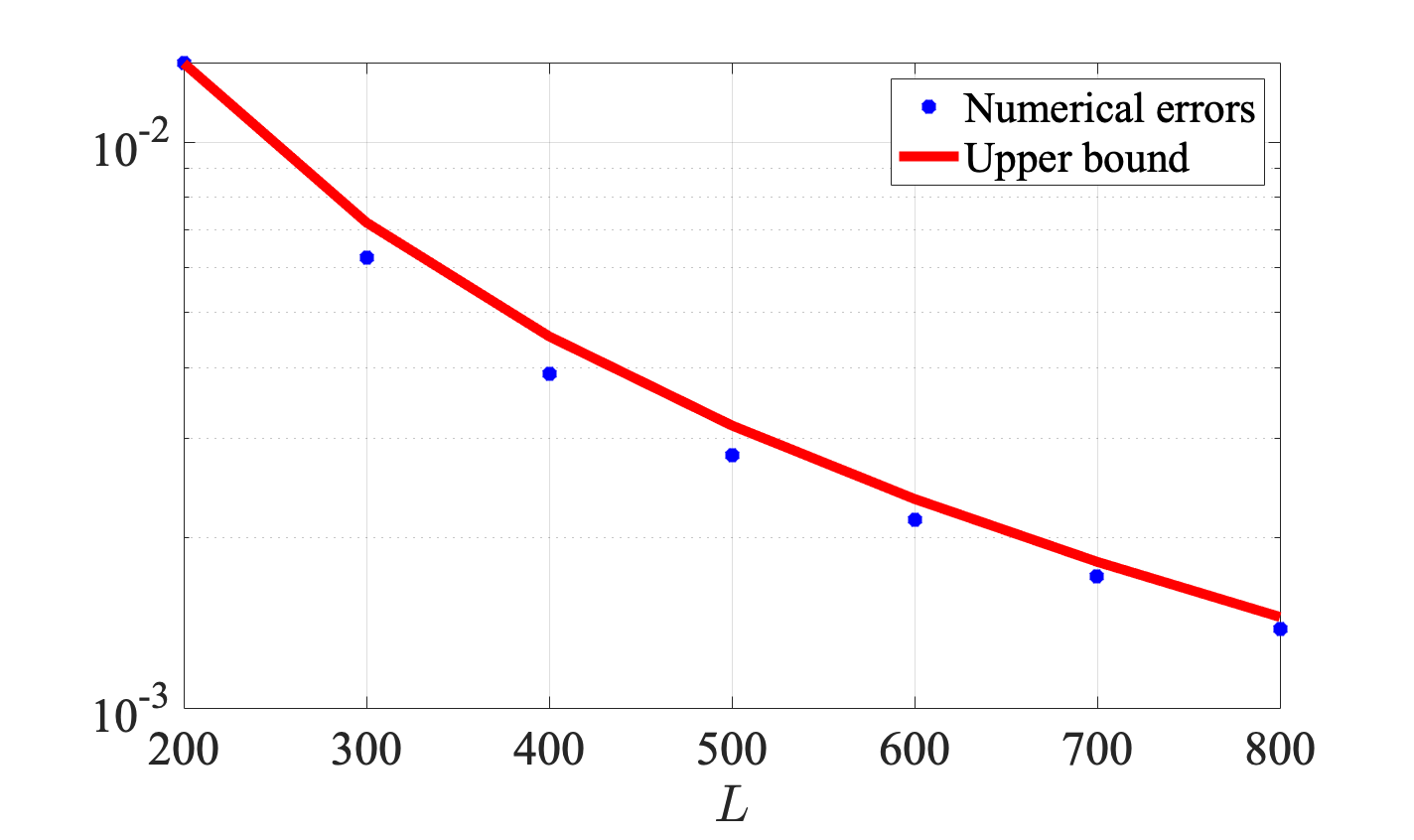}}\qquad
		\subfloat[case \ref{itm:AA2} with $t=10\tau$, $\tau=10^{-5}$\label{TQ3}]{\includegraphics[width=0.5\textwidth]{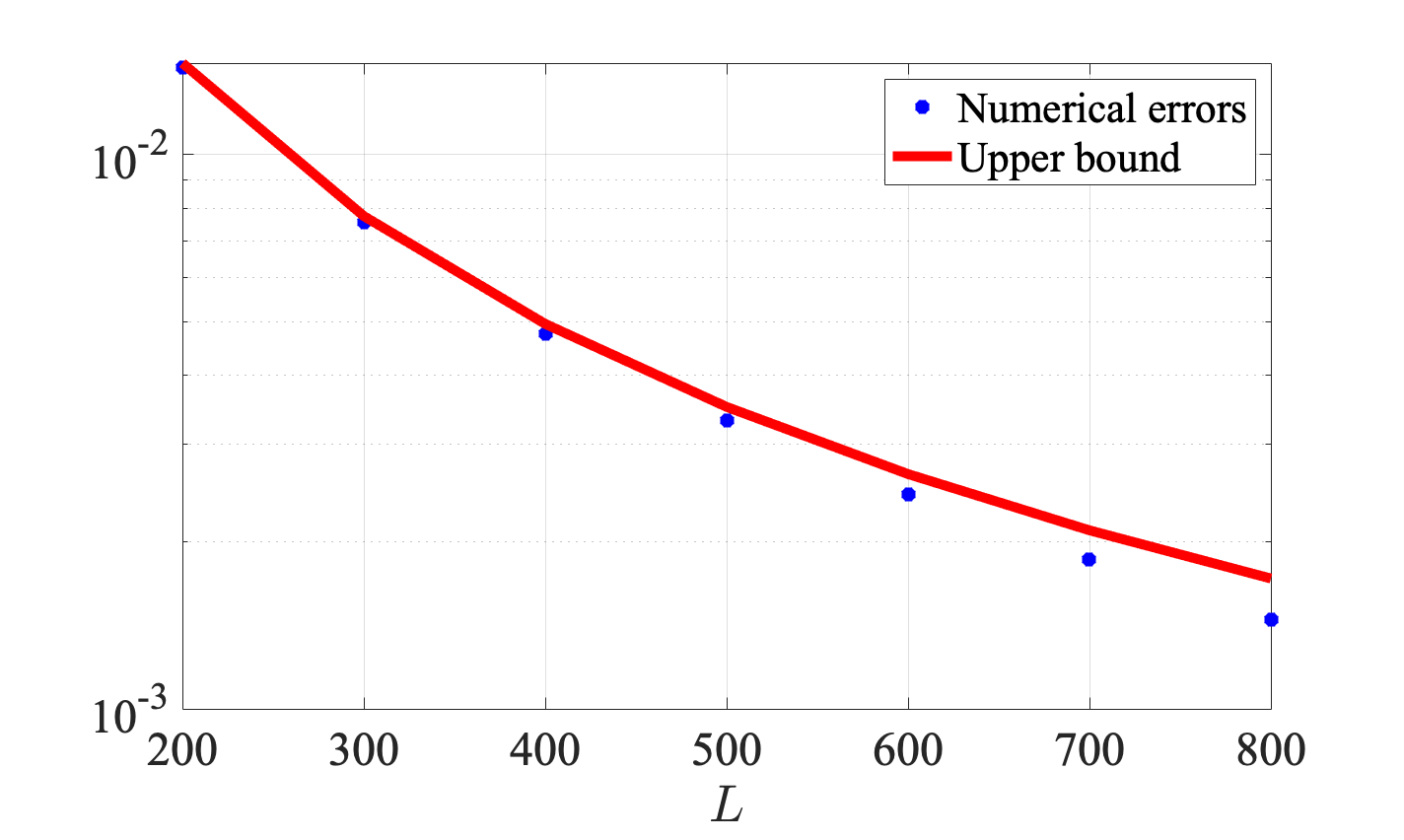}}
		\caption{Numerical errors for the combined solution $U(t)$ with
			$\kappa_1=2.3$, $\kappa_2=2.5$, $\alpha=0.75$.}
		\label{fig:twostage model}
	\end{figure}
	
	\subsection{Simulations of temporal increments}\label{Sim-inc}
	This subsection presents a numerical example for the stochastic solution $U$ of the equation \eqref{System} in time. In particular, the numerical example explains the $L_2-$norm temporal increments $\calJ_h(t)$, given by \eqref{Jinc}, of equation \eqref{System}.

	To produce numerical results, we use the angular power spectra given by \eqref{New-Cl} and \eqref{New-Al}. Then, we use the truncated version of the  $L_2-$norm temporal increments of the solution $U(t)$ for $t\in(\tau,\infty)$ which can be estimated as
	\begin{align}\label{timeErr} 	
		(\calJ_{h, L}(t))^2:&=\Big\Vert U_{L}(t+h)-U_{L}(t)	\Big\Vert_{L_{2}(\Omega\times\bS^2)}^2\notag\\
		&=\bE\Big[\Vert U_{L}(t+h)-U_{L}(t)	\Vert_{L_{2}(\bS^2)}^2\Big]\notag\\ 
		&= \bE\Big[\Big\Vert \sum_{\ell=0}^{\widetilde{L}}\sum_{m=-\ell}^{\ell}(\widehat{U}_{\ell,m}(\omega,t+h)-\widehat{U}_{\ell,m}(\omega,t))Y_{\ell,m}\Big\Vert_{L_2(\bS^2)}^2\Big]\notag\\ 	 	&\approx\dfrac{1}{\widehat{N}}\sum_{j=1}^{\widehat{N}}\sum_{\ell=0}^{\widetilde{L}}\sum_{m=0}^{\ell}\big\vert (V_{\ell,m}(\omega_j,t+h)-V_{\ell,m}(\omega_j,t))\big\vert^2,
	\end{align}
	where $\{V_{\ell,m}:\ell=0,\dots,\widetilde{L}; m=0,\dots,\ell\}$ are computed using \eqref{Vl0} and \eqref{Vlm}.
	
	To illustrate the results of Theorem \ref{Theo6}, we use $U_{\widetilde{L}}$ with $\widetilde{L}=1500$ as a substitution of the solution $U$ to the equation \eqref{System} which is given by \eqref{Exact}. Then, the  $L_2-$norm temporal increments $\calJ_{h, L}(t)$ are computed using the estimate \eqref{timeErr} with $\alpha=0.5$, $t=\tau+\delta$, with $\tau=10^{-5}$ and $\delta=10^{-6}$.
	Figure \ref{fig:inc time} shows numerical temporal increments $\calJ_{h, L}(t)$ of $\widehat{N}=100$ realizations and the corresponding theoretical upper bound for the temporal increments with $t=\tau+\delta$, 
	$\alpha=0.5$, $\kappa_1=2.3$, and $\kappa_2=2.5$, and time increment $h$ ranging from $\delta$ to $11\delta$.  
	\begin{figure}[ht]
		\centering
		\includegraphics[width=0.65\textwidth]{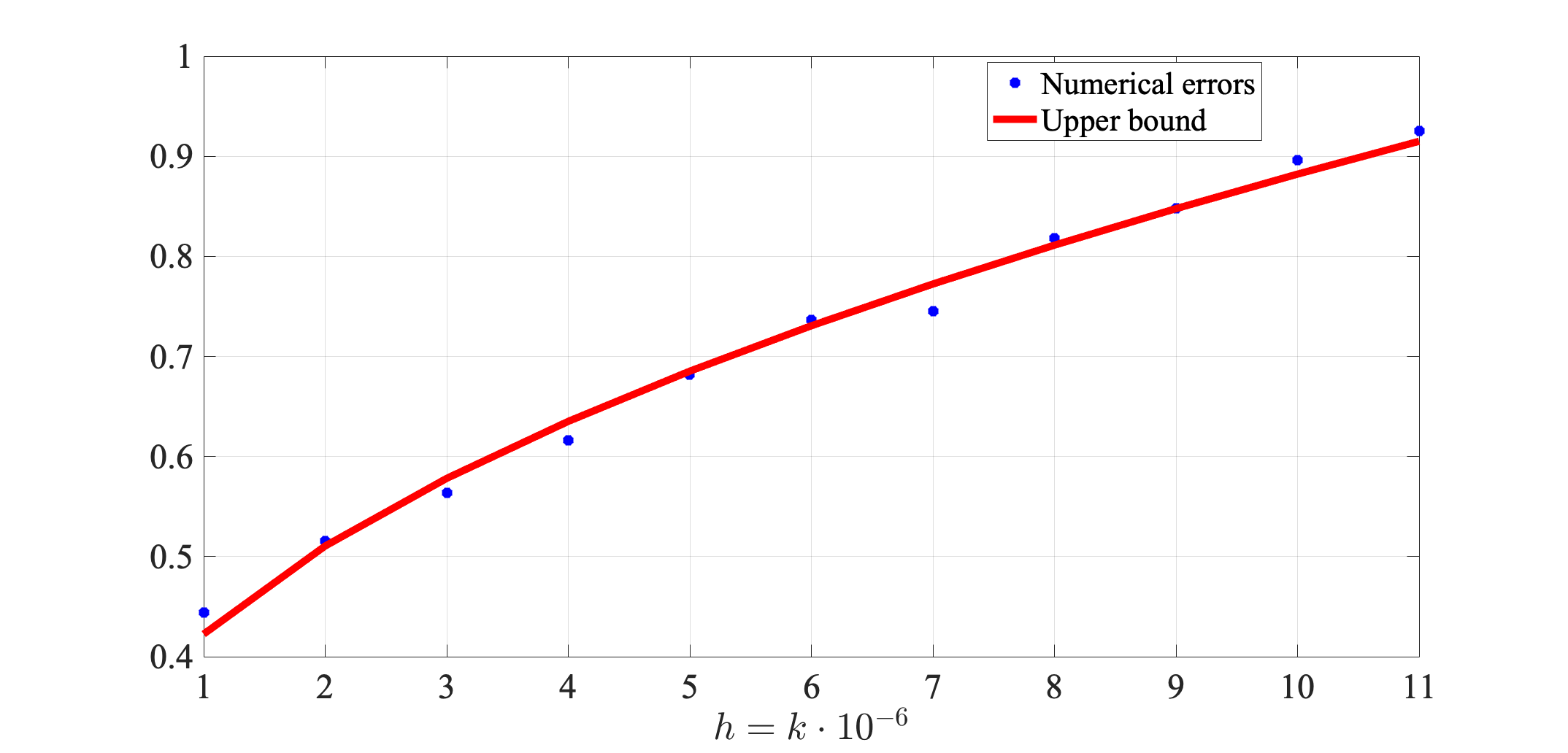}
		\caption{Numerical values of $\mathcal{J}_{h,L}(t) = \|U_L(t+h)-U_L(t)\|_{L_2(\Omega\times\mathbb{S}^2)}$ with $L=1500$ and $t=\tau+\delta$,\;$\tau=10^{-5}$,\; $\delta=10^{-6}$ and $h=k\delta$, $k=1,\ldots,11$.}
		\label{fig:inc time}
	\end{figure}
	The blue points in Figure
	\ref{fig:inc time} show the
	(sample) mean square approximation errors of $\hat{N}=100$ realizations of $\mathcal{J}_h$. The red straight line shows the theoretical approximation upper bound with time increment $h$. 
	The numerical results show that the convergence rates of the $L_2-$norm temporal increments of the solution $U$ of the equation \eqref{System} are consistent with the theoretical results in Theorem \ref{Theo6}.

	\section*{Acknowledgements} 
	This research was supported under the Australian Research Council's Discovery Project funding scheme (Discovery Project number DP180100506). This research includes extensive
	computations using the Linux computational cluster Katana~\cite{katana} supported by the Faculty of
	Science, The University of New South Wales, Sydney. The authors are also grateful to Bill McLean for discussions on some properties of the Mittag-Lefler function and to an anonymous referee's comments to improve the presentation.

	\section*{Declarations}
	\begin{itemize}
		\item All authors declare that they have no conflicts of interest.
		\item Availability of data and code: The Python code used for numerical simulations in the paper can be
		found at \url{https://github.com/qlegia/Stochastic-diffusion-on-sphere}.
	\end{itemize}

	
	
	

	
	

	

	
	
	
\end{document}